\documentclass{amsart}

\usepackage{amsmath}
\usepackage{amssymb}
\usepackage{graphicx}
\usepackage{amsmath,amscd}
\usepackage{tikz}
\usepackage{tikz-cd}
\usetikzlibrary{positioning,arrows}
\usepackage[utf8]{inputenc} 
\usepackage[T1]{fontenc}
\usepackage{enumitem}
\usepackage[small,nohug,heads=vee]{diagrams}
\usepackage{hyperref}
\hypersetup{breaklinks=true}

\numberwithin{equation}{section}
\numberwithin{figure}{equation}
\theoremstyle{plain}

\newtheorem{lem}[equation]{Lemma}

\newtheorem{prop}[equation]{Proposition}
\newtheorem{thm}[equation]{Theorem}
\newtheorem{cor}[equation]{Corollary}
\newtheorem{que}[equation]{Question}
\newtheorem*{que*}{Question}
\newtheorem{ob}[equation]{Observation}

\theoremstyle{definition}
\newtheorem{definition}[equation]{Definition}
\newtheorem{remark}[equation]{Remark}
\newtheorem{example}[equation]{Example}
\newtheorem{claim}[equation]{Claim}

\newcommand*{\wpj}{\ensuremath{\mathsf{WProj}}}
\newcommand{\C}{{\mathsf C}}

\newcommand{\A}{{\mathcal A}}
\newcommand{\B}{{\mathcal B}}

\newcommand{\E}{\mathbb E}

\newcommand{\G}{{\mathcal G}}

\newcommand{\h}{{\mathcal H}}

\newcommand{\K}{{\mathcal K}}
\renewcommand{\L}{{\mathcal L}}

\newcommand{\n}{\mathcal{N}}

\renewcommand{\P}{{\mathcal P}}

\newcommand{\R}{\mathbb R}

\newcommand{\Z}{\mathbb Z}

\newcommand{\acts}{\curvearrowright}

\newcommand{\Ga}{\Gamma}
\newcommand{\isom}{\operatorname{Isom}}

\newcommand{\lra}{\longrightarrow}
\newcommand{\out}{\operatorname{Out}}
\newcommand{\aut}{\operatorname{Aut}}
\newcommand{\ra}{\rightarrow}
\newcommand{\stab}{\operatorname{Stab}}

\begin{document}

\title{Commensurability of groups quasi-isometric to RAAG's}
\author{Jingyin Huang}
\address{The Department of Mathematics and Statistics\\
McGill University\\
Burnside Hall, 805 Sherbrooke W.\\
Montreal, QC, H3A 0B9, Canada}

\email{jingyin.huang@mcgill.ca}
\date{\today}

\maketitle

\begin{abstract}
Let $G$ be a right-angled Artin group with defining graph $\Ga$ and let $H$ be a finitely generated group quasi-isometric to $G(\Ga)$. We show if $G$ satisfies (1) its outer automorphism group is finite; (2) $\Ga$ does not contain any induced 4-cycle; (3) $\Ga$ is star-rigid; then $H$ is commensurable to $G$. We show condition (2) is sharp in the sense that if $\Ga$ contains an induced 4-cycle, then there exists an $H$ quasi-isometric to $G$ but not commensurable to $G$. Moreover, one can drop condition (1) if $H$ is a uniform lattice acting on the universal cover of the Salvetti complex of $G$. As a consequence, we obtain a conjugation theorem for such uniform lattices. The ingredients of the proof include a blow-up building construction in \cite{cubulation} and a Haglund-Wise style combination theorem for certain class of special cube complexes. However, in most of our cases, relative hyperbolicity is absent, so we need new ingredients for the combination theorem.
\smallskip 

\noindent
\textbf{2010 AMS classification numbers}.  20F65, 20F67, 20F69 

\smallskip

\end{abstract}

\tableofcontents

\section{Introduction}

\subsection{Background and overview}
It is well-known that if a finitely generated group is quasi-isometric to a free group, then it is commensurable to a free group. In this paper, we seek higher dimensional version of this fact in the class of right-angled Artin groups. Recall that given a finite simplicial graph $\Ga$, the \textit{right-angled Artin group} (RAAG) $G(\Ga)$ with defining graph $\Ga$ is given by the following presentation:
\begin{center}
\{$v_i\in\textmd{Vertex}(\Ga)\ |\ [v_i,v_j]=1$ if $v_{i}$ and $v_{j}$ are joined by an edge\}.
\end{center}
This class of groups has deep connections with a lot of other objects. Wise has proposed a program for embedding certain groups virtually into RAAG's (\cite[Section~6]{wise2009research}). One highlight in this direction is the recent solution of virtual Haken conjecture, which relies on the embedding of fundamental groups of 3-dimensional hyperbolic manifolds into RAAG's \cite{agol2013virtual,wisestructure}. Other examples include Coxeter groups \cite{haglund2010coxeter}, certain random groups \cite{ollivier2011cubulating} and small cancellation groups \cite{wise2004cubulating}, certain classes of 3-manifold groups \cite{przytycki2012mixed,przytycki2013graph,liu2013virtual,hagen2013cocompactly}, hyperbolic free-by-cyclic groups \cite{hagen2013cubulating,hagen2014cubulating}, one relator groups with torsion and limit groups \cite{wisestructure} etc. 

A theory of special cube complexes has been developed along the way \cite{wise2002residual,haglund2008finite,MR2377497,haglund2012combination}, which can be viewed as a higher-dimensional analogue of earlier work of Stallings \cite{stallings1983topology} and Scott \cite{scott1978subgroups}. In particular, it offers a geometric way to study finite index subgroups of virtually special groups.

This paper is motivated by the following question, which falls into Gromov's program of classifying groups and metric spaces up to quasi-isometry. 
\begin{que}
\label{motivating question}
Let $H$ be a finitely generated group quasi-isometric to $G(\Gamma)$. When is $H$ commensurable to $G(\Ga)$? Which $G(\Ga)$ is rigid in the sense that any such $H$ quasi-isometric to $G(\Gamma)$ is commensurable to $G(\Ga)$?
\end{que}

This question naturally leads to studying finite index subgroups of $G(\Ga)$. The aforementioned theory of special cube complexes turns out to be one of our main ingredients. We will give a class of quasi-isometrically rigid RAAG's, based upon a series of previous works \cite{bestvina2008quasiflats,bks2,huang_quasiflat,huang2014quasi,cubulation}. 

Before going into detailed discussion about RAAG's, we would like to compare RAAG's with other objects studied under Gromov's program. Previously quasi-isometry classification or rigidity results of spaces and groups can be very roughly divided into the case where various notions of non-positive curvature conditions are satisfied (like relative hyperbolicity, $CAT(0)$, coarse median etc), and the case where the non-positive curvature condition is absent (like certain solvable groups and graph of groups etc). The non-positively curved world can be further divided into the hyperbolic case (including relative hyperbolic with respect to flats), and the higher rank case, where there are lot of higher dimensional flats in the space, and they have non-trivial intersection pattern. Most RAAG's belong to the last case. Other higher rank cases include (1) higher rank lattices in semisimple Lie groups \cite{kleiner1997rigidity,eskin1997quasi,eskin1998quasi}; (2) mapping class groups and related spaces \cite{eskin2015rigidity,bowditch2015large2,bowditch2015large1,bowditch2015large,behrstock2012geometry,hamenstaedt2005geometry}; (3) certain 3-manifold groups and their generalizations \cite{kapovich1997quasi,behrstock2008quasi,MR2727658,frigerio2011rigidity}. The rigidity of these spaces often depends on the complexity of the intersection pattern of flats inside the space. The classes of RAAG's studied in this paper will be more \textquotedblleft rigid\textquotedblright\ than (3), but more \textquotedblleft flexible\textquotedblright\ than (1) and (2), see \cite{bks2,huang2014quasi,cubulation} for a detailed discussion.

In general, quasi-isometries of $G(\Ga)$ may not be of bounded distance from isometries, even if they are equivariant with respect to a cobounded group action. However, it is shown in \cite{bks2,huang2014quasi} that when the outer automorphism group $\out(G(\Ga))$ is finite, every quasi-isometry is of bounded distance from a canonical representative which preserves standard flats of $G(\Ga)$. Based on this, we show in \cite{cubulation} that if $H$ is quasi-isometric to $G(\Ga)$ with $\out(G(\Ga))$ finite, then $H$ acts properly and cocompactly on a $CAT(0)$ cube complex $X$ which is closely related to the universal cover of the Salvetti complex (Section \ref{Salvetti complex}) of $G(\Ga)$. 

Thus the problem is related to the commensurability classification of uniform lattices in $\aut(X)$ for a $CAT(0)$ cube complex $X$. This is well-understood when $X$ is a tree \cite{bass1990uniform}. In the case of right-angled buildings, Haglund observed a relation between commensurability of lattices and separability of certain subgroups \cite{haglund2006commensurability}, this together with \cite{agol2013virtual} give commensurability results for some classes of uniform lattices in hyperbolic right-angled buildings. However, when $X$ is not hyperbolic, there exist uniform lattices with very exotic behaviour \cite{wise1996non,burger_mozes}. The cases we are interested in is not (relatively) hyperbolic, and the above literature will serve as a guideline for putting suitable conditions on $\Ga$.

We close this section with a summary of previously known examples and non-examples of classes of RAAG's which satisfy commensurability rigidity in the sense of Question \ref{motivating question}. The rigid ones include
\begin{itemize}
\item The free group $F_m$ of $m$-generators \cite{stallings1968torsion,dunwoody1985accessibility,karrass1973finite}, \cite[1.C]{gromov1996geometric}.
\item The free Abelian groups $\Z^{n}$ \cite{gromov1981groups,bass1972degree}.
\item $F_{m}\times \Z^{n}$ \cite{whyte2010coarse}.
\item Free products of free groups and free Abelian groups \cite{behrstock2009commensurability}.
\end{itemize}

The non-rigid classes of RAAG's include
\begin{itemize}
\item $F_m\times F_{\ell}$ with $m,\ell\ge 2$ \cite{wise1996non,burger_mozes}.
\item $G(\Ga)$ with $\Ga$ being a tree of diameter $\ge 3$ \cite{behrstock2008quasi}.
\end{itemize}

\subsection{Main results}
All graphs in this section are simple. A finite graph $\Gamma$ is \textit{star-rigid} if for any automorphism $\alpha$ of $\Gamma$, $\alpha$ is identity whenever it fixes a closed star of some vertex $v\in\Gamma$ point-wise. Two groups $H$ and $G$ are \textit{commensurable} if there exist finite index subgroups $H'\le H$, $G'\le G$ such that $H'$ and $G'$ are isomorphic. 
\begin{thm}
\label{rigidity}
Suppose $\Ga$ is a graph such that
\begin{enumerate}
\item $\Ga$ is star-rigid;
\item $\Ga$ does not contain induced 4-cycle;
\item $\out(G(\Ga))$ is finite.
\end{enumerate}
Then any finite generated group quasi-isometric to $G(\Ga)$ is commensurable to $G(\Ga)$.
\end{thm}

One may find it helpful to think about the case when $\Ga$ is a pentagon. Condition (2) is sharp in the following sense.

\begin{thm}
\label{non-rigid RAAG's}
Suppose $\Ga$ is a graph which contains an induced 4-cycle. Then there exists a finitely generated group $H$ which is quasi-isometric to $G(\Ga)$ but not commensurable to $G(\Ga)$.
\end{thm}

It follows that there are plenty of RAAG's with finite outer automorphism group which are not rigid in the sense of Question \ref{motivating question}. This is quite different from the conclusion in the internal quasi-isometry classification of RAAG's \cite{bks2,huang2014quasi}.

\begin{remark}
It is natural to ask how much portion of RAAG's which are rigid in the sense of Question \ref{motivating question} are characterized by Theorem \ref{rigidity}. This is related to a graph theoretic question as follows. There is a 1-1 correspondence between finite simplicial graphs up to graph isomorphisms and RAAG's up to group isomorphisms \cite{droms1987isomorphisms}. Thus it makes sense to talk about random RAAG's by considering random graphs. It was shown in \cite{charney2012random,day2012finiteness} that a random RAAG satisfies (3). Moreover, since a finite random graph is asymmetric, it satisfies (1). However, we ask whether (1) and (3) are also generic properties among finite simplicial graphs which do not have induced 4-cycles. A positive answer to this question, together with Theorem \ref{rigidity} and Theorem \ref{non-rigid RAAG's}, would mean almost all RAAG's which are rigid in the sense of Question \ref{motivating question} are characterized by Theorem \ref{rigidity}.
\end{remark}

Condition (3) of Theorem \ref{rigidity} is motivated by the study of asymptotic geometry of RAAG's. It turns out that $G(\Ga)$ satisfies a form of vertex rigidity (\cite[Theorem 1.6]{bks2} and \cite[Theorem 4.15]{huang2014quasi}) if and only if $\out(G(\Ga))$ is finite. The motivation for condition (1) and condition (2) is explained after Corollary \ref{arithmeticity}.

Let $S(\Ga)$ be the Salvetti complex of $G(\Ga)$ (Section \ref{Salvetti complex}) and let $X(\Ga)$ be the universal cover of $S(\Ga)$. The following is a simpler version of Theorem \ref{rigidity} in the case of lattices in $\aut(X(\Ga))$.

\begin{thm}
\label{lattice}
Suppose $\Ga$ is a graph such that
\begin{enumerate}
\item $\Ga$ is star-rigid;
\item $\Ga$ does not contain induced 4-cycle.
\end{enumerate}
Let $H$ be a group acting geometrically on $X(\Ga)$ by automorphisms. Then $H$ and $G(\Ga)$ are commensurable. Moreover, let $H_1, H_2\le\aut(X(\Ga))$ be two uniform lattices. Then there exists $g\in\aut(X(\Ga))$ such that $gH_1 g^{-1}\cap H_2$ is of finite index in both $gH_1 g^{-1}$ and $H_2$.
\end{thm}

Let $H\le \aut(X(\Ga))$. Recall that the \textit{commensurator} of $H$ in $\aut(X(\Ga))$ is
$\{g\in\aut(X(\Ga))\mid gHg^{-1}\cap H\ \textmd{is of finite index in}\ H\ \textmd{and}\ gHg^{-1}\}$. Then Theorem \ref{lattice} and \cite[Theorem I]{haglund2008finite} imply the following arithmeticity result.

\begin{cor}
\label{arithmeticity}
Let $\Ga$ be as in Theorem \ref{lattice}. Then the commensurator of any uniform lattice in $\aut(X(\Ga))$ is dense in $\aut(X(\Ga))$.
\end{cor}

Recall that edges in $X(\Ga)$ can be labelled by vertices of $\Ga$ (see Section \ref{Salvetti complex}). Condition (1) is to guarantee that every $H$ has label-preserving finite index subgroup. We caution the reader that this condition is not equivalent to the \textquotedblleft type preserving\textquotedblright\ condition introduced in \cite{haglund2006commensurability} in the studying of right-angled buildings, though these two conditions have similar flavour. 

Theorem \ref{lattice} is a consequence of the following result.

\begin{thm}
\label{label-preserving rigidity}
Suppose $\Ga$ is a graph which does not contain induced 4-cycle. Let $H$ be a group acting geometrically on $X(\Ga)$ by label-preserving automorphisms. Then $H$ is commensurable to $G(\Ga)$.
\end{thm} 

For a more general version of Theorem \ref{label-preserving rigidity}, see Theorem \ref{good cover}.
 
The no induced 4-cycle condition is motivated by the examples of exotic groups acting geometrically on the product of two trees \cite{wise1996non,burger_mozes}. It is not hard to see that $X(\Ga)$ contains an isometrically embedded product of two trees of infinitely many ends if and only if $\Ga$ contains an induced 4-cycle. It turns out that one can modify their examples such that the exotic action on the product of two trees extends to an exotic action on $X(\Ga)$, which gives a converse to Theorem \ref{label-preserving rigidity}.

\begin{thm}
\label{label-preserving non-rigid RAAG's}
Let $\Ga$ be a graph which contains an induced 4-cycle. Then there exists a torsion free group $H$ acting geometrically on $X(\Ga)$ by label-preserving automorphisms such that $H$ is not residually finite. In particular, $H$ and $G(\Ga)$ are not commensurable.
\end{thm} 

Note that even in the case where $\Ga$ is a 4-cycle, the above theorem is not completely trivial, since the examples mentioned above in \cite{wise1996non,burger_mozes} are not label-preserving. Also Theorem \ref{non-rigid RAAG's} follows from Theorem \ref{label-preserving non-rigid RAAG's}.

\subsection{Sketch of proof and organization of the paper}
We refer to Section \ref{Salvetti complex} for our notations. Let $G(\Ga)$ be a RAAG with finite outer automorphism group and let $H$ be a group quasi-isometric to $G(\Ga)$. For simplicity we assume $H$ is torsion free and $H$ acts geometrically on $X(\Ga)$. In general, $H$ admits a geometric model which is quite similar to $X(\Ga)$, which is discussed in Section \ref{sec_blow-up building} and Section \ref{sec_geometric model}.

\textit{Step 1:} We orient edges in $S(\Ga)$ and label them by vertices of $\Ga$. This lifts to $G(\Ga)$-equivariant orientation and labelling of edges in $X(\Ga)$. We are done if $H$ happens to preserve both the orientation and labelling. By a simple observation (Lemma \ref{label-preserving}), we can assume $H$ acts on $G(\Gamma)$ in a label-preserving way by passing to a finite index subgroup. However, the issue with orientation is more serious.

We induct on the complexity of the underlying graph, and assume Theorem \ref{label-preserving rigidity} holds for any proper subgraph of $\Ga$. Pick a vertex $v\in\Ga$ and let $\Gamma'\subset\Gamma$ be the induced subgraph spanned by vertices in $\Ga\setminus\{v\}$. Then there is a canonical embedding $S(\Ga')\to S(\Ga)$. Note that $G(\Gamma)$ is an HNN-extension of $G(\Gamma')$ along the subgroup $G(lk(v))\subset G(\Ga')$. Let $T$ be the associated Bass-Serre tree. Alternatively, $T$ is obtained from $X(\Ga)$ by collapsing edges which are in the lifts of $S(\Ga')$. 

Since $H$ is label-preserving, there is an induced action $H\curvearrowright T$. Up to passing to a subgroup of index 2, we assume $H$ acts on $T$ without inversion. This induces a graph of groups decomposition of $H$ and a graph of spaces decomposition of $K=X(\Gamma)/H$. Each vertex group acts geometrically on $X(\Ga')$ by label-preserving automorphisms (since the universal cover of each vertex space is isometric to $X(\Gamma')$), hence it is commensurable to $G(\Gamma')$ by the induction assumption. It follows that each vertex space has a finite sheet cover which is a special cube complex. At this point, the reader is not required to know what is exactly a special cube complex. One may perceive it as a cube complex with some nice combinatorial features and we will explain the relevant properties later.

Each cover of $K$ has an induced graph of spaces structure. We claim there exists a finite sheet cover $\bar{K}\to K$ such that each vertex space of $\bar{K}$ is a special cube complex. It is not hard to deduce the above theorem from this claim. The relation between specialness and commensurability is discussed in Section \ref{subsec_special and commensurability}.

Any edge group of $T/H$ acts geometrically on $X(lk(v))$. In general they could be very complicated, however, their intersections are controlled as follows.
\begin{lem}
\label{intersection}
Suppose $\Ga$ has no induced 4-cycle. Then the largest intersection of edge groups and their conjugates in a vertex group is virtually a free Abelian group. 
\end{lem}

\textit{Step 2:} For simplicity, we assume $K$ only has two vertex spaces $L$ and $R$, and one edge space $E\subset K$ such that the two ends of $E\times[0,1]$ are attached isometrically to locally convex subcomplexes $E_L\subset L$ and $E_R\subset R$ respectively. Let $L_1=L\cup (E\times [0,1/2])$ and $R_1=R\cup (E\times [1/2,1])$.

It follows from the work of Haglund and Wise \cite{haglund2012combination} that if $\pi_{1}(K)$ is hyperbolic, and if $\pi_{1}(E_L)$ and $\pi_{1}(E_R)$ are malnormal in $\pi_{1}(L)$ and $\pi_{1}(R)$ respectively, then $K$ has the desired finite sheet cover. Later the malnormality is dropped in \cite{wisestructure} by using the cubical small cancellation theory, and hyperbolicity is relaxed to relative hyperbolicity with respect to Abelian subgroups. However, in general it is easy to give counterexamples without any hyperbolicity assumption. Moreover, these works did not take quasi-isometry invariance into consideration.

In most of our cases, both relative hyperbolicity and malnormality fail. However, $K$ has more structure than a generic special cube complex and malnormality does not fail in a terrible way (Lemma \ref{intersection}). Due to lack of hyperbolicity, we will get around the cubical small cancellation or Dehn filling argument, and use a different collapsing argument based on the blow-up building construction in \cite{cubulation}, which also applies to group quasi-isometric to $G(\Ga)$.

Let $L'\to L$ be a finite sheet special cover. It also induces a finite sheet cover $L'_1\to L_1$. Similarly we define $R'$ and $R'_1$. Usually there are more than one lifts of $E_L$ in $L'$, each lift gives rise to a half-tube in $L'_1$. It suffices to match half-tubes in $L'_1$ with half-tubes in $R'_1$. Suppose there are bad half-tubes in $L'_1$ which do not match up with half-tubes in $R'_1$. It is natural to ask whether it is possible to pass to finite sheet covers of $L'_1$ such that bad half-tubes are modified in the cover while good half-tubes remain unchanged.

One ideal situation is the following. Suppose $A_{1},A_{2}\subset L'$ are elevations of $E_L$ and suppose there exists a retraction $r:L'\to A_{1}$ such that $r_{\ast}(\pi_{1}(A_{2}))$ is trivial. Then there exists cover $L''\to L'$ which realizes any further cover of $A_{1}$ without changing $A_{2}$. We want to achieve at least some weaker version of this ideal situation. Since $L'$ is a special cube complex, $L'$ has a finite sheet cover which retracts onto $A_{1}$. This is constructed in \cite{MR2377497}, in which it is called the \textit{canonical retraction}. Then we need to look at the images of the lifts of $A_2$ under this retraction. Due to the failure of malnormality, $r_{\ast}(\pi_{1}(A_{2}))$ is in generally non-trivial, however, Lemma \ref{intersection} suggested that it is reasonable to control the retraction image such that it is not more complicated than a torus.

The rough idea is that we first collapse all the tori in $K$. Then the tubes become collapsed tubes and the previous statement is equivalent to that the projection of a collapsed tube to another collapsed tube is contractible. The reader may notice that if we collapse all the tori in $K$, then the resulting space is a point. In order to make this idea work, we first exploded $K$ with respect to the intersection pattern of the tori in $K$, then collapse the tori. The resulting space $K'$ is a space which encodes the intersection pattern of tori in $K$. This is done in the setting of blow-up building developed in \cite{cubulation}. Moreover, this also works in the case when $H$ is a group quasi-isometric to $X(\Ga)$.

By killing certain holonomy in $K$ (see Section \ref{sec_branched complexes with trivial holonomy}), $K'$ becomes a special cube complex. Moreover, one deduce that $\pi_1(K')$ is Gromov-hyperbolic from the no induced 4-cycle condition (see Section \ref{subsec_wall projection}). Now we are in a situation to apply the work of Haglund and Wise on hyperbolic special cube complex.

\textit{Step 3:} While matching the tubes, we also need to keep track of finer information about these retraction tori, such as the length of the circles in the tori, and how other circles retracts onto a particular circle. It turns out that the construction in \cite{MR2377497} does not quite preserve this information since the canonical completion is too large in some sense (Remark \ref{larger circle}). We need a modified version of completion and retraction, which is discussed in Section \ref{subsec_modified completeions and retractions}.

\textit{Step 4:} Given that the retraction images are tori, and the retraction preserves finer combinatorial information of tori in $K$, we construct the desired cover of $K$ in Section \ref{subsec_matching}. The argument is a modified version of \cite[Section 6]{haglund2012combination}. 

\textit{Step 5:} We show how to drop the finite outer automorphism condition in the case of group acting geometric on $X(\Ga)$. We will explode $K$ in a different way and decompose it into suitable vertex spaces and edge spaces. See Section \ref{sec_uniform lattice}.

\textit{Step 6:} When there is an induced 4-cycle in $\Ga$, the largest intersection of edges groups and their conjugates in a vertex group may contain a free group of rank 2. In this case it is impossible for $K$ to be virtually special in general. The counterexamples are given in Section \ref{sec_failure of commensurability}.

\subsection{Acknowledgement} This paper would not be possible without the helpful discussions with B. Kleiner. In particular, he pointed out a serious gap in the author's previous attempt to prove a special case of the main theorem. Also the author learned Lemma \ref{normal subgroup} from him. The author thanks D. T. Wise for pointing out the reference \cite{haglund2006commensurability} and X. Xie for helpful comments and clarifications.

\section{Preliminaries}

\subsection{Right-angled Artin groups and Salvetti complexes}
\label{Salvetti complex}
We refer to \cite{charney2007introduction} for background on right-angled Artin groups. Throughout this section, $\Ga$ will be a finite simplicial graph.

\begin{definition}[Salvetti complex]
Denote the vertex set of $\Gamma$ by $\{v_{i}\}_{i=1}^{m}$. We associated each $v_i$ with a standard circle $\Bbb S^{1}_{v_{i}}$ and choose a base point $p_{i}\in \Bbb S^{1}_{v_{i}}$. Let $\Bbb T^{m}=\Pi_{i=1}^{m} \Bbb S^{1}_{v_{i}}$. Then $\Bbb T^{m}$ has a natural cube complex structure. Then $\Delta$ gives rise to a subcomplex $T_{\Delta}=\Pi_{v_{i}\in v(\Delta)}\Bbb S^{1}_{v_{i}}\times \Pi_{v_{i}\notin v(\Delta)}\{p_{i}\}$. Then $S(\Gamma)$ is defined to be the subcomplex of $\Bbb T^{m}$ which is the union of all $T_{\Delta}$'s with $\Delta$ ranging over all clique subgraphs of $\Gamma$.  
\end{definition}

$S(\Gamma)$ is a non-positively curved cube complex whose 2-skeleton is the presentation complex of $G(\Gamma)$, so $\pi_{1}(S(\Gamma))\cong G(\Gamma)$. The closure of each $k$-cell in $S(\Gamma)$ is a $k$-torus, which is called a \textit{standard torus}. A standard torus of dimension 1 is also called a \textit{standard circle}. The \textit{dimension} of $G(\Gamma)$ is the dimension of $S(\Gamma)$. Recall that $\Ga'\subset\Ga$ is an \textit{induced subgraph} if vertices of $\Ga'$ are adjacent in $\Ga'$ if and only if they are adjacent in $\Ga$. Each induced subgraph $\Ga'\subset\Ga$ gives rise to an isometric embedding $S(\Ga')\to S(\Ga)$. The universal cover of $S(\Gamma)$ is a $CAT(0)$ cube complex, which we denote by $X(\Ga)$. We label standard circles of $S(\Ga)$ by vertices of $\Gamma$, and this lifts to a $G(\Ga)$-invariant edge labelling of $X(\Gamma)$. 

\begin{definition}($\Ga'$-components)
\label{components}
Let $K$ be a cube complex. Suppose edges of $K$ are labelled by vertices of $\Ga$. Pick a induced subgraph $\Ga'\subset\Ga$. A \textit{$\Ga'$-component} $L\subset K$ is a subcomplex such that
\begin{enumerate}
\item $L$ is connected.
\item Each edge in $L$ is labelled by a vertex in $\Ga'$. Moreover, for each vertex $v\in\Ga'$, there exists an edge in $L$ which is labelled by $v$.
\item $L$ is maximal with respect to (1) and (2).
\end{enumerate}
Here $\Ga'$ is called the \textit{support} of $L$.
\end{definition}

Note that an $\emptyset$-component is a vertex in $K$. If $\Ga'$ is a complete graph of $k$ vertices, $\Ga'$-components of $X(\Ga)$ are isometrically embedded $k$ dimensional Euclidean spaces. They are called \textit{standard flats}. When $k=1$, we also call them \textit{standard geodesics}. Vertices in $X(\Ga)$ are understood to be $0$-dimensional standard flats. Also note that in order to define $\Ga'$-components in $S(\Ga)$ and $X(\Ga)$, the second part of Definition \ref{components} (2) is not necessary, since for each vertex $x$ in $S(\Ga)$ or $X(\Ga)$, and for each vertex $v\in\Ga$, there exists a $v$-labelled edge containing $x$. However, we will encounter other complexes later where this property is not true.

\begin{lem}
\label{convexity of components}
Suppose $K$ is non-positively curved and parallel edges of $K$ are labelled by the same vertex. Then each $\Ga'$-component $L$ is locally convex. In particular, if $K$ is $CAT(0)$, then $L$ is $CAT(0)$.
\end{lem}

\begin{proof}
It suffices to check for each vertex $x\in L$, if a collection of edges $\{e_i\}_{i=1}^{n}$ emanating from $x$ span a cube in the ambient complex, then this cube is in $L$. However, this follows from the fact that parallel edges have the same label, and the maximality of $L$.
\end{proof}

The following object was first introduced in \cite{kim2013embedability}. 

\begin{definition}[extension complex]
\label{extension complex}
The \textit{extension complex} $\P(\Ga)$ of a finite simplicial graph $\Ga$ is defined as follows. The vertices of $\mathcal{P}(\Gamma)$ are in 1-1 correspondence with the parallel classes of standard geodesics in $X(\Gamma)$. Two distinct vertices $v_{1},v_{2}\in\mathcal{P}(\Gamma)$ are connected by an edge if and only if there is a standard geodesic $l_{i}$ in the parallel class associated with $v_{i}$ ($i=1,2$) such that $l_{1}$ and $l_{2}$ span a standard 2-flat. Then $\mathcal{P}(\Gamma)$ is defined to be the flag complex of its 1-skeleton, namely we build $\mathcal{P}(\Gamma)$ inductively from its 1-skeleton by filling a $k$-simplex whenever we see the $(k-1)$-skeleton of a $k$-simplex. 

Since each complete subgraph in the 1-skeleton of $\mathcal{P}(\Gamma)$ gives rise to a collection of mutually orthogonal standard geodesics lines in $X(\Ga)$, there is a 1-1 correspondence between $k$-simplexes in $\mathcal{P}(\Gamma)$ and parallel classes of standard $(k+1)$-flats in $X(\Gamma)$. For standard flat $F\subset X(\Gamma)$, we denote the simplex in $\mathcal{P}(\Gamma)$ which corresponds to standard flats parallel to $F$ by $\Delta(F)$. 
\end{definition}

Each vertex $v\in\P(\Ga)$ is labelled a vertex of $\Ga$ in the following way. Pick standard geodesic $\ell\subset X(\Ga)$ such that $\Delta(\ell)=v$ and pick edge $e\subset\ell$. We label $v$ by the label of $e$. Note that this does not depend on the choice of $\ell$ in the parallel class, and the edge inside $\ell$. This labelling induces a map from the 0-skeleton of $\P(\Ga)$ to vertices in $\Ga$, which can be extended to a simplicial map $\P(\Ga)\to F(\Ga)$, here $F(\Ga)$ is the flag complex of $\Ga$.

We refer to \cite[Section 2.3.3]{kleiner1997rigidity} for the definition and properties of the \textit{parallel set} of a convex subset in a $CAT(0)$ space.

\begin{definition}
\label{v-parallel set}
For vertex $v\in\P(\Ga)$, the \textit{$v$-parallel set}, which we denote by $P_v$, is the parallel set of a standard geodesic $\ell\subset X(\Ga)$ such that $\Delta(\ell)=v$. Note that $P_v$ does not depend on the choice of the standard geodesic $\ell$ in the parallel class.
\end{definition}

\begin{definition}
\label{definition of links}
Pick induced subgraph $\Ga'\subset\Ga$. The \textit{link} of $\Ga'$, denoted by $lk(\Ga')$, is the induced subgraph spanned by vertices which are adjacent to every vertex in $\Ga'$. The \textit{closed star} of $\Ga'$, denoted by $St(\Ga')$, is the induced subgraph spanned by vertices in $\Ga'$ and $lk(\Ga')$.

Let $K$ be a polyhedral complex and pick $x\in K$. The \textit{geometrical link} of $x$ in $K$, denoted by $Lk(x,K)$, is the object defined in \cite[Chapter I.7.15]{bridson1999metric}.
\end{definition}

In general, these two notion of links do not agree on graphs.

\begin{lem}
\label{parallel set}
$($\cite[Lemma 3.4]{huang2014quasi}$)$
Let $K$ be a $\Ga'$-component in $X(\Ga)$. Then the parallel set $P_{K}$ of $K$ is exactly the $St(\Ga')$-component containing $K$. Moreover, $P_K$ admits a splitting $P_K=K\times K^{\perp}$, where $K^{\perp}$ is isomorphic to a $lk(\Ga')$-component.  
\end{lem}

In particular, for vertex $v\in\P(\Ga)$, $P_v$ is the $St(\bar{v})$-component that contains $\ell$, where $\ell$ is a standard geodesic with $\Delta(\ell)=v$ and $\bar{v}\in\Ga$ is the label of edges in $\ell$. 
\subsection{Special cube complex}
We refer to Section 2.A and Section 2.B of \cite{haglund2012combination} for background about cube complexes and hyperplanes. An edge is \textit{dual} to a hyperplane if it has nonempty intersection with the hyperplane. A hyperplane $h$ is \textit{2-sided} if it has a small neighbourhood which is a trivial bundle over $h$. Two edges are \textit{parallel} if they are dual to the same hyperplane.

\label{subsec_special cube complex}
Let $X$ be a cube complex and pick a two-sided hyperplane $h\subset X$. Then the parallelism induces a well-defined orientation for edges dual to $h$. For each vertex $v\in X$, let $g_{v}$ be the graph made of all dual edges of $h$ that contains $v$ ($g_v$ could be empty). Then the geometric link $Lk(v,g_{v})$ is a discrete graph and each of its vertices is either incoming or outcoming depending on the orientations of edges.

We say $h$ \textit{self-osculates} if 
\begin{enumerate}
\item $h$ is two-sided and embedded;
\item there exists vertex $v\in X$ such that $Lk(v,g_{v})$ has more than one points.
\end{enumerate}
In this case, $h$ \textit{directly self-osculates} if there exist vertex $v\in X$ such that $Lk(v,g_{v})$ has at least two vertices which are both incoming or both outcoming, otherwise, $h$ \textit{indirectly self-osculates}. For example, a circle with one edge has an indirectly self-osculating hyperplane.

Let $h_{1}$ and $h_{2}$ be a pair of embedded two-sided hyperplanes. Then they \textit{interosculate} if there exist edges $a_{i},b_{i}$ dual to $h_{i}$ $(i=1,2)$, and vertices $v_{a}\in a_{1}\cap a_{2}$, $v_{b}\in b_{1}\cap b_{2}$ such that 
\begin{enumerate}
\item there exist a vertex in $Lk(v_{a},a_{1})$ and a vertex in $Lk(v_{a}, a_{2})$ which are adjacent in $Lk(v_{a},X)$ (note that if $a_1$ is not embedded, then $Lk(v_{a},a_{1})$ has two points);
\item there exist a vertex in $Lk(v_{b},b_{1})$ and a vertex in $Lk(v_{b}, b_{2})$ which are not adjacent in $Lk(v_{b},X)$.
\end{enumerate}

\begin{definition}
\label{special}
A non-positively curved cube complex $X$ is \textit{special} if
\begin{enumerate}
\item Each hyperplane is 2-sided and embedded.
\item No hyperplane directly self-osculate.
\item No two hyperplanes interosculate.
\end{enumerate}
$X$ is \textit{directly special} if $X$ is special and no hyperplane of $X$ self-osculate.
\end{definition}

\begin{thm}
\cite[Lemma 2.6]{haglund2012combination}
Let $X$ be a non-positively cube complex. Then $X$ is special if and only if there exists a local-isometry $X\to S(\Gamma)$ for some Salvetti complex $S(\Gamma)$.
\end{thm}

\begin{lem}
\label{directly special}
\cite{haglund2012combination} Let $X$ be a compact special cube complex. Then $X$ has a finite cover which is directly special. Moreover, being directly special is preserved under passing to any further cover.
\end{lem}

\subsection{Canonical completions and retractions}
\label{subset_the canonical comletion}
We follow \cite[Section 3]{haglund2012combination}.

Given compact special cube complexes $A$ and $X$, and a locally isometry $A\to X$, one can construct a finite cover of $X$ that contains a copy of $A$. This finite cover is called the \textit{canonical completion} of the pair $(A,X)$ and is denoted by $\mathsf{C}(A,X)$. Moreover, there is a \textit{canonical retraction} map $r:\mathsf{C}(A,X)\to A$. Now we describe the construction of $\mathsf{C}(A,X)$ and the retraction map $r$.

\textit{Case 1:} $X$ is a circle with a single edge. Then each connected component of $A$ is either a point, or a circle, or a path. We attach an extra edge to each component which is not a circle to make it a circle. The resulting space is denoted by $\mathsf{C}(A,X)$. It is clear that $\mathsf{C}(A,X)$ is finite cover of $X$ and $A\subset\mathsf{C}(A,X)$. Moreover, there is a retraction $\mathsf{C}(A,X)\to A$ by sending each extra edge to the component of $A$ it was attached along.

\textit{Case 2:} $X$ is a wedge of finitely many circles $\{c_{i}\}_{i=1}^{n}$. Let $A_{i}$ be the inverse image of $c_{i}$ under $A\to X$. We define $\mathsf{C}(A,X)$ to be the union of all $\mathsf{C}(A_{i},c_{i})$'s identified along their vertex sets, and define the canonical retraction $r:\mathsf{C}(A,X)\to A$ to be the map induced by $\mathsf{C}(A_{i},c_{i})\to A_{i}$. It is still true that $\mathsf{C}(A,X)\to X$ is a finite sheet covering map and $A\subset\mathsf{C}(A,X)$.

\text{Case 3:} $X$ is the Salvetti complex of some RAAG. Then the 1-skeleton $X^{1}$ of $X$ is a wedge of circles. It turns out that there is a natural way to attach higher dimensional cells to $\mathsf{C}(A^{1},X^{1})$ to obtain a non-positively curved cube complex $\mathsf{C}(A,X)$ (see \cite[Theorem 2.6]{bou2015residual}). Moreover, the covering map $\mathsf{C}(A^{1},X^{1})\to X^{1}$ and the retraction $\mathsf{C}(A^{1},X^{1})\to A^{1}$ extent to a covering map $\mathsf{C}(A,X)\to X$ and a retraction $\mathsf{C}(A,X)\to A$. The following is immediate from our construction.

\begin{lem}
\label{product and canonical completion}
For $i=1,2$, let $B_i\to X_i$ be a local isometry from $B_i$ to a Salvetti complex. Then the canonical completion $\C(B_1\times B_2,X_1\times X_2)$ with respect to the product of these two local isometries is naturally isomorphic (as cube complexes) to $\C(B_1,X_1)\times\C(B_2,X_2)$.
\end{lem}

\textit{Case 4:} $X$ is any compact special cube complex. Let $\Gamma$ be a graph such that its vertex set corresponds to the hyperplanes in $X$, and two vertex are adjacent if the corresponding hyperplanes cross each other. Such graph is called the \textit{intersection graph} of $X$. Suppose $R=S(\Gamma)$. Since $X$ is special, there is a local isometry $X\to R$, which induces a local isometry $A\to R$. We define $\mathsf{C}(A,X)$ to be the pull-back of the covering map $\mathsf{C}(A,R)\to R$ which fits into the following commuting diagram: 
\begin{center}
$\begin{CD}
@. \mathsf{C}(A,X)                         @>>>        \mathsf{C}(A,R)\\
@. @VVV                                   @VVV\\
A @>>>X          @>>>        R
\end{CD}$
\end{center}

Recall that an $i$-cube in $\mathsf{C}(A,X)$ can be represented by a pair of $i$-cubes in $\mathsf{C}(A,R)$ and $X$ that are mapped to the same $i$-cube in $S(\Gamma)$. Thus there exists a naturally embedded copy of $A$ in $\mathsf{C}(A,X)$ by considering $A\to \mathsf{C}(A,R)$ and $A\to X\to R$. The component of $\mathsf{C}(A,X)$ which contains this copy of $A$ is called the \textit{main component}. The canonical retraction is defined to be the composition $\mathsf{C}(A,X)\to\mathsf{C}(A,R)\to A$.

We recall the following notion from \cite[Section 2.1]{caprace2011rank}. It is also called a projection-like map in \cite{haglund2012combination}.
\begin{definition}
\label{cubical map}
A cellular map between cube complexes is \textit{cubical} if its restriction $\sigma\to\tau$ between cubes factors as $\sigma\to\eta\to\tau$, where the first map $\sigma\to\eta$ is a natural projection onto a face of $\sigma$ and the second map $\eta\to\tau$ is an isometry. 
\end{definition}

Note that if the inverse image of each edge in $S(\Gamma)$ under $A\to S(\Gamma)$ is a disjoint union of vertices and single edges, then the retraction $\mathsf{C}(A^{1},R^{1})\to A^{1}$ is cubical. Hence the canonical retraction $\mathsf{C}(A,X)\to A$ is cubical. For example, the assumption is satisfied when $X$ is directly special.

\begin{definition}
\cite[Definition 3.14]{haglund2012combination}  Let $X$ denote a cube complex. Let $A$ and $B$ be subcomplexes of $X$. We define $\wpj_{X} (A\to B)$, the \textit{wall projection} of $A$ onto $B$ in $X$, to be the union of $B^{0}$ together with all cubes of $B$ whose edges are all parallel to edges of $A$.
\end{definition}

\begin{lem}
\label{wpj}
\cite[Lemma 3.16]{haglund2012combination}  Let $A$ and $D$ be locally convex subcomplex of a directed special cube complex $B$. Let $\hat{D}$ denote the preimage of $D$ in $\mathsf{C}(A,B)$, and let $r:\mathsf{C}(A,B)\to A$ be the canonical retraction map. Then $r(\hat{D})\subset\wpj_{B}(D\to A)$.
\end{lem}

Let $A$ be a locally convex subcomplex in a special cube complex $X$. It is natural to ask what is the inverse image of $A$ under the covering $\mathsf{C}(A,X)\to X$. This usually depends on how $A$ sits inside $X$.

\begin{definition}
\cite[Definition 3.10]{haglund2012combination}
Let $D\to C$ be a combinatorial map between cube complexes. Then a hyperplane of $D$ is mapped to a hyperplane of $C$. Hence there is an induced map $V_{D}\to V_{C}$ between the set of hyperplanes in $D$ and $C$ respectively. The map $D\to C$ is \textit{wall-injective} if $V_{D}\to V_{C}$ in injective.
\end{definition}

\begin{lem}
\cite[Lemma 3.13]{haglund2012combination}
Let $X$ be a special cube complex and let $A\subset X$ be a wall-injective locally convex subcomplex. Then the preimage of $A$ in $\mathsf{C}(A,X)$ is canonically isomorphic to $\mathsf{C}(A,A)$.
\end{lem}

\begin{remark}
\label{larger circle}
Let $A$ be a special cube complex. Usually $\mathsf{C}(A,A)$ is much larger than $A$. For example, let $A$ be a circle made of $n$-edges for $n\ge 3$. Then $\mathsf{C}(A,A)$ is the disjoint union of a circle of length $n$ and a circle of length $n(n-1)$.
\end{remark}

\begin{lem}
\label{wall-injective}
\cite[Corollary 3.11]{haglund2012combination}
Let $A\to X$ be a local-isometry. If the canonical retraction $\C(A,X)\to A$ is cubical, then $A$ is wall-injective in $\C(A,X)$. In particular, $A$ is wall-injective in $\C(A,X)$ when $X$ is directly special. 
\end{lem}

\begin{definition}
\label{elevation}
\cite[Definition 3.17]{haglund2012combination}  Let $\pi:\bar{X}\to X$ denote a covering map. Let $A\subset X$ denote a connected subspace. An \textit{elevation} of $A$ to $\bar{X}$ is a connected component of the preimage of $A$ under $\bar{X}\to X$. If we choose base points $p\in A$ and $\bar{p}\in\bar{X}$ with $\pi(\bar{p})=p$, then the \textit{based elevation} of $A$ is the component of the preimage of $A$ that contains $\bar{p}$.

If $f:A\to X$ is any continuous map for $A$ connected. We define an elevation $\bar{A}$ of $A$ to be a cover of $A$ corresponding to the subgroup $f^{-1}_{\ast}(\pi_{1}(\bar{X}))$ (this depends on our choice of base points in $A$, $X$ and $\bar{X}$). Then $\bar{A}$ fits into the following commuting diagram:
\begin{center}
$\begin{CD}
\bar{A}                           @>>>        \bar{X}\\
@VVV                                   @VVV\\
A          @>>>        X
\end{CD}$
\end{center}
\end{definition}

\begin{remark}
The above two notions of elevation are consistent in the following sense. Suppose $f:A\to X$ is continuous and let $C$ be the mapping cylinder of $f$. Since $C$ is homotopic to $X$, the covering map $\bar{X}\to X$ induces a covering map $\bar{C}\to C$. Then elevations of $A\to X$ corresponds to components in the inverse image of $A$ in $\bar{C}\to C$. 

Alternatively, let $\tilde{A}\to A$ be the pull-back of the $\bar{X}\to X$. Then an elevation of $A\to X$ corresponds to a connected component in $\tilde{A}$.
\end{remark}

\section{The structure of blow-up building}
\label{sec_blow-up building}
\subsection{The blow-up building} We briefly review the Davis realization of right-angled buildings and the construction of blow-up buildings in \cite{cubulation}. The reader can find detailed proofs of the statements mentioned below in \cite{cubulation}.

Let $\P$ be the poset of standard flats in $X(\Ga)$. Recall that an \textit{interval} of $\P$ is a subset of form $I_{a,b}=\{x\in \P\mid a\le x\le b\}$ for some $a,b\in\P$. Since every interval of $\P$ is a Boolean lattice of finite rank, one can construct a cube complex $|\B|$ whose poset of cubes is isomorphic to poset of intervals in $\P$ (\cite[Proposition A.38]{abramenko2008buildings}). $|\B|$ is called the \textit{Davis realization of the right-angled building associated with $G(\Ga)$} and it is a $CAT(0)$ cube complex \cite{davis1994buildings}.

More precisely, there is a 1-1 correspondence between $k$-cubes in $|\B|$ and intervals of form $I_{F_1,F_2}$ where $F_1\subset F_2$ are two standard flats in $X(\Ga)$ with $\dim(F_2)-\dim(F_1)=k$. In particular, vertices of $|\B|$ correspond to standard flats in $X(\Ga)$. We label each vertex of $|\B|$ by the support (Definition \ref{components}) of the corresponding standard flat (if a vertex of $|\B|$ corresponds to a $0$-dimensional standard flat, then it is labelled by the empty set), and the \textit{rank} of this vertex is defined to be the number of vertices in its label. Moreover, the vertex set of $|\B|$ inherits a partial order from $\P$. 

We label each edge of $|\B|$ by the unique vertex of $\Ga$ in the symmetric difference of the labels of its two endpoints. Note that two parallel edges are labelled by the same vertex of $\Ga$, hence there is an induced labelling of hyperplanes in $|\B|$. If two hyperplanes cross, then their labels are adjacent in $\Ga$.

Let $K\subset X(\Ga)$ be a $\Ga'$-component and let $\P_K\subset \P$ be the sub-poset of standard flats inside $K$. Then $\P_K$ gives rise to a convex subcomplex $|\B|_{K}\subset |\B|$. By definition, $|\B|_K$ is isomorphic to the Davis realization of the right-angled building associated with $G(\Ga')$.

There is an induced action $G(\Ga)\acts|\B|$ which preserves the labellings of vertices and edges. The action is cocompact, but not proper - the stabilizer of a cube is isomorphic to $\Z^{n}$ where $n$ is the rank of the minimal vertex in this cube. The following construction is motivated by the attempt to blow-up $|\B|$ with respect to this data of stabilizers (in a possibly non-equivariant way). See Theorem \ref{inverse image are flats} for a precise statement.

\begin{definition}[branched lines and flats]
A \textit{branched line} is constructed from a copy of $\Bbb R$ with finitely many edges of length 1 attached to each integer point. We also require that the valence of each vertex in a branched line is bounded above by a uniform constant. This space has a natural simplicial structure and is endowed with the path metric. A \textit{branched flat} is a product of finitely many branched lines.

Let $\beta$ be a branched line. We call those vertices with valence $=1$ in $\beta$ the \textit{tips} of $\beta$, and the collection of all tips is denoted by $t(\beta)$. The copy of $\Bbb R$ in $\beta$ is called the \textit{core} of $\beta$. For a branched flat $F$, we define $t(F)$ to be the product of the tips of its factors, and the \textit{core} of $F$ to be the product of the cores of its factors.
\end{definition}

\begin{definition}[blow-up building]
\label{construction}
The following construction is a special case of \cite[Section 5.2, Section 5.3]{cubulation}. Let $\P(\Ga)$ be the extension complex. Pick a vertex $v\in\P(\Ga)$, we associate $v$ with a branched line $\beta_{v}$. Moreover, for each standard geodesic line $\ell$ with $\Delta(\ell)=v$ (see Definition \ref{extension complex}), we associate a bijection $f_{v,\ell}$ from the 0-skeleton $\ell^{(0)}$ of $\ell$ to $t(\beta_{v})$ such that
if two standard lines $\ell,\ell'$ are parallel, then $f_{v,\ell'}=f_{v,\ell}\circ p$, where $p:\ell'^{(0)}\to \ell^{(0)}$ is the map induced by parallelism. These $f_{v,\ell}$'s are called \textit{blow-up data}.

We associate each standard flat $F\subset X(\Ga)$ with a space $\beta_F$ as follows. If $F$ is a 0-dimensional standard flat (i.e. $\Delta(F)=\emptyset$), then $\beta_F$ is a point. Suppose $\Delta(F)\neq\emptyset$. Let $F=\prod_{v\in\Delta(F)}\ell_{v}$ be a product decomposition of $F$, here $v$ is a vertex in $\Delta(F)$, and $\ell_{v}\subset F$ is a standard geodesic line with $\Delta(\ell_{v})=v$. Then $\beta_{F}=\prod_{v\in\Delta(F)}\beta_{v}$.

For standard flats $F'\subset F$, suppose $F'=\prod_{v\in\Delta(F')}\ell_{v}\times\prod_{v\notin\Delta(F')}\{x_{v}\}$ ($x_{v}$ is a vertex in $\ell_{v}$). Then we define an isometric embedding $\beta_{F'}\hookrightarrow \beta_F$ as follows:
\begin{center}
$\beta_{F'}=\prod_{v\in\Delta(F')}\beta_v\cong \prod_{v\in\Delta(F')}\beta_{v}\times\prod_{v\notin\Delta(F')}\{f_{v,\ell_{v}}(x_{v})\}\hookrightarrow \prod_{v\in\Delta(F)}\beta_{v}=\beta_F$.
\end{center}
One readily verify that the above construction gives a functor from the poset $\P$ to the category of branched flats with isometric embeddings as morphisms. Let $Y(\Ga)$ be the space obtained by gluing the collection of all $\beta_F$'s according the isometric embeddings defined as above. The following properties follows from functorality of $F\to \beta_F$.
\begin{enumerate}
\item Each $\beta_F$ is embedded in $Y(\Ga)$. It is called a \textit{standard branched flat}. Then core of a standard branched flat is called a \textit{standard flat}. There is a 1-1 correspondence between standard flats in $X(\Ga)$ and standard flats in $Y(\Ga)$.
\item $\beta_{F_1}\cap \beta_{F_2}=\beta_{F_{1}\cap F_{2}}$ \cite[Lemma 8.1]{cubulation}. Thus if the cores of $\beta_{F_1}$ and $\beta_{F_2}$ have nontrivial intersection, then $\beta_{F_1}=\beta_{F_2}$. In particular, different standard flats in $Y(\Ga)$ are disjoint.
\item We can glue all the $f_{v,\ell}$'s together to obtain an injective map $f:X(\Ga)^{(0)}\to Y(\Ga)$ such that $f$ map the vertex set of each standard flat bijectively to the tips of a standard branched flat. The image of $f$ is exactly the collection of $0$-dimensional standard flats in $Y(\Ga)$.
\end{enumerate}
$Y(\Ga)$ is called the \textit{blow-up building of type $\Ga$} and it is a $CAT(0)$ cube complex (\cite{cubulation}). We claim for each vertex $x\in Y(\Ga)$, there is a unique standard flat containing $x$. Then this induces a map $\L_{1}:Y(\Ga)^{(0)}\to\P$ (recall that $\P$ is the poset of standard flats in $X(\Ga)$). The claim is clear when $\Ga$ is a clique (in this case $Y(\Ga)$ is a product of branched lines). In general we can choose a standard branched flat containing $x$ and find a standard flat inside which contains $x$. Since different standard flats in $Y(\Ga)$ are disjoint, thus the uniqueness follows. We label each vertex $x\in Y(\Ga)$ by the support of $\L_1(x)$. The \textit{rank} of $x$ is the number of vertices in this clique.
\end{definition}

We record the following simple observation.
\begin{lem}
\label{label and core}
Pick standard flat $F\subset X(\Ga)$, then for any vertex $v$ in the core of $\beta_F$, we have $\L_1(v)=F$. Conversely, if a vertex $v\in Y(\Ga)$ satisfies $\L_1(v)=F$, then $v$ is in the core of $\beta_{F}$.
\end{lem}

\begin{definition}[edge labelling and hyperplane labelling of $Y(\Ga)$]
\label{edge and hyperplane labelling}
We label each edge of $Y(\Ga)$ by a vertex of $\Ga$ as follows. First we define a map $\L_2$ from the collection of edges of $Y(\Ga)$ to vertices of $\P(\Ga)$. Pick edge $e\subset Y(\Ga)$, then $e$ is contained in a standard branched flat $\beta_F$. There is a unique product factor $\beta_v$ of $\beta_F$ such that $e$ is parallel to some edge in $\beta_v$. We define $\L_2(e)=v$. There may be several $\beta_F$'s that contain $e$, however they will give rise to the same $v$. $\L_2$ actually induces an edge-labelling of $Y(\Ga)$ by vertices of $\Ga$, since vertices of $\P(\Ga)$ are labelled by the vertices of $\Ga$ (see Section \ref{Salvetti complex}).

Since every cube of $Y(\Ga)$ is contained in some $\beta_F$, the opposite edges of a 2-cube are mapped to the same vertex under $\L_2$. Thus $\L_2(e_1)=\L_2(e_{2})$ if $e_{1}$ is parallel to $e_{2}$. This induces labelling of hyperplanes in $Y(\Ga)$ by vertices of $\Ga$. Note that if two hyperplanes in $Y(\Ga)$ cross, then their labels are adjacent in $\Ga$.
\end{definition}

It follows from Lemma \ref{convexity of components} that each $\Ga'$-component in $Y(\Ga)$ or $|\B|$ is a convex subcomplex. 

\begin{example}
\label{label of products}
Let $\Ga$ be a clique and $\{v_{i}\}_{i=1}^{n}$ be its vertex set. Then we can identify vertices of $\P(\Ga)$ with vertices of $\Ga$. In this case, $Y(\Ga)\cong\prod_{i=1}^{n}\beta_{v_i}$ where $\beta_{v_i}$ is a branched line. Let $p_i:Y(\Ga)\to \beta_{v_i}$ be the projection map. For vertex $x\in Y(\Ga)$, then $v_i$ is in the label of $x$ if and only if $p_i(x)$ is in the core of $\beta_{v_i}$. For edge $e\subset Y(\Ga)$, the label of $e$ is the unique $v_i$ such that $p_i(e)$ is an edge.
\end{example}

Let $e\subset Y(\Ga)$ be an edge and $v_1,v_2$ be its endpoints. If $v_1$ and $v_2$ have the same rank, then $\L_1(v_1)=\L_1(v_2)$, hence they are labelled by the same clique in $\Ga$, and the label of $e$ belongs to this clique. Otherwise, one of $\L_1(v_1)$ and $\L_1(v_2)$ is a codimension one flat in another. Thus $v_1$ and $v_2$ are labelled by two cliques such that one is contained in another. The label of $e$ is the unique vertex in the difference of these two cliques. Again, to prove these statements, it suffices to consider the case when $\Ga$ is a clique. The general case reduces to this special case by considering a standard branched flat containing $e$. Moreover, we have the following consequence.

\begin{lem}\
\label{labelling diference}
\begin{enumerate}
\item Pick an edge $e$ in $Y(\Ga)$ $($or $|\B|)$. Then the label of $e$ is contained in the label of at least one of its endpoints.
\item Pick an edge path $\omega$ connecting vertices $x$ and $y$ in $Y(\Ga)$ $($or $|\B|)$. Then for every vertex in the symmetric difference of the labellings of $x$ and $y$, there exists an edge in $\omega$ labelled by this vertex.
\end{enumerate}
\end{lem}

\begin{definition}
\label{def_beta_K}
Pick a $\Ga'$-component $K\subset X(\Ga)$. We define $\beta_{K}=\cup_{F}\beta_F$ with $F$ ranging over standard flats in $K$. 
\end{definition}

We recall the following properties of $\beta_{K}$, where (1) follows from \cite[Lemma 5.18]{cubulation} and (2) follows from \cite[Corollary 5.16]{cubulation}.
\begin{lem}\
\label{product decomposition of standard components}
\begin{enumerate}
\item The subcomplex $\beta_{K}$ is convex in $Y(\Ga)$, and it is the convex hull of $f(K^{(0)})$ ($f$ is introduced in Definition \ref{construction}).
\item Suppose $\Ga'$ admits a join decomposition $\Ga'=\Ga'_1\circ\Ga'_2$. For $i=1,2$, pick $\Ga'_{i}$-component $K_i\subset K$ and let $\pi_i:\beta_K\to \beta_{K_i}$ be the $CAT(0)$ projection (\cite[Proposition II.2.4]{MR1744486}). Then $\pi_1\times \pi_2:\beta_{K}\to \beta_{K_1}\times\beta_{K_2}$ is a cubical isomorphism. Similarly, $|\B|_{K}\cong|\B|_{K_1}\times|\B|_{K_2}$.
\end{enumerate}
\end{lem}

\begin{lem}\
\label{parallel set in blow-up building}
Pick vertex $v\in\P(\Ga)$. Let $P_v$ be as in Definition \ref{v-parallel set}.
\begin{enumerate}
\item The subcomplex $\beta_{P_v}$ admits a natural $CAT(0)$ cubical product decomposition 
\begin{equation}\label{branched product decomposition}
\beta_{P_v}\cong \beta_v\times \beta^{\perp}_v.
\end{equation}
\item Let $e\subset Y(\Ga)$ be an edge such that $v=\L_2(e)$. Then the hyperplane $h_e$ dual to $e$ can be identified as the product factor $\beta^{\perp}_v$ in (\ref{branched product decomposition}).
\end{enumerate}
\end{lem}

\begin{proof}
Suppose $v$ is labelled by $\bar{v}\in\Ga$. Note that $P_v$ is a $St(\bar{v})$-component in $X(\Ga)$. By construction, any standard $\{\bar{v}\}$-component in $\beta_{P_v}$ can be naturally identified with $\beta_{v}$. Let $\beta_{P_v}\cong \beta_v\times \beta^{\perp}_v$ be the product decomposition induced by the join decomposition $St(\bar{v})=\{\bar{v}\}\circ lk(\bar{v})$ as in Lemma \ref{product decomposition of standard components} (2). 

To see (2), recall that if $e'$ is parallel to $e$, then $\L_2(e')=\L_2(e)$. Thus it suffices to prove if an edge $e'\subset Y(\Ga)$ satisfies $\L(e')=v$, then $e'\subset\beta_{P_v}$. To see this, note that by the definition of $\L_2$, there exists standard flat $F \subset X(\Ga)$ such that $e' \subset\beta_{F}$ and $v\in\Delta(F)$. Thus $F\subset P_v$ and $e'\subset \beta_{P_v}$.
\end{proof}

\subsection{Canonical projection and standard components}
If we collapse each standard flat in $Y(\Ga)$, then the resulting space is $|\B|$. More precisely, we define the \textit{canonical projection} $\pi:Y(\Ga)\to |\B|$ as follows. Let $\L_1:Y(\Ga)^{(0)}\to \P\cong |\B|^{(0)}$ be as in Definition \ref{construction}. We claim $\L_1$ can be extended to a cubical map, which defines $\pi$ (such extension, if exists, must be unique). When $\Ga$ is a clique, $Y(\Ga)$ is a product of branched line, then it is clear that $\L_1$ can be extended to a map which collapses the core of each factors. In general, for each standard flat $F\subset X(\Ga)$, we can extend $\L_1|_{\beta^{(0)}_F}$ to a map $\pi_{F}:\beta_F\to |\B|_F$. The uniqueness of extension implies we can piece together these $\pi_F$'s to define $\pi$. By definition, $\pi$ preserves the rank and labelling of vertices, and it induces a bijection between vertices of rank 0 in $Y(\Ga)$ and $|\B|$.

Pick an edge $e\subset Y(\Ga)$, $e$ is \textit{vertical} if $\pi(e)$ is a point, otherwise $e$ is \textit{horizontal}. In the latter case, $e$ and $\pi(e)$ have the same label. Note that $e$ is vertical if and only if its two endpoints of $e$ have the same rank. Edges parallel to a vertical (or horizontal) edge are also vertical (or horizontal), thus it makes sense to talk about vertical (or horizontal) hyperplanes. It is shown in \cite{cubulation} that $\pi$ is a restriction quotient in the sense that if one start with the wall space associated with $Y(\Ga)$ and pass to a new wall space by forgetting all the vertical hyperplanes, then the new dual cube complex is $|\B|$. The following is a consequence of Theorem 4.4, Corollary 5.15 and the discussion in the beginning of Section 5.2 of \cite{cubulation}.

\begin{thm}
\label{inverse image are flats}
For any cube $\sigma\subset|\B|$ and an interior point $x\in\sigma$, $\pi^{-1}(x)$ is isomorphic to $\E^{n}$ where $n$ is the rank of the minimal vertex in $\sigma$. Conversely, if $Y$ is any $CAT(0)$ cube complex such that there exists a cubical map $\pi:Y\to |\B|$ which satisfies the property in the previous sentence, then $Y$ can be constructed via Definition \ref{construction}. 
\end{thm}

\begin{definition}
\label{standard components}
A \textit{standard $\Ga'$-component} of $Y(\Ga)$ (or $|\B|$) is a $\Ga'$-component which contains a vertex of rank $0$. 
\end{definition}

\begin{lem}
\label{standard and label of vertices}
Let $K$ be a $\Ga'$-component in $|\B|$ $($or $Y(\Ga))$. Then the following are equivalent.
\begin{enumerate}
\item $K$ is standard.
\item There exists one vertex in $K$ whose label is contained in $\Ga'$.
\item The label of each vertex in $K$ is contained in $\Ga'$.
\end{enumerate}
\end{lem}

\begin{proof}
$(1)\Rightarrow (2)$ is trivial ($K$ contains a vertex of rank 0). $(2)\Rightarrow(3)$ follows from Lemma \ref{labelling diference} (2) and the connectedness of $K$. For $(3)\Rightarrow(1)$, pick vertex $x\in K$ labelled by $\Ga_x\subset\Ga'$ and let $\beta_{F}$ be the minimal standard branched flat containing $x$. Then $x$ is in the core of $\beta_{F}$ by minimality. Then the discussion in Example \ref{label of products} implies the label of each edge of $\beta_{F}$ is in $\Ga_x$. Thus $\beta_{F}\subset K$ and in particular $K$ contains a rank 0 vertex.
\end{proof}

\begin{lem}
\label{correspondence of standard components}
The map $\pi$ induces a one to one correspondence between standard $\Ga'$-components in $|\B|$ and $Y(\Ga)$ in the following sense.
\begin{enumerate}
\item If $K\subset|\B|$ is a standard $\Ga'$-component, then $\pi^{-1}(K)$ is a standard $\Ga'$-component.
\item If $L\subset Y(\Ga)$ is a standard $\Ga'$-component, then $\pi(L)$ is a standard $\Ga'$-component.
\end{enumerate}
\end{lem}

\begin{proof}
We prove (1). It follows from \cite[Theorem 4.4]{cubulation} that $\pi^{-1}(K)$ is convex, in particular connected. Now we verify Definition \ref{components} (2). Suppose there exists an edge in $\pi^{-1}(K)$ labelled by $v\notin\Ga'$. Then there is an endpoint of this edge labelled by a clique $C\subset\Ga$ such that $v\in C$. Let $x\in K$ be the $\pi$-image of this endpoint. Then $x$ is also labelled by $C$. Pick a rank 0 vertex $y\in K$. Then the connectedness of $K$ implies we can join $x$ and $y$ by an edge path $\omega\subset K$. Since $y$ is labelled by the empty set, Lemma \ref{labelling diference} (2) implies at least one edge in $\omega$ is labelled by $v$, which is a contradiction. The second part of Definition \ref{components} (2) is clear. The maximality of $\pi^{-1}(K)$ follows from the maximality of $K$.

To see (2), note that each edge of $\pi(L)$ is labelled by a vertex in $\Ga'$, and $\pi(L)$ contains a rank 0 vertex. Let $K'$ be the standard $\Ga'$-component containing $\pi(L)$. By (1), $\pi^{-1}(K')$ is a standard $\Ga'$-component containing $L$, thus $L=\pi^{-1}(K')$ and $K'=\pi(L)$.
\end{proof}

\begin{lem}
\label{characterization of standard components}
For each $\Ga'$-component $K\subset X(\Ga)$, $|\B|_{K}$ and $\beta_K$ are standard $\Ga'$-components. Conversely, each standard $\Ga'$-component in $|\B|$ or $Y(\Ga)$ must arise in such a way. In particular, each vertex of rank 0 is contained in some standard $\Ga'$-component.
\end{lem}

\begin{proof}
We prove the above statement in $Y(\Ga)$, the case of $|\B|$ is similar. The definition of $\beta_{K}$ implies each edge of $\beta_{K}$ is labelled by a vertex in $\Ga'$, thus there exists a standard $\Ga'$-component $L\subset Y(\Ga)$ containing $\beta_{K}$. Let $\omega\subset L$ be an edge path from a vertex in $\beta_{F}$ to an arbitrary vertex in $L$ and let $\{v_{i}\}_{i=1}^{n}$ be consecutive vertices in $\omega$. Let $F_i=\L_1(v_i)$ (see Definition \ref{construction} for the map $\L_1$). Then $F_{i}\cap F_{i+1}\neq\emptyset$ and the support of each $F_i$ is in $\Ga'$ (by Lemma \ref{standard and label of vertices}). So all $F_i$'s are in the same $\Ga'$-component of $X(\Ga)$. Thus for each vertex $v\in L$, $F_v=\L_1(v)\subset K$. It follows from Lemma \ref{label and core} that $v\in\beta_{F_v}\subset\beta_K$. Thus $\beta_{K}=L$. To see the converse, pick a rank 0 vertex $x$ in an arbitrary standard $\Ga'$-component $L'$. Let $K'\subset X(\Ga)$ be the $\Ga'$-component containing $\L_1(x)$. Then $\beta_{K'}\cap L'\neq\emptyset$. Since $\beta_{K'}$ and $L'$ are both standard $\Ga'$-components, they are the same complex.
\end{proof}

\begin{lem}
\label{intersection of standard components}
For $i=1,2$, let $L_i$ be a standard $\Ga_i$-component in $|\B|$ $($or $Y(\Ga))$ and let $K_i\subset X(\Ga)$ be the associated $\Ga_i$-component. Then $L_1\cap L_2\neq\emptyset$ if and only if $K_1\cap K_2\neq\emptyset$. In this case, $L_1\cap L_2$ is a standard $\Ga_1\cap\Ga_2$-component.
\end{lem}

\begin{proof}
We prove the lemma in $Y(\Ga)$, the case of $|\B|$ is similar. It is clear from the definitions that $K_1\cap K_2\neq\emptyset$ implies $L_1\cap L_2\neq\emptyset$. To see the conserve, pick vertex $x\in L_1\cap L_2$, then the argument in Lemma \ref{standard and label of vertices} implies that the minimal standard branched flat containing $x$ is inside both $L_1$ and $L_2$. In particular, $L_1\cap L_2$ contains a vertex of rank 0, which implies $K_1\cap K_2\neq\emptyset$. By Lemma \ref{characterization of standard components}, there exists a standard $\Ga_1\cap \Ga_2$-component $N$ containing this rank 0 vertex. It is clear that $N\subset L_1\cap L_2$ and $L_1\cap L_2\subset N$.
\end{proof}

Let $L$ be a (not necessarily standard) $\Ga_L$-component in $Y(\Ga)$ and let $x\in L$ be a vertex. Suppose $\Ga^{\perp}_{L}$ is the clique spanned by vertices in $\Ga_x\setminus \Ga_L$, where $\Ga_x$ is the label of $x$. Then Lemma \ref{labelling diference} (2) implies $\Ga^{\perp}_{L}$ is contained in the label of each vertex of $L$, and Lemma \ref{labelling diference} (1) implies $\Ga$ contains the graph join $\Ga_L\circ \Ga^{\perp}_{L}$. 

The proof of Lemma \ref{standard and label of vertices} implies $x$ is connected to a rank $0$ vertex $x_0$ via an edge-path such that the label of each edge in this path is contained in $\Ga^{\perp}_{L}$. Let $M$ be the $\Ga_L\circ \Ga^{\perp}_{L}$-component containing $x_0$ and let $M=N\times N^{\perp}$ be the product decomposition in Lemma \ref{product decomposition of standard components} (2). Note that $L\subset M$ and $L$ has trivial projection to $N^{\perp}$. Thus $L$ must be of form $\{y\}\times N$ for some vertex $y\in N^{\perp}$, and $L$ is standard if and only if $y$ is of rank 0. We can characterize standard $\Ga'$-components in $|\B|$ in a similar way. 

\begin{lem}\
\label{Ga'-component}
\begin{enumerate}
\item Any $\Ga'$-component in $Y(\Ga)$ or $|\B|$ is parallel to (hence isomorphic to) a standard $\Ga'$-component.
\item If $\Ga'$ is a subgraph satisfying $\Ga'\nsubseteq lk(v)$ for any vertex $v\in\Ga\setminus\Ga'$, then each $\Ga'$-component is standard.
\item If $e_1$ and $e_2$ are two edges in $Y(\Ga)$ or $|B|$ emanating from the same vertex, then they span a square if and only if their labels are adjacent in $\Ga$.
\end{enumerate}
\end{lem}

\begin{proof}
It remains to prove the if direction of (3). Let $L$ be the $\Ga'$-component containing $e_1$ and $e_2$ where $\Ga'$ is the edge spanned by the label of $e_1$ and $e_2$. Then (1) implies $L$ is a product of two branched lines in the case of $Y(\Ga)$, and is a product of two infinite trees of diameter 2 in the case of $|\B|$, thus (3) follows. Alternatively, one can also prove (3) directly by looking at the link of vertices, see \cite[Section 5.4]{cubulation}.
\end{proof}



\section{A geometric model for groups q.i. to RAAGs}
\label{sec_geometric model}
In \cite{cubulation}, it was shown that if a group $H$ is quasi-isometric to $G(\Ga)$ such that $\out(G(\Ga))$ is finite, then $H$ acts geometrically on a $CAT(0)$ cube complex which is very similar to $X(\Ga)$. In this section, we briefly recall the construction in \cite{cubulation}, and prove several further properties of this construction.

\subsection{Quasi-action on RAAG's}
\label{subsec_quasi action}
We pick an identification between $G(\Ga)$ and the 0-skeleton $X^{(0)}(\Ga)$ of $X(\Ga)$. It follows from \cite{huang2014quasi} that quasi-isometries between RAAG's with finite outer-automorphism group are at bounded distance from a canonical representative that respects standard flats, in the following sense.

\begin{definition}
\label{def_flat_preserving}
A quasi-isometry $\phi:X^{(0)}(\Ga)\ra X^{(0)}(\Ga)$ is {\em flat-preserving} if it is a bijection and for every standard flat $F\subset X(\Ga)$ there is a standard flat $F'\subset X(\Ga)$ such that $\phi$ maps the $0$-skeleton of $F$ bijectively onto the $0$-skeleton of $F'$. The standard flat $F'$ is uniquely determined, and we denote it by $\phi_*(F)$.
\end{definition}

\begin{thm}[Vertex rigidity \cite{huang2014quasi,bks2}]
\label{thm_intro_vertex_rigidity}
Suppose $\out(G(\Ga))$ is finite and $G(\Ga)\not\simeq\Z$. Let $\phi:X^{(0)}(\Ga)\ra X^{(0)}(\Ga)$ be an $(L,A)$-quasi-isometry. Then there is a unique flat-preserving quasi-isometry $\bar\phi:X^{(0)}(\Ga)\ra X^{(0)}(\Ga)$ at finite distance from $\phi$, and moreover
$$
d(\bar\phi,\phi)=\sup\{v\in X^{(0)}(\Ga)\mid d(\bar\phi(v),\phi(v))\}
<D=D(L,A)\,.
$$
\end{thm}

Given any quasi-action $H\stackrel{\rho}{\acts}X^{(0)}(\Ga)$, we may apply Theorem \ref{thm_intro_vertex_rigidity} to the associated family of quasi-isometries $\{\rho(h):X^{(0)}(\Ga)\ra X^{(0)}(\Ga)\}_{h\in H}$ to obtain a new quasi-action $H\stackrel{\bar\rho}{\acts}X^{(0)}(\Ga)$ by flat-preserving quasi-isometries. Due to the uniqueness assertion in Theorem \ref{thm_intro_vertex_rigidity}, the quasi-action $\bar\rho$ is actually an action, i.e. the composition law $\bar\rho(hh')=\bar\rho(h)\circ\bar\rho(h')$ holds exactly, not just up to bounded error.  

Pick vertex $v\in\P(\Ga)$ and let $P_{v}$ be the $v$-parallel set (Definition \ref{v-parallel set}). Note that there is a cubical product decomposition $P_{v}\simeq \R_v\times Q_v$, where $\R_v$ is parallel to some standard geodesic $\ell$ with $\Delta(\ell)=v$. Likewise, there is a product decomposition of the $0$-skeleton 
\begin{equation}
\label{eqn_0_skeleton_product_decomposition}
P_{v}^{(0)}
\simeq \Z_{v}\times Q_{v}^{(0)}
\end{equation}
where $\Z_{v}$ is the $0$-skeleton of $\R_{v}$ equipped with the induced metric.

Now let $\rho:H\acts X^{(0)}$ be an action by flat-preserving $(L,A)$-quasi-isometries. It follows readily from Definition \ref{def_flat_preserving} that the action respects parallelism of standard geodesics and standard flats, and this implies there is an induced action of $H$ on $\P(\Ga)$. For each vertex $v\in\P(\Ga)$, let $H_v\le H$ be the stabilizer of $v$. The action $H_{v}\acts P_{v}^{(0)}$ preserves the product decomposition (\ref{eqn_0_skeleton_product_decomposition}), and therefore gives rise to a factor action $\rho_{v}:H_{v}\acts \Z_{v}$ by $(L,A)$-quasi-isometries.

\begin{thm}\label{conjugate to left translation}
\cite[Corollary 6.17]{cubulation} Let $\rho:H\acts G(\Gamma)$ be an action by flat-preserving bijections. Suppose 
\begin{enumerate}
\item the induced action $H\acts\P(\Ga)$ is preserves the label of vertices in $\P(\Ga)$;
\item for each vertex $v\in\mathcal{P}(\Gamma)$, the action $\rho_{v}:H_{v}\acts \Z$ is conjugate to an action by translations.
\end{enumerate}
Then the action $\rho$ is conjugate to an action $H\acts G(\Gamma)$ by left translations via a flat-preserving bijection.
\end{thm}

Now we impose a condition which guarantees assumption (1) of Theorem \ref{conjugate to left translation}.
\begin{definition}
A finite simplicial graph $\Gamma$ is \textit{star-rigid} if for any automorphism $\alpha$ of $\Gamma$, $\alpha$ is identity whenever it fixes a closed star of some vertex $v\in\Gamma$ point-wise.
\end{definition}

\begin{lem}
\label{label-preserving}
Suppose $\Ga$ is star-rigid, and $\out(G(\Ga))$ is finite. If $H$ acts on $\P(\Ga)$ by simplicial isomorphisms, then there exists a finite index subgroup $H'\le H$ such that the induced action $H'\acts\P(\Ga)$ is label-preserving. 
\end{lem}

\begin{proof}
Since $\out(G(\Ga))$ is finite, there is an induced action $H\acts G(\Ga)$ by flat-preserving maps. Let $F(\Ga)$ be the flag complex of $\Ga$. We define a homomorphism $\alpha:H\to \aut(F(\Ga))$ as follows.

Each vertex $x\in X(\Ga)$ gives rise to a simplicial embedding $i_{x}:F(\Ga)\to\P(\Ga)$ by considering the collection of standard flats passing through $x$. Pick $h\in H$ and vertex $x\in X(\Ga)$, we define $\alpha(h,x)=i^{-1}_{h(x)}\circ h\circ i_{x}$. We claim $\alpha(h,x)$ does not depend on the choice of $x$. It suffices to show $\alpha(h,x')=\alpha(h,y')$ for any two vertices $x'$ and $y'$ in the same standard geodesic $\ell$. Let $v=i^{-1}_{x'}(\Delta(\ell))$. Then $\alpha(h,x')$ and $\alpha(h,y')$ coincident on the closed star of $v$, hence they are equal. Now there is a well-defined map $\alpha:H\to \aut(F(\Ga))$. One readily verify that $\alpha$ is a group homomorphism and we take $H'=\ker(\alpha)$. 
\end{proof}

Assumption (2) of Theorem \ref{conjugate to left translation} is not true in general, see Proposition \ref{prop_intro_semiconjugacy}. However, the following weaker version holds under additional conditions. The proof is postponed to the next section.

\begin{lem}\label{conjugate to action by translations}
Let $\rho:H\acts G(\Gamma)$ be an action by flat-preserving bijections which are $(L,A)$-quasi-isometries. Suppose the action is cobounded and discrete. Then for each vertex $v\in\P(\Ga)$, there exists a subgroup $H'_{v}\le H_{v}$ of finite index such that the action $\rho_{v}:H_v\acts\Z_v$ restricted on $H'_{v}$ is conjugate to an action by translations. 
\end{lem}

The following is the key to understand the factor actions $\rho_{v}:H_v\acts\Z_v$.

\begin{prop}
\label{prop_intro_semiconjugacy}
\cite[Proposition 6.2]{cubulation}
If $U\stackrel{\rho_0}{\acts} \Z$ is an action by $(L,A)$-quasi-isometries, 
then there exist a branched line $\beta$ without valence 2 vertices, an action $U\stackrel{\rho_1}{\acts}\beta$ by simplicial isomorphisms and a bijective equivariant $(L',A')$-quasi-isometry (here $L',A'$ depends only on $L$ and $A$)
$$
U\stackrel{\rho_0}{\acts} \Z\lra
U\stackrel{\rho_1}{\acts}t(\beta).
$$
\end{prop}

Note that the valence of each vertex in $\beta$ is uniformly bounded above by some constant depending only on $L$ and $A$.

\begin{definition} (an equivariant construction, \cite[Section 5.6 and Section 6.1]{cubulation})
\label{an equivariant construction}
Let $\rho:H\acts G(\Ga)$ be an action by flat-preserving bijections which are also $(L,A)$-quasi-isometries. We will use Definition \ref{construction} to construct an isometric action $\rho':H\acts Y(\Ga)$ from $\rho$. First we need to choose the blow-up data $f_{v,\ell}$'s which are compatible with the $H$-action. Actually, it suffices to associate a branched line $\beta_v$ and a map $f_v:\Z_v\to t(\beta_v)$ to each vertex $v\in\P(\Ga)$, since for each standard geodesic $\ell$ such that $\Delta(\ell)=v$, we can identify vertices of $\ell$ with $\Z_v$ via (\ref{eqn_0_skeleton_product_decomposition}), which induces the map $f_{v,\ell}$. To define $f_v$'s, we start with the factor actions $\rho_{v}:H_{v}\acts \Z_{v}$. Note that there exists $L_1,A_1>0$ such that each $\rho_v$ is an action by $(L_1,A_1)$-quasi-isometries. 

Consider the induced action $H\acts \P(\Ga)$ and pick a representative from each $H$-orbit of vertices of $\P(\Ga)$. The resulting set is denoted by $V$. For each $v\in V$, by Proposition \ref{prop_intro_semiconjugacy}, there exist a branched line $\beta_v$, an isometric action $H_v\acts \beta_v$ and an $H_v$-equivariant bijection $f_v: H_v\acts\Z_v\lra H_v\acts t(\beta_v)$. Moreover, the valence of vertices in each $\beta_v$ is uniformly bounded from above in terms of $L_1$ and $A_1$. For vertex $v\notin V$, we pick an element $h\in H$ such that $h(v)\in V$. Note that $h$ induces a bijection $h':\Z_v\to \Z_{h(v)}$. We define $\beta_v=\beta_{h(v)}$ and $f_{v}=f_{h(v)}\circ h':\Z_v\to t(\beta_v)$.

Let $Y(\Ga)$ be the associated blow-up building and let $f:G(\Ga)\to Y(\Ga)$ be the map as in Definition \ref{construction}. The following are true (see \cite[Theorem 6.4]{cubulation}).
\begin{enumerate}
\item The complex $Y(\Ga)$ is uniformly locally finite (this essentially follows from the fact the complexity of all $\beta_v$'s are uniformly bounded from above).
\item The bijections $f_{v,\ell}$'s and the action $\rho$ induce an action $\rho':H\acts Y(\Ga)$ by cubical isomorphisms. For each element $h\in H$, $\rho'(h)$ maps standard branched flats to standard branched flats. Hence $\rho'(h)$ also preserves the collection of $\beta_{P_v}$'s (see Lemma \ref{parallel set in blow-up building} and Definition \ref{def_beta_K} for $\beta_{P_v}$).
\item The map $f$ is an $H$-equivariant quasi-isometry.
\item If $\rho$ is discrete and cobounded, then $\rho'$ is geometric, i.e. the action of $\rho'$ on $Y(\Ga)$ is proper and cocompact.
\end{enumerate}
\end{definition}

\subsection{Stabilizer of parallel sets} In this section we restrict ourselves to the case when $\rho:H\acts G(\Ga)$ is a proper and cobounded action by $(L,A)$-quasi-isometries which are also flat-preserving bijections. In this case, the action $\rho':H\acts Y(\Ga)$ constructed in Definition \ref{an equivariant construction} is geometric. Since the collection $\{\beta_{P_v}\}_{v\in \textmd{Vertex}(\P(\Ga))}$ is $H$-invariant and locally finite, the action $\stab_{H}(\beta_{P_v})\acts \beta_{P_v}$ is cocompact. However, $H_v\le \stab_{H}(\beta_{P_v})$ is of finite index. Thus the action $H_v\acts \beta_{P_v}$ is geometric. Note that $H_v\acts \beta_{P_v}$ respects the product decomposition $\beta_{P_v}\simeq \beta_v\times \beta^{\perp}_v$ (Lemma \ref{parallel set in blow-up building}), hence there is an induced factor action $H_v\acts \beta_v$, which is cobounded. If $H_v$ contains element that flips the two ends of $\beta_v$, then $\beta_v/H_v$ is a tree. Otherwise $H_v$ acts by translations on the core of $\beta_v$, which gives rise to a homomorphism $\phi:H_v\to \Z$ by associating each element in $H_v$ its translation length. Let $K_v$ be the kernel of $\phi$. Then we have a splitting exact sequence $1\to K_v\to H_v\to \Z\to 1$. Let $\alpha\in H_v$ be a lift of a generator of $\Z$.

\begin{lem}
\label{subgroup with no twist}
Suppose $\rho:H\acts G(\Ga)$ is discrete and cobounded and suppose $H_v$ does not flip the two ends of $\beta_v$. Then the conjugation map $K_v\to \alpha K_v\alpha^{-1}$ gives rise to an element of finite order in $\out(K_v)$. Thus there exists integer $n>0$ such that $\phi^{-1}(n\Z)$ is isomorphic to $K_v\oplus n\Z$. Moreover, up to passing to a possibly larger $n$, we can choose the $n\Z$ factor such that it acts trivially on $\beta^{\perp}_v$.
\end{lem}

\begin{proof}
Since $H_v$ preserves the product structure of $\beta_{P_v}$, we have a homomorphism $h:H_v\to \aut(\beta^{\perp}_v)$. Moreover, the following diagram commutes:
\begin{center}
$\begin{CD}
K_v                           @>C_{\alpha}>>        K_v\\
@VhVV                                   @VhVV\\
h(K_v)          @>C_{h(\alpha)}>>        h(K_v)
\end{CD}$
\end{center}
Here $C_{\alpha}$ and $C_{h(\alpha)}$ denote conjugation by $\alpha$ and $h(\alpha)$ respectively. Note that the action $h(K_v)\acts \beta^{\perp}_v$ is geometric. We apply Lemma \ref{normal subgroup} below to the vertex set of $\beta^{\perp}_v$ to deduce that $h(K_v)$ is of finite index in $h(H_v)$. Thus $C_{h(\alpha)}$ gives rise to a finite order element in $\out(h(K_v))$. Since $h:K_v\to h(K_v)$ has finite kernel, $C_{\alpha}$ is of finite order in $\out(K_v)$. Then we can find integer $n>0$ such that $\phi^{-1}(n\Z)=K_v\oplus\langle n\alpha\rangle$. Apply Lemma~\ref{normal subgroup} to the vertex set of $\beta^{\perp}_v$ again, we know there exists integer $m>0$ such that $h(mn\alpha)\in h(K_v)$. Suppose $h(mn\alpha)=h(\beta)$ for $\beta\in K_v$. Then $\phi^{-1}(mn\Z)=K_v\oplus\langle (mn\alpha)\beta^{-1}\rangle$, where $(mn\alpha)\beta^{-1}$ acts trivially on $\beta^{\perp}_v$.
\end{proof}

\begin{lem}[B. Kleiner]\label{normal subgroup}
Suppose $Z$ is a metric space such that every $r$-ball contains at most $N=N(r)$ elements. Assume that $A \acts Z$ is a faithful action by $(L,A)$-quasi-isometries, and $B$ is a normal subgroup of $A$ that acts discretely cocompactly on $Z$. Then $B$ is of finite index in $A$.
\end{lem}

\begin{proof}
Suppose $B$ is of infinite index in $A$. Let $\{\alpha_i\}_{i\in I}\subset A$ be a collection of elements such that different $\alpha_i$'s are in different $B$-cosets. Pick a base point $p\in Z$. Since the $B$-action is cocompact, we can assume there exists $D>0$ such that $d(\alpha_i(p),p)<D$ for all $i\in I$.

Note that $B$ is finitely generated from our assumption. Let $\{b_{\lambda}\}_{\lambda\in\Lambda}$ be a finite generating for $B$. For each $b_{\lambda}$, there exists $D'>0$ such that $d(\alpha_ib_{\lambda}\alpha^{-1}_i(p),p)<D'$ for all $i\in I$. Since $\alpha_ib_{\lambda}\alpha^{-1}_i\in B$, by the discreteness of $B$-action, we can assume without loss of generality that $\alpha_ib_{\lambda}\alpha^{-1}_i=\alpha_jb_{\lambda}\alpha^{-1}_j$ for any $i\neq j$. Since there are only finitely many $b_{\lambda}$'s, we can also assume the equality holds for all $b_{\lambda}$'s. Thus $\alpha_ib\alpha^{-1}_i=\alpha_jb\alpha^{-1}_j$ for any $b\in B$ and $i,j\in I$. We can further assume without loss of generality that $\alpha_ib\alpha^{-1}_i=b$ for any $i\in I$ and $b\in B$.

Since $\alpha_i$ commutes with each element of $B$, by the cocompactness of $B$-action, there exists $R>0$ such that $\alpha_i$ is completely determined by its behavior on the $R$-ball centred at $p$. However, $d(\alpha_i(p),p)<D$ for all $i\in I$, so there are only finitely many ways to define $\alpha_i$. Thus there exists a pair $i\neq j$ such that $\alpha_i=\alpha_j$, which yields a contradiction.
\end{proof}

\begin{proof}[Proof of Lemma \ref{conjugate to action by translations}]
Up to passing to a subgroup of index 2, we assume $H_v$ does not flip the ends of $\beta_v$. Let $h:H_v\to\isom(\beta_v)$ be the homomorphism induced by the action $H_v\acts\beta_v$ and let $\bar{H}_v=\textmd{Im\ }h$. We claim the action $\bar{H}_{v}\acts\beta_v$ is geometric. Assuming the claim, we deduce that $\bar{H}_v$ has a finite index subgroup $A$ isomorphic to $\Z$. Define $H'_{v}=h^{-1}(A)$ and the lemma follows.

It suffices to show the subgroup of $\bar{H}_v$ which stabilizes a tip of $\beta_v$ is finite. Pick $x\in t(\beta_v)$. Then the subgroup $K\le H_v$ which stabilizes $\{x\}\times \beta^{\perp}$ acts geometrically on itself, thus $K$ is finitely generated. Since $H_v$ does not flip the ends of $\beta_v$, $K$ acts trivially on the core of $\beta_v$. Thus $h(K)$ is a finitely generated subgroup of $\prod_{i\in I}G_{i}$, here each $G_{i}$ is the permutation group of $n_{i}$ points and there is a uniform upper bound for all the $n_{i}$'s. Hence $h(K)$ is finite.
\end{proof}

\subsection{Special cube complexes and blow-up buildings}
\label{subsec_special and commensurability}
Suppose $\Ga$ is star-rigid and $\out(G(\Ga))$ is finite. Let $H$ be a finitely generated group quasi-isometric to $G(\Ga)$ and let $\rho:H\acts G(\Ga)$ be the associated action by flat-preserving maps. By Lemma \ref{label-preserving}, we can assume the induced action $H\acts\P(\Ga)$ preserves the labelling of vertices (up to passing to a finite index subgroup). Then Definition \ref{an equivariant construction} gives rise to a geometric action of $H$ on a blow-up building $Y(\Ga)$, moreover, this action preserves the labelling of edges.

\begin{lem}
\label{orientation and specialness}
Suppose there exists a torsion free subgroup $H'\le H$ of finite index. Then the following are equivalent:
\begin{enumerate}
\item For each vertex $v\in\mathcal{P}(\Gamma)$, the restricted action $\rho_{v}:H'_{v}\acts \Z_v$ is conjugate to an action by translations ($H'_v\le H'$ is the stabilizer of $v$).
\item $Y(\Ga)/H'$ is a special cube complex.
\end{enumerate}
If either (1) or (2) is true, then $H'$ is isomorphic to a finite index subgroup of $G(\Ga)$, hence $H$ is commensurable to $G(\Ga)$.
\end{lem}

\begin{proof}
(1)$\Rightarrow$(2): Let $H'_{v}\acts\beta_v$ be the restricted action on the associated branched line, then (1) implies the image of $H'_v\to \isom(\beta_v)$ is isomorphic to $\Z$. Thus there exists an $H'_v$-invariant orientation of edges such that for any two adjacent edges of $\beta_v$ which are in the same $H'_v$ orbit, the orientation on them avoid following two configurations:
\begin{center}
\includegraphics[scale=0.5]{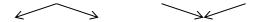}
\end{center}
For each $H'$-orbit of vertices in $\P(\Ga)$, we pick a representative $v$ and assign an $H'_v$-equivariant orientation of edges in $\beta_v$ which satisfies the above condition. This induces orientation of edges in $\beta_{P_v}$ which are parallel to the $\beta_v$ factor (see (\ref{branched product decomposition})). The $H'$-orbits of these oriented edges cover the 1-skeleton of $Y(\Ga)$ (see (2) of Lemma \ref{parallel set in blow-up building}), thus we have an $H'$-invariant edge orientation of $Y(\Ga)$. This orientation respects parallelism of edges, thus each hyperplane of $Y(\Ga)/H'$ is 2-sided. Since $H'$ preserves the edge-labelling of $Y(\Ga)$, hyperplanes of $Y(\Ga)/H'$ do not self-intersect. Inter-osculation is ruled out by Lemma \ref{Ga'-component} (3). The above condition for orientation of edges of $\beta_v$ implies hyperplanes of $Y(\Ga)/H'$ do not directly self-osculate.

(2)$\Rightarrow$(1): Let $\bar{H}'_v$ be the image of $H'_v\to \isom(\beta_v)$. It suffices to show $\bar{H}'_v$ is isomorphic to $\Z$. The proof of Lemma \ref{conjugate to action by translations} implies $\bar{H}'_v$ has a finite index subgroup isomorphic to $\Z$. Thus it suffices to show $\bar{H}'_v$ is torsion-free. Suppose the contrary is true, either there is an element of $\bar{H}'_v$ that flips the two endpoints of an edge in $\beta_v$, which gives rise to a 1-sided hyperplane in $Y(\Ga)/H'$, or there exists an element of $\bar{H}'_v$ that permutes edges in the closed star of some vertex of $\beta_v$ in a non-trivial way, which gives rise to a directly self-osculating hyperplane in $Y(\Ga)/H'$.

The last statement of the lemma follows from Theorem \ref{conjugate to left translation}.
\end{proof}

\section{Branched complexes with trivial holonomy}
\label{sec_branched complexes with trivial holonomy}
In this section, we look at more properties of the quotient of a blow-up building by a proper and cocompact group action. In particular, we introduce a notion of trivial holonomy for such quotient, under which condition one can collapse tori in the quotient to obtain another special cube complex.

\subsection{The branched complex}
\label{subsec_branched complex basics}
Let $Y(\Ga)$ be a blow-up building as in Definition \ref{construction} and let $|\B|$ be the Davis realization of the right-angled building associated with $G(\Ga)$. In the light of Section \ref{subsec_special and commensurability}, we assume $H$ acts geometrically on $Y(\Ga)$ by label-preserving cubical isomorphisms. A \textit{branched complex of type $\Ga$} is the orbifold $K(\Ga)=Y(\Ga)/H$. It is \textit{torsion free} if $H$ is torsion free (in this case the action $H\acts Y(\Ga)$ is free), and it is \textit{special} if $H$ is torsion free and $K(\Ga)$ is a special cube complex. If $H$ is torsion free, and $\Ga$ is a clique or a point, then the corresponding branched complex is also called a \textit{branched torus} or a \textit{branched circle}. The \textit{core} of a branched torus is the quotient of the core of $Y(\Ga)$.

\begin{lem}
\label{rank and label preserving}
$H$ preserves the rank and label of vertices in $Y(\Ga)$. 
\end{lem}

\begin{proof}
The discussion before Lemma \ref{labelling diference} implies for any edge $e\subset Y(\Ga)$ with its endpoints denoted by $v_1$ and $v_2$, one can deduce the label of $v_2$ from the label of $v_1$, label of $e$, and the rank of $v_1$ and $v_2$. Thus it suffices to prove $H$ preserves the rank of vertices in $Y(\Ga)$.

Let $\Ga=\Ga_1\circ\Ga_2\circ\cdots\circ\Ga_n$ be the join decomposition of $\Ga$ and pick a standard $\Ga_{i}$-component $Y_i$ in $Y(\Ga)$. Then Lemma \ref{product decomposition of standard components} (2) implies that $\prod_{i=1}^{n}p_{i}:Y(\Ga)\to\prod_{i=1}^{n} Y_i$ is a cubical isomorphism, where $p_i:Y(\Ga)\to Y_i$ is the $CAT(0)$ projection. For each $i$, all $\Ga_i$-components are parallel to each other (see the discussion before Lemma \ref{Ga'-component}), and $H$ permutes these $\Ga_i$-components (since $H$ is label-preserving). Thus $H$ respects the product decomposition $Y(\Ga)\cong\prod_{i=1}^{n} Y_i$ and induce trivial permutation of the factors. Note that the label of a vertex $y\in Y(\Ga)$ is the clique spanned by the union of labels of $p_{i}(y)$ for $1\le i\le n$. Thus it suffices to prove the lemma when $n=1$. If $\Ga$ is discrete, then $Y(\Ga)$ is a tree. The label-preserving condition implies $H$ preserves rank 0 vertices. If $\Ga$ is not discrete, then \cite[Theorem 5.24]{cubulation} implies $H$ preserves the rank of all vertices in $Y(\Ga)$.
\end{proof}

It follows from the above lemma that the action $H\acts Y(\Ga)$ descends to an action $H\acts|\B|$ through the canonical projection $\pi:Y(\Ga)\to |\B|$. Moreover, the action $H\acts|\B|$ preserves the labelling and rank of vertices.

\begin{lem}
\label{finite index subgroup without inversion}
There is a finite index subgroup $H'\le H$ such that $H'\acts Y(\Ga)$ is an action without inversion in the sense that if $h\in H'$ fixes a cube in $Y(\Ga)$, then it fixes the cube pointwise.
\end{lem}

\begin{proof}
Pick vertex $u\in\Ga$. Let $\h_u$ be the collection of hyperplanes in $Y(\Ga)$ that are labelled by $u$ (see Definition~\ref{edge and hyperplane labelling}). Then distinct elements in $\h_u$ have empty intersection. It follows that the dual cube complex to the wall space $(Y(\Ga),\h_u)$ is a tree, which we denote by $T$. In other words, $T$ has a vertex for each connected component of $Y(\Ga)\setminus \h_u$, and an edge when one travels from one component to another by crossing a hyperplane.

Since $H$ is label-preserving, it permutes elements in $\h_u$. Thus there is an induced action $H\acts T$. Up to passing to an index 2 subgroup, we assume $H$ acts on $T$ without inversion. Thus $H\acts Y(\Ga)$ does not flip any edge in $Y(\Ga)$ labelled by $u$. By repeating the above argument for each vertex in $\Ga$, there exists a finite index subgroup $H'\le H$ which does not flip any edge in $Y(\Ga)$. Since $H'$ is label-preserving, it acts on $Y(\Ga)$ without inversion.
\end{proof}

For comparison, we look at the action of $H\acts|\B|$. It preserves rank of vertices and label of edges, hence it is already an action without inversion.

From now on, we assume $H$ acts on $Y(\Ga)$ without inversion. Then the cube complex structure on $Y(\Ga)$ descends to $K(\Ga)$, and there is a well-defined labelling of edges of the $Y(\Ga)/H$ by vertices of $\Ga$. Moreover, each vertex of $K(\Ga)$ has a well-defined rank and label by Lemma \ref{rank and label preserving}. Similarly, we can define cube complex structure on $|\B|/H$ as well as labelling of vertices and edges and rank of vertices. 

As cube complexes, $K(\Ga)$ and $|\B|/H$ may not be non-positively curved. For example, let $\Ga$ be an edge, then $X(\Ga)\cong \E^2$. Suppose $H$ acts on $\E^{2}$ by translations generated by $(1,1)$ and $(1,-1)$. Then we can obtain $|\B|/H$ by taking two unit squares and identify them along two consecutive edges.

\begin{lem}
\label{inverse image of Ga'-component}
Let $L$ be a $\Ga'$-component in $|\B|/H$ and let $K$ be a connected component of $p^{-1}(L)$, where $p:|\B|\to |\B|/H$. Let $H_K\le H$ be the stabilizer of $K$. Then the natural map $i:K/H_K\to L$ is a cubical isomorphism. In particular, $p$ maps $\Ga'$-components to $\Ga'$-components. Similar statement holds for $\Ga'$-components in $K(\Ga)$.
\end{lem}

\begin{proof}
Note that each connected component of $p^{-1}(L)$ is a $\Ga'$-component. The cube complex structure of $K$ descends to $K/H_K$, and $i$ is a cubical map which maps cubes in $K/H_K$ to cubes in $L$ of the same dimension. To show $i$ is surjective, it suffices to show $p:p^{-1}(L)\to L$ is surjective. Since the action of $H$ respects cubical structure, for each point $x\in |\B|/H$ and $y\in p^{-1}(x)$, we can lift each piecewise linear path starting at $x$ to a path starting at $y$. Thus the surjectivity follows. For the injectivity of $i$, it suffices to show for any $h\in H$, if $h\cdot K\cap K\neq\emptyset$, then $h\in H_K$. This is true because $H$ is label-preserving. The $K(\Ga)$ case is similar. 
\end{proof}

We define \textit{standard $\Ga'$-components} in $K(\Ga)$ (or $|\B|/H$) to be those $\Ga'$-components which contain a rank 0 vertex. Since the canonical projection $\pi:Y(\Ga)\to |\B|$ is $H$-equivariant, it induces a cubical map $\pi_{H}:K(\Ga)\to |\B|/H$. 

\begin{remark}
\label{surjection on fundamental groups}
We define a natural map $h_{\ast}:H\to \pi_1(|\B|/H)$ as follows. Pick a base point $\star\in|\B|/H$ and one of its lifts $x\in|\B|$. For each $h\in H$, we pick a path $\omega_h$ connecting $x$ and $h\cdot x$, and map $h$ to the element in $\pi_1(|\B|/H,\star)$ represented by the image of $\omega_h$ in $|\B|/H$. Since $|\B|$ is simply connected, $h_{\ast}$ is well-defined and is a group homomorphism. Moreover, since we can lift piecewise linear paths from $|\B|/H$ to $|\B|$, $h_{\ast}$ is surjective.

When $H$ is torsion free, $h_{\ast}$ can be defined alternatively as follows. We can pick a base point $\bar{\star}\in K(\Ga)$ such that $\pi_{H}(\bar{\star})=\star$ and identify $\pi_1(K(\Ga),\bar{\star})$ with $H$ by choosing a lift of $\bar{\star}$ in $Y(\Ga)$. Then the map $(\pi_{H})_{\ast}:\pi_1(K(\Ga),\bar{\star})\to\pi_1(|\B|/H,\star)$ coincides with $h_{\ast}$ up to conjugation.
\end{remark}

\begin{lem}\
\label{correspondence between standard components downstairs}
\begin{enumerate}
\item The inverse image of a standard $\Ga'$-component in $|\B|/H$ under $\pi_H$ is a standard $\Ga'$-component.
\item The $\pi_H$-image of a standard $\Ga'$-component is a standard $\Ga'$-component.
\end{enumerate}
Thus $\pi_H$ induces a 1-1 correspondence between standard $\Ga'$-components in $K(\Ga)$ and standard $\Ga'$-components in $|\B|/H$.
\end{lem}

\begin{proof}
Note that for any edge $e\subset K(\Ga)$, either $\pi_{H}(e)$ is a point, or $\pi_{H}(e)$ is an edge whose label is the same as $e$. Moreover, $\pi_H$ induces a bijection between rank 0 vertices (since this is true for the canonical projection $\pi:Y(\Ga)\to |\B|/H$). Thus (2) follows from (1). Now we prove (1). Consider the following commuting diagram.
\begin{center}
$\begin{CD}
Y(\Ga)                        @>\pi>>       |\B|\\
@VVqV                              @VVpV\\
K(\Ga)                          @>\pi_H>>       |\B|/H
\end{CD}$
\end{center}
Pick a standard $\Ga'$-component $L\subset|\B|/H$. Then $p^{-1}(L)$ is a disjoint union of standard $\Ga'$-components. Lemma \ref{inverse image of Ga'-component} implies $H$ acts transitively on the components of $p^{-1}(L)$. By Lemma \ref{correspondence of standard components}, $\pi^{-1}(p^{-1}(L))$ is also a disjoint union of standard $\Ga'$-components, and $H$ acts transitively on these components. Thus $q\circ\pi^{-1}\circ p^{-1}(L)=\pi^{-1}_{H}(L)$ is a standard $\Ga'$-component. 
\end{proof}

The following is a consequence of Lemma \ref{intersection of standard components}.
\begin{lem}
\label{intersection of standard components downastairs}
For $1\le i\le n$, let $L_i$ be a standard $\Ga_i$-component in $K(\Ga)$ or $|\B|/H$. If $\cap_{i=1}^{n}L_i\neq\emptyset$, then it is a disjoint union of standard $\cap_{i=1}^{n}\Ga_i$-components.
\end{lem}

\subsection{Trivial holonomy and specialness}
\label{subsec_trivial holonomy and specialness}
For each vertex $v\in \P(\Ga)$, let $P_v$ be as in Definition \ref{v-parallel set} and let $\beta_{P_v}$ be as in Definition \ref{def_beta_K}. Suppose $H_v\le H$ is the stabilizer of $\beta_{P_v}$ and let $\beta_{P_v}=\beta_v\times \beta^{\perp}_v$ be as in Lemma \ref{parallel set in blow-up building}. Since $H_v$ preserves this decomposition, there is a factor action $\rho_v:H_v\acts \beta_{v}$. The action $H_v\acts \beta_{P_v}$ is \textit{reducible} if there is a decomposition $H_v=L_v\oplus\Z$ such that $L_v$ acts trivially on the $\beta_{v}$ factor and $\Z$ acts trivially on the $\beta^{\perp}_v$ factor.

\begin{lem}
\label{intersection and normal subgroup for reducible action}
Let $H'_v$ be a finite index torsion free normal subgroup of $H_v$ and let $H_1$ and $H_2$ be two finite index subgroups of $H'_v$. If the induced action $H_i\acts \beta_{P_v}$ is reducible for $i=1,2$, then 
\begin{enumerate}
\item the induced action of $H_1\cap H_2$ is also reducible;
\item the induced action of the largest normal subgroup of $H_v$ contained in $H_1$ is reducible.
\end{enumerate}
\end{lem}

\begin{proof}
Since $H'_v$ is torsion free, the action $H'_v\acts \beta_{P_v}$ is faithful. Then (1) follows readily from the definition. Note that for any $g\in H_v$, $g H_1g^{-1}\subset H'_v$. Moreover, $gH_1g^{-1}$ is reducible. Thus (2) follows from (1).
\end{proof}

We caution the reader that Lemma \ref{intersection and normal subgroup for reducible action} is not true if we drop the first sentence of the lemma. For example, one can take $H_v=\Z\oplus\Z \oplus\Z/2\Z$ acting on $\Bbb R^2$, where then first and second $\Z$ factor act by translation along the $x$-axis and $y$-axis respectively, and the $\Z/2\Z$ acts trivially. Take $H_1=\langle 1,0,0\rangle\oplus \langle 0,1,0\rangle$ and $H_2=\langle 1,0,1\rangle\oplus \langle 0,1,1\rangle$. It is easy to check that $H_1$ and $H_2$ are reducible, but $H_1\cap H_2$ is not reducible.

Pick vertex $v\in\Ga$, it follows from Lemma \ref{Ga'-component} (2) and Lemma \ref{characterization of standard components} that each $St(v)$-component in $Y(\Ga)$ is of form $\beta_{P_w}=\beta_w\times \beta^{\perp}_w$ for some vertex $w\in\P(\Ga)$, hence each $St(v)$-component $L\subset K(\Ga)$ is of form $\beta_{P_w}/H_w$ by Lemma \ref{inverse image of Ga'-component}. If $K(\Ga)$ is special (recall that we require $H$ to be torsion free in our definition of specialness), then Lemma \ref{orientation and specialness} implies that $L$ is a fibre bundle over a branched circle, where the fibres come from the $\beta^{\perp}_w$ factor in $\beta_{P_w}$. The \textit{$v$-holonomy} of $L$ is defined to be the holonomy group of the connection on this fibre bundle induced by the cube complex structure on $L$. Note that $L$ has trivial $v$-holonomy if and only if $L$ is isomorphic (as cube complexes) to the product of a $v$-component in $L$ and $lk(v)$-component in $L$, such splitting is called a \textit{$v$-splitting}. Since the holonomy group is finite, $L$ has a finite sheet cover with trivial $v$-holonomy. The \textit{$v$-holonomy} of $K(\Ga)$ is the collection of $v$-holonomy of all of its $St(v)$-components.

\begin{definition}
$K(\Ga)$ \textit{has trivial holonomy} if $K(\Ga)$ is special and it has trivial $v$-holonomy for each vertex $v\in\Ga$. Similarly, if $S'(\Ga)$ is a finite cover of the Salvetti complex $S(\Ga)$, we define the $v$-holonomy of its $St(v)$-component in a similar way. Again, if all such holonomy vanishes, then $S'(\Ga)$ \textit{has trivial holonomy}.
\end{definition}

\begin{remark}
We caution the reader that the notion of trivial holonomy defined here is different from Haglund's notion \cite[Defition 5.6]{haglund2006commensurability} in the studying of right-angled building. In particular, our notion is not stable when passing to subgroups while Haglund's notion is.
\end{remark}
Suppose $K(\Ga)$ is special. Then $K(\Ga)$ has trivial holonomy if and only if $H_w\acts \beta_{P_w}$ is reducible for each vertex $w\in\P(\Ga)$. This together with Lemma \ref{intersection and normal subgroup for reducible action} imply the following result.

\begin{lem}
\label{normal subgroup and trivial holonomy}
Let $H'$ be a finite index torsion free normal subgroup of $H$. Suppose there exist finite index subgroups $H_1,H_2\le H'$ such that $Y(\Ga)/H_i$ has trivial holonomy for $i=1,2$. Then $Y(\Ga)/(H_1\cap H_2)$ has trivial holonomy. In particular, let $N\vartriangleleft H$ be the largest normal subgroup contained in $H_1$, then $Y(\Ga)/N$ has trivial holonomy. 
\end{lem}

\begin{lem}
\label{product decomposition of actions in building}
Let $L$ be a $\Ga'$-component in $|\B|$. Suppose $\Ga'$ admit a join decomposition $\Ga'=\Ga_1\circ \Ga_2$ where $\Ga_1$ is a clique with its vertices denoted by $\{v_i\}_{i=1}^{n}$. Let $L=\prod_{i=1}^{n}L_i\times L_{n+1}$ be the product decomposition such that $L_{i}$ corresponds to $v_i$ for $1\le i\le n$ and $L_{n+1}$ corresponds to $\Ga_2$. Let $H_{L}\le H$ be the stabilizer of $L$. If $K(\Ga)$ has trivial holonomy, then $H_L=\oplus_{i=1}^{n+1}H_i$ such that $H_i$ acts trivially on $L_j$ if $i\neq j$, and $H_i\cong \Z$ for $1\le i\le n$.
\end{lem}

\begin{proof}
First we look at the case $\Ga_1=\{v\}$ and $\Ga_2=lk(v)$. Then $L$ is a standard component by Lemma \ref{Ga'-component} (2) and $\pi^{-1}(L)=\beta_{P_w}$ for some vertex $w\in\P(\Ga)$ by Lemma \ref{correspondence of standard components} and Lemma \ref{characterization of standard components} ($\pi$ is the canonical projection). Moreover, the stabilizer is exactly $H_{L}$. Since $K(\Ga)$ has trivial holonomy, the action $H_{L}\acts \beta_{P_w}$ is reducible. Since $\pi$ respects the product decomposition of $\beta_{P_w}$, the action $H_{L}\acts L$ has the required decomposition. Now we look at the case $\Ga_1=\{v\}$ and $\Ga_2$ is any induced subgraph of $lk(v)$. This follows from the previous case by consider the $St(v)$-component that contains $L$. In general, for each vertex $v_i\in\Ga_1$, $\Ga'$ can be written as a join of $\{v_i\}$ and a induced subgraph of $lk(v_i)$. Thus we can induct on the number of vertices in $\Ga_1$ and apply case 2 to reduce the number of vertices.
\end{proof}

\begin{lem}
\label{trivial holonomy and special}
If $K(\Ga)$ has trivial holonomy, then $|\B|/H$ is a compact special cube complex. 
\end{lem}

\begin{proof}
First we show $|\B|/H$ is a non-positively curved cube complex. Let $x\in |\B|$ be a vertex of rank $n$ and let $p:|\B|\to |\B|/H$ be the projection map. It suffices to show the link of $p(x)$ in $|\B|/H$ is flag.

Suppose $x$ is labelled by $\Delta_x$ with its vertex set denoted by $\{v_i\}_{i=1}^{n}$. Let $\Delta^{\perp}_x\subset\Ga$ be the induced subgraph spanned by vertices in $\Ga$ which are adjacent to each vertex in $\Delta_x$, and let $L$ be the $\Delta_{x}\circ\Delta^{\perp}_x$-component containing $x$. For each $1\le i\le n$, let $L_i$ be the $v_i$-component that contains $x$. Then each $L_i$ is a infinite tree of diameter 2. Let $L_{n+1}$ be the $\Delta^{\perp}_x$-component that contains $x$. Then $L$ admits a product decomposition $L=\prod_{i=1}^{n}L_i\times L_{n+1}$. By the construction of $|\B|$, each edge in $|\B|$ which contains $x$ is labelled by a vertex in $\Delta_{x}\circ\Delta^{\perp}_x$, thus a small neighbourhood of $x$ in $|\B|$ is contained in $L$. Let $H_L\le H$ be the stabilizer of $L$. By Lemma \ref{inverse image of Ga'-component}, it suffices to show the link of $p(x)$ in $L/H_L$ is flag.

By Lemma \ref{product decomposition of actions in building}, $H_L=\oplus_{i=1}^{n+1}H_i$ such that $H_i$ acts trivially on $L_j$ if $i\neq j$, and $H_i\cong \Z$ for $1\le i\le n$. Then $H/H_L$ and $\prod_{i=1}^{n+1} L_i/H_i$ are isomorphic as cube complexes. We represent $x$ as $(x_1,x_2,\cdots,x_{n+1})$ with respect to the product decomposition of $L$. Note that different edges in $H_{n+1}$ that contains $x_{n+1}$ have different labels (actually they are in 1-1 correspondence with vertices in $\Delta^{\perp}_x$). Since the action $H_{n+1}\acts L_{n+1}$ is label-preserving, any element in $H_{n+1}$ which fixes $x_{n+1}$ also fixes each edge emanating from $x$, hence fixes a small neighbourhood around $x_{n+1}$. Thus the link of $p(x_{n+1})$ in $L_{n+1}/H_{n+1}$ is isomorphic to the link of $x_{n+1}$ in $L_{n+1}$, hence is flag. On the other hand, all $L_i/H_i$'s are trees for $1\le i\le n$. Thus the link of $p(x)$ in $L/H_L$ is flag.

Now we show $|\B|/H$ is special. We can orient each edge of $|\B|$ from vertex of lower rank to vertex of higher rank. This orientation is compatible with parallelism. Since $H$ preserves the rank of vertices, it preserves this orientation, hence hyperplanes in $|\B|/H$ are two-sided. They do not self-interest since the action $H\acts|\B|$ preserve the labelling of hyperplanes. Inter-osculation is ruled out by Lemma \ref{Ga'-component} (3).

It remains to show hyperplanes in $|\B|/H$ do not self-osculate. Pick a hyperplane $h\subset |\B|/H$ and let $v$ be the label of an edge dual to $h$ (all edges dual to $h$ have the same label). Then Lemma \ref{Ga'-component} (3) implies there is a $St(v)$-component $L'\subset |\B|/H$ such that $L'$ contains each cube which intersects $h$. It suffices to show $h$ does not self-osculate in $L'$. Lemma \ref{product decomposition of actions in building} implies $L'$ admits a product decomposition $L'\cong L'_1\times L'_2$ which is induced by the join decomposition $St(v)=\{v\}\circ lk(v)$. Note that $L'_1$ is a finite tree of diameter 2 and $h$ is dual to some edge in the $L'_1$ factor, thus $h$ does not self-osculate.  
\end{proof}

\begin{definition}
\label{reduction}
For each special $K(\Ga)$ we associate a cube complex, which is called the \textit{reduction} of $K(\Ga)$ and is denoted by $S_K(\Ga)$, as follows. Let $f:X(\Ga)^{(0)}\to Y(\Ga)$ be the bijection between $X(\Ga)^{(0)}$ and rank 0 vertices of $Y(\Ga)$ in Definition \ref{construction}. Recall that we have identified $X(\Ga)^{(0)}$ with $G(\Ga)$, thus $f$ induces an action $H\acts G(\Ga)$ by flat-preserving bijections. Lemma \ref{orientation and specialness} and Theorem \ref{conjugate to left translation} imply that up to pre-composing $f$ with a flat-preserving bijection, we can assume $H\acts G(\Ga)$ is an action by left translations. Using the identification $X(\Ga)^{(0)}\cong G(\Ga)$ again, we have an isometric action $H\acts X(\Ga)$. Define $S_K(\Ga)=X(\Ga)/H$. Note that $S_K(\Ga)$ is a finite sheet cover of the Salvetti complex $S(\Ga)$.
\end{definition}

Pick vertex $w\in\P(\Ga)$. Note that $f:P_{w}^{(0)}\cong \Z_{w}\times Q_{v}^{(0)}\to \beta_{P_w}\cong\beta_w\times \beta^{\perp}_{w}$ respects the product decompositions. Thus $H_w\acts \beta_{P_w}$ is reducible if and only if $H_w\acts P_{w}^{(0)}$ is reducible. Hence $K(\Ga)$ has trivial holonomy if and only if its reduction $S_K(\Ga)$ has trivial holonomy. 

Let $\mathcal{C}_K$ be the category whose objects are standard $\Ga'$-components in finite covers of $K(\Ga)$ ($\Ga'$ can be any induced subgraph of $\Ga$), and morphisms are label and rank preserving local isometries. Let $\mathcal{C}_S$ be the category whose objects are $\Ga'$-components in finite covers of the reduction $S_K(\Ga)$, and morphisms are label-preserving local isometries. Note that $f$ is $H$-equivariant, and it maps the $0$-skeleton of a $\Ga'$-component in $X(\Ga)$ to vertices of rank 0 in a standard $\Ga'$-component of $Y(\Ga)$. Thus the following holds.

\begin{lem}
\label{trivial holonomy of reduction}
The map $f$ induces a functor $\Phi_f$ from $\mathcal{C}_S$ to $\mathcal{C}_K$. Moreover, an object in $\mathcal{C}_S$ has trivial holonomy if and only if its image in $\mathcal{C}_K$ has trivial holonomy.
\end{lem}

\subsection{Coverings of complexes with trivial holonomy}
\begin{lem}
\label{trivial holonomy and completion}
Suppose $K(\Ga)$ has trivial holonomy. Let $L\subset K(\Ga)$ be a $\Ga'$-component and let $L'\to L$ be a finite cover with trivial holonomy. Then the canonical complete $\C(L',K(\Ga))$ also has trivial holonomy. Moreover, $L'$ is wall-injective in $\C(L',K(\Ga))$.
\end{lem}

\begin{proof}
Let $\Lambda$ be the intersection graph of $K(\Ga)$. Recall that $\C(L',K(\Ga))$ is the pull-back defined as follows (recall that we require $K(\Ga)$ to be special in our definition of trivial holonomy). 
\begin{center}
$\begin{CD}
\C(L',K(\Ga))                        @>>>       \C(L',S(\Lambda))\\
@VcVV                              @VbVV\\
K(\Ga)                          @>a>>       S(\Lambda)
\end{CD}$
\end{center}
Pick $St(v)$-component $N\subset K(\Ga)$. Since $K(\Ga)$ has trivial holonomy, let $N=N_1\times N_2$ be the $v$-splitting of $N$, where $N_1$ is a $\{v\}$-component in $N$. Note that $a(N)=a(N_1)\times a(N_2)$ and $a|_{N}$ splits as a product of two maps. Let $N'\subset \C(L',S(\Lambda))$ be a lift of $a(N)\subset S(\Lambda)$. We claim $N'$ has a similar cubical splitting as $a(N)$ and $b|_{N'}$ splits as a product of two maps. Then one deduce from this claim and the definition of pull-back that each component of $c^{-1}(N)$ admits a similar cubical splitting, which implies $\C(L',K(\Ga))$ has trivial holonomy.

Now we prove the claim. It follows from the definition of $\C(L',S(\Lambda))$ that $N'$ is of form $\C(N'\cap L',a(N))$, here the map $N'\cap L'\to a(N)$, which we denoted by $i$, is the restriction of the map $L'\to L\stackrel{a}{\to}  S(\Lambda)$ to $N'\cap L'$. By Lemma \ref{product and canonical completion}, it suffices to show $i$ has a splitting which is compatible with $a(N)=a(N_1)\times a(N_2)$. Let $p:L'\to L$ be the covering map. Note that $p(N'\cap L')$ is a connected component in the wall projection $\wpj_{K(\Ga)} (N\to L)$. If no edge of $p(N'\cap L')$ is parallel to some edge in $N_1$, then the image of $i$ is inside an $a(N_2)$-slice of $a(N)$ and $i$ splits trivially. If there is an edge $e\subset p(N'\cap L')$ parallel to an edge in $N_1$, then $e\subset N$. We also deduce that $p(N'\cap L')\subset N$. Thus $N'\cap L'$ is contained in some $St(v,\Ga')$-component of $L'$. Since $L'$ has trivial holonomy, this component admits a $v$-splitting. Since $N'\cap L'$ is locally convex, it inherits a splitting, and $i$ also splits as required.

Since each $St(v)$-component of $K(\Ga)$ has a $v$-splitting, the assumption of Lemma \ref{wall-injective} is satisfied. Thus $L'$ is wall-injective in $\C(L',K(\Ga))$.
\end{proof}

\begin{lem}
\label{finite cover with trivial holonomy}
Suppose $K(\Ga)$ is special. Then $K(\Ga)$ has a finite cover which has trivial holonomy.
\end{lem}

\begin{proof}
By Lemma \ref{trivial holonomy of reduction}, we can assume $K(\Ga)$ is a finite cover of $S(\Ga)$. We induct on the number of vertices in $\Ga$. The case where $\Ga$ has only one vertex is trivial. If there exists vertex $v\in\Ga$ such that $\Ga=St(v)$, then up to passing to a finite sheet cover we can assume $K(\Ga)$ is isomorphic to a product of a circle and a $lk(v)$-component. However, by induction, the $lk(v)$-component has a finite cover with trivial holonomy, thus $K(\Ga)$ has the required cover. Now we assume $St(v)\subsetneq\Ga$ for any vertex $v\in\Ga$.

By induction, there exists a finite sheet cover $A_v$ of the Salvetti complex $S(St(v))$ such that $A_v$ has trivial holonomy and $A_v$ factors through any $St(v)$-components of $K(\Ga)$. By Lemma \ref{normal subgroup and trivial holonomy}, we can assume $A_v$ is a regular cover. Let $K_v$ be the pull back as follows.
\begin{center}
$\begin{CD}
K_v                        @>>>       \C(A_v,S(\Ga))\\
@VVV                              @VVV\\
K(\Ga)                          @>>>       S(\Ga)
\end{CD}$
\end{center}
Then each $St(v)$-component of $K_v$ has trivial $v$-holonomy. 

We claim if we already know $K(\Ga)$ has trivial $v'$-holonomy for a vertex $v'\neq v$, then $K_v$ also has trivial $v'$-holonomy. Then the lemma follows by repeating the above argument for each vertex in $\Ga$. However, this claim follows from $\C(A_v,S(\Ga))$ has trivial holonomy (the proof is identical to Lemma \ref{trivial holonomy and completion}) and Lemma \ref{intersection and normal subgroup for reducible action}.
\end{proof}

\section{Control of retraction image}

\subsection{Wall projections in branched complexes}
\label{subsec_wall projection}
Let $K(\Ga)$, $Y(\Ga)$, $H$ and $|\B|$ be as in Section \ref{sec_branched complexes with trivial holonomy}. In this section we study wall projections in $K(\Ga)$. Assuming hyperbolicity, a powerful tool for this purpose is the following connected intersection theorem proved by Haglund and Wise.
\begin{thm}
\label{haglund-wise connected intersection}
$($\cite[Theorem 4.23]{haglund2012combination}$)$ Let $X$ be a compact special cube complex whose universal cover is Gromov-hyperbolic. Let $(B_0,...,B_n,A)$ be connected locally convex subcomplexes containing the basepoint of $X$. Suppose that $A\subset\cap_{j=0}^{n}B_j$.

Then there is a based finite cover $\bar{X}$ and base elevations $\bar{B}_0,...,\bar{B}_n$ with $\bar{A}\cong A$, such that $\cap_{j\in J}\bar{B}_{j}$ is connected for each $J\subset\{0,...,n\}$.
\end{thm}

However, this result does not apply to $K(\Ga)$ directly since $Y(\Ga)$ is not hyperbolic in general. We first deal with this issue.

For a $CAT(0)$ cube complex $X$, we can endow each cube with the $l^{1}$-metric, which gives rise to a $l^{1}$-metric on $X$. An \textit{$l^{1}$-flat} in $X$ is an isometric embedding $\E^{n}\to X$ with respect to the $l^{1}$-metric such that its image is a subcomplex of $X$. An \textit{$l^{1}$-geodesic} is a 1-dimension $l^{1}$-flat. The following is a version of Eberlein's theorem (see \cite[Theorem II.9.33]{bridson1999metric}) in the cubical setting.

\begin{lem}
\label{Eberlin}
Let $X$ be a proper and cocompact $CAT(0)$ cube complex. Then $X$ is Gromov-hyperbolic if and only if $X$ does not contain a 2 dimensional $l^1$-flat.
\end{lem}

\begin{proof}
It suffices to show the only if direction. Suppose $X$ is not Gromov-hyperbolic, then \cite[Theorem C]{kleiner1999local} implies $X$ contains a 2-flat $F$ (with respect to the $CAT(0)$ metric). Let $L$ be the smallest convex subcomplex of $X$ that contains $F$. Then \cite[Corollary 3.5]{huang2015cocompactly} implies $L$ admits a splitting $L=K_1\times K_2\times\cdots K_{n}\times K$ such that $n\ge 2$ and each $K_i$ contains a geodesic line. Since each $K_i$ is uniformly locally finite, it is not hard to approximate a geodesic line in $K_i$ by an $l^{1}$-geodesic line. Hence $X$ contains a 2-dimension $l^{1}$-flat, which is a contradiction.
\end{proof}

\begin{lem}
\label{hyperbolicity of building quotient}
Suppose $K(\Ga)$ has trivial holonomy. If $\Ga$ does not contain any induced 4-cycle, then the universal cover $X$ of $|\B|/H$ does not contain a 2 dimensional $l^1$-flat, hence hyperbolic.
\end{lem}

\begin{proof}
$X$ has an induced edge-labelling from $|\B|/H$. Let $\ell\subset X$ be a $l^{1}$-geodesic line and let $\Ga_{\ell}\subset\Ga$ be the set of labels of edges of $\ell$. We claim the diameter of $\Ga_{\ell}$ is $\ge 2$. Suppose the contrary is true. Then there exists a clique $\Delta\subset\Ga$ which contains $\Ga_{\ell}$. It follows from Lemma \ref{inverse image of Ga'-component} and Lemma \ref{product decomposition of actions in building} that each $\Delta$-component in $|\B|/H$ is a product of trees of diameter 2, hence contractible. Thus each $\Delta$-component in $X$ is bounded, which is contradictory to that it contains an $l^{1}$-geodesic line $\ell$.

Since $K(\Ga)$ has trivial holonomy, $|\B|/H$ is non-positively curved by Lemma \ref{trivial holonomy and special}. Thus by Lemma \ref{Eberlin}, it suffices to show $X$ can not contain any 2 dimensional $l^{1}$-flat. Suppose the contrary is true and let $\ell_1$ and $\ell_2$ be two $l^{1}$-geodesic lines which spans a 2-flat. Note that Lemma \ref{Ga'-component} (3) is true for $|\B|/H$, hence for $X$. Thus each vertex in $\Ga_{\ell_1}$ is adjacent to every vertex in $\Ga_{\ell_2}$. By the previous claim, we can find a pair of non-adjacent vertices in both $\Ga_{\ell_1}$ and $\Ga_{\ell_2}$, which gives rise to an induced 4-cycle in $\Ga$. 
\end{proof}

\begin{cor}
\label{connected intersection}
Suppose $K(\Ga)$ has trivial holonomy and $\Ga$ does not contain induced 4-cycle. Let $(B_0,...,B_n,A)$ be a collection of standard components containing the basepoint of $K(\Ga)$. Suppose that $A\subset\cap_{j=0}^{n}B_j$.

Then there is a based finite cover $\bar{K}(\Ga)$ and base elevations $\bar{B}_0,...,\bar{B}_n$ with $\bar{A}\cong A$, such that $\cap_{j\in J}\bar{B}_{j}$ is connected for each $J\subset\{0,...,n\}$.
\end{cor}

\begin{proof}
Let $\pi_{H}:K(\Ga)\to |\B|/H$ be the map induced by the canonical projection. Suppose $\pi_{H}(B_i)=D_i$ and $\pi_{H}(A)=C$. Then $(D_0,...,D_n,C)$ are standard components in $|\B|/H$ by Lemma \ref{correspondence between standard components downstairs}, in particular they are locally convex by Lemma \ref{convexity of components}. It follow from Lemma \ref{trivial holonomy and special} and Lemma \ref{hyperbolicity of building quotient} that $|\B|/H$ satisfies the assumption of Theorem \ref{haglund-wise connected intersection}. Thus there is a based finite cover $\bar{L}\to |\B|/H$ and based elevations $\bar{D}_0,...,\bar{D}_n$ with $\bar{C}\cong C$, such that $\cap_{j\in J}\bar{D}_{j}$ is connected for each $J\subset\{0,...,n\}$. Let $\bar{K}(\Ga)$ be the based cover of $K(\Ga)$ corresponding to the subgroup $H'=(\pi_{H})^{-1}_{\ast}(\pi_1(\bar{L},\star))$. This gives rise to the following commuting diagram between based spaces.
\begin{center}
$\begin{CD}
(\bar{K}(\Ga),\star)                      @>>\pi'>       (\bar{L},\star)\\
@VVV                              @VVV\\
(K(\Ga),\star)                        @>>\pi_H>       (|\B|/H,\star)
\end{CD}$
\end{center}
Let $(\bar{B}_0,...,\bar{B}_n,\bar{A})$ be the based lifts in $\bar{K}(\Ga)$. Then they are also standard components, $\pi'(\bar{B}_i)=\bar{D}_i$ and $\pi'(\bar{A})=\bar{C}$. Thus $\bar{A}\cong A$. Recall that $\pi_H$ induces a surjection on the fundamental groups, then $\pi'$ is exactly the $H'$-equivariant quotient of the canonical projection $\pi$ (see Remark \ref{surjection on fundamental groups}). Then the connectedness of $\cap_{j\in J}\bar{B}_{j}$ follows from Lemma \ref{intersection of standard components downastairs} and Lemma \ref{correspondence between standard components downstairs}.
\end{proof}

\begin{cor}
\label{projections are circles}
Let $\Ga'$ be a graph without induced 4-cycles. Pick vertex $v\in\Ga'$ and let $\Ga\subset\Ga'$ be the induced subgraph spanned by vertices in $\Ga'\setminus\{v\}$. Let $\Ga_0=lk(v,\Ga')\subset\Ga$. Suppose $K=K(\Ga)$ is special and pick a base point $p\in K$. Let $i:A\to K$ be a local isometry such that
\begin{enumerate}
\item the image of $A$ is a standard $\Ga_0$-component which contains the base point $p$;
\item the map $i:A\to i(A)$ is a covering map of finite degree.
\end{enumerate}
Then there exists a finite cover $A_{0}\to A$ such that for any further finite cover $\bar{A}\to A_{0}$ with trivial holonomy, there exists a finite directly special cover $\bar{K}\to K$ and an embedding $\bar{A}\to \bar{K}$ such that they fit into the following commutative diagram:
\begin{center}
$\begin{CD}
\bar{A}                           @>>>        \bar{K}\\
@VVV                                   @VVV\\
A          @>>>        K
\end{CD}$
\end{center}
Moreover, the following statements are true.
\begin{enumerate}
\item All elevations of $A\to K$ to $\bar{K}$ are embedded.
\item $\bar{A}$ is wall-injective in $\bar{K}$.
\item Let $\bar{B}\subset\bar{K}$ be any $\Ga_0$-component distinct from $\bar{A}$. Then the wall projection $\wpj_{\bar{K}}(\bar{B}\to\bar{A})$ is a disjoint union of branched tori of various dimension and isolated points.
\end{enumerate}
\end{cor}

In the following proof, for any vertex $v\in \Ga$, we always consider its link $lk(v)$ and its closed star $St(v)$ inside $\Ga$ rather than $\Ga'$.

\begin{proof}
By Lemma \ref{directly special}, we assume $K$ is directly special. Let $K_{1}$ be the canonical completion $\mathsf{C}(A,K)$. Take $K_{2}\to K$ to be the smallest regular cover factoring through each component of $K_{1}$ (more precisely, each component of $K_1$ gives rise to a finite index subgroup of $\pi_1(K,p)$, and $K_2$ is the regular cover corresponding to the intersection of all these subgroups, as well as their conjugates). It is clear that all elevations of $A\to K$ to $K_{2}$ are injective. Let $K_{3}\to K_2$ be a finite cover which has trivial holonomy (Lemma \ref{finite cover with trivial holonomy}). We can assume $K_3\to K$ is regular by Lemma \ref{normal subgroup and trivial holonomy}. Let $A_{0}\to K_{3}$ be an elevation of $A\to K$. We claim $A_{0}$ satisfies our condition.

Let $\C(\bar{A},K_3)$ be the canonical completion with respect to the map $\bar{A}\to A_0\to K_3$ and let $\hat{K}$ be its main component. Then $\hat{K}$ has trivial holonomy and $\bar{A}$ is wall injective in $\hat{K}$ by Lemma~\ref{trivial holonomy and completion}. 

Suppose $\{v_{i}\}_{i=1}^{n}$ are the vertices in $\Ga_0$. For each $i$, let $\Ga_i\subset\Ga$ be the induced subgraph spanned by $St(v_i)\cup\Ga_0$. We claim $\Ga_i\cap \Ga_j=\Ga_0$ if $v_i$ and $v_j$ are not adjacent. If there exists a vertex $u\in (\Ga_i\cap \Ga_j)\setminus\Ga_0$, then $u,v_i,v_j$ and $v$ will form a induced $4$-cycle in $\Ga'$, which yields a contradiction.

By Corollary \ref{connected intersection}, there exists a finite cover $\bar{K}\to\hat{K}$ such that we can find an elevation of $\bar{A}$ in $\bar{K}$ which is isomorphic to $\bar{A}$ (hence we also denote this elevation by $\bar{A}$) such that the standard $\Ga_i$-component $\bar{K}_i\subset\bar{K}$ that contains $\bar{A}$ satisfies $\cap_{i\in I}\bar{K}_i$ is connected for any $I\subset\{1,...,n\}$. Thus by Lemma \ref{intersection of standard components downastairs} and the previous paragraph, $\bar{K}_i\cap\bar{K}_j=\bar{A}$ if $v_i$ and $v_j$ are not adjacent. We claim $\bar{A}$ and $\bar{K}$ satisfy the requirements of the corollary.

Since the covering map $\bar{K}\to K$ factors through $K_2$, every elevation of $A$ to $\bar{K}$ is embedded. Since $\bar{K}$ covers $K$, it is directly special. Since $\bar{A}$ is wall-injective in $\hat{K}$ (Lemma \ref{wall-injective}), it is wall-injective in $\bar{K}$. Let $\bar{B}\neq\bar{A}$ be another $\Ga_0$-component. Pick two distinct edges $e_i$ and $e_j$ in $\wpj_{\bar{K}}(\bar{B}\to\bar{A})$ and suppose they are labelled by $v_i$ and $v_j$ in $\Ga_0$ respectively. Then $\bar{B}\subset\bar{K}_i$ and $\bar{B}\subset\bar{K}_j$. Thus $v_i=v_j$ or they are adjacent, otherwise we would have $\bar{K}_i\cap\bar{K}_j=\bar{A}$, which contradicts $\bar{B}\neq\bar{A}$. Thus the labels of all edges in $\wpj_{\bar{K}}(\bar{B}\to\bar{A})$ are contained in a clique of $\Ga$. On the other hand, by the definition of wall projection, if each edge in a corner of some cube is contained $\wpj_{\bar{K}}(\bar{B}\to\bar{A})$, then this cube and the smallest branched torus containing this cube is contained in $\wpj_{\bar{K}}(\bar{B}\to\bar{A})$. Thus (3) follows.
\end{proof}

\subsection{The modified completions and retractions}
\label{subsec_modified completeions and retractions}
Let $K=K(\Ga)=Y(\Ga)/H$ be a branched complex which is directly special. Recall that parallel edges have the same label, thus there is a well-defined labelling for hyperplanes of $K$. Note that each edge $e\subset K$ is contained in a $v$-component ($v\in\Ga$ is the label of $e$), which is a branched circle. If $e$ is contained in the core of this branched circle, then $e$ is called a \textit{core} edge. Since a core edge is parallel to another core edge, it makes sense to talk about \textit{core hyperplanes} in $K$. We orient each edge in $K$ such that
\begin{enumerate}
\item the orientation respects parallelism between edges;
\item for any circle in $K$ made of core edges, the orientation of each edge in the circle fits together to give an orientation of the circle.
\end{enumerate}
Condition (2) is possible by Lemma~\ref{orientation and specialness}. Namely for each standard branched line in $Y(\Ga)$, we can choose an orientation for its core. And we can require this choice is compatible with parallelism and is $H$-equivariant, thus it descends to orientation of corresponding circles in $K$.

Two hyperplanes of $K$ are \textit{equivalent} if and only if they are both core hyperplanes and they are dual to the same $v$-component for some vertex $v\in\Ga$. We claim this is indeed an equivalent relationship. Suppose for $i=1,2$, $h$ and $h_i$ are core hyperplanes which are dual to the same $v_i$-component. Then $v_1=v_2$, and $h$ and $h_i$ are in the same $St(v_i)$-component. Thus $h,h_1$ and $h_2$ are in the same $St(v_1)$-component. Recall that each $St(v_1)$-component is of form $\beta_{P_{\bar{v}}}/H_{\bar{v}}$ for some vertex $\bar{v}\in\P(\Ga)$ (Lemma \ref{Ga'-component} (2), Lemma \ref{characterization of standard components} and Lemma \ref{inverse image of Ga'-component}). Then $h_1$ and $h_2$ are equivalent by Lemma \ref{parallel set in blow-up building}.

For two equivalent classes of hyperplanes $\mathcal{C}_{1}$ and $\mathcal{C}_{2}$, if an element in $\mathcal{C}_{1}$ crosses an element in $\mathcal{C}_{2}$, then each element in $\mathcal{C}_{1}$ will cross every element in $\mathcal{C}_{2}$. Let $\Gamma_{K}$ be a graph on the equivalence classes of hyperplanes in $K$ such that vertices are adjacent if and only if the corresponding equivalent classes cross. 

For each class of hyperplanes $\mathcal{C}$, we define whether $\mathcal{C}$ \textit{$($directly or indirectly$)$ self-osculates} in the same way as the beginning of Section~\ref{subsec_special cube complex}, except we change the definition of $g_v$ there to be the graph made of all edges which are dual to some hyperplane in $\mathcal{C}$ and contain $v$. Similarly, the notion of \textit{inter-osculation} is also well-defined for two classes of hyperplanes. It follows from our definition of the equivalence relationship, as well as our choice of the edge orientation of $K$, that each equivalent class of hyperplanes does not directly self-osculate, and no two classes inter-osculate (note that Lemma \ref{Ga'-component} (3) is also true for $K(\Ga)$, and this excludes the inter-osculation of two classes). 

There is a natural map $K\to S(\Gamma_{K})$ by sending an oriented edge $\vec{a}$ in $K$ to the oriented edge in $S(\Ga_K)$ that corresponds to the hyperplane dual to $\vec{a}$. The above discussion implies that this map is a local isometry. Suppose $A\subset K$ is an $\Ga'$-component for $\Ga'\subset\Ga$, and let $A\to K\to S(\Gamma_{K})$ be the composition map. Let $\mathsf{C'}(A,K)$ be the pullback which fits into the following diagram:
\begin{center}
$\begin{CD}
\mathsf{C'}(A, K)                         @>>>       \mathsf{C}(A, S(\Gamma_{K}))\\
@VVV                                   @VVV\\
K         @>>>        S(\Gamma_{K})
\end{CD}$
\end{center}
Here $\mathsf{C}(A, S(\Gamma_{K}))$ is the canonical completion defined in Section~\ref{subset_the canonical comletion}. An edge of $S(\Ga_K)$ is a \textit{core edge} if it arises from a family of core hyperplanes. Inverse image of core edges in $S(\Ga_K)$ are defined to be \textit{core edges} in $\mathsf{C}(A, S(\Gamma_{K}))$.

There is a natural copy of $A$ inside $\mathsf{C'}(A, K)$. Let $r:\mathsf{C}(A, S(\Ga_K))\to A$ be the canonical retraction. Note that if $e\subset S(\Ga_K)$ is a core edge, then its inverse image under the map $A\to S(\Ga_K)$ is a disjoint union of circles or isolated points; if $e$ is not a core edge, then its inverse image is a disjoint union of edges or points (since $K(\Ga)$ is directly special). It follows that $r$ is a cubical map. Moreover, 

\begin{lem}
\label{retraction}
Let $e\subset \mathsf{C}(A, S(\Ga_K))$ be a core edge. If $r(e)$ is also an edge, then $r(e)=e$.
\end{lem}

Let $r'$ be the composition $\mathsf{C'}(A, K)\to \mathsf{C}(A,S(\Ga_K))\to A$ which is also a retraction. Note that $\mathsf{C'}(A,K)$ is different from $\mathsf{C}(A,K)$ in general. $\mathsf{C'}(A,K)$ is called the \textit{modified completion} and $r'$ is called the \textit{modified retraction}.

\begin{lem}
\label{modified wpj}
Let $D\subset K$ be a $\Lambda$-component for a induced subgraph $\Lambda\subset\Ga$ and let $\hat{D}$ denote the preimage of $D$ in $\mathsf{C'}(A,K)$. Then $r'(\hat{D})\subset\wpj_{K}(D\to A)$.
\end{lem}

\begin{proof}
By the definition of pull-back, edges of $\mathsf{C'}(A,K)$ are of form $(b_{1},b_{2})$ where $b_{1}\subset K$ and $b_{2}\subset \mathsf{C}(A,S(\Ga_K))$ are sent to the same edge of $S(\Ga_K)$. Moreover, $r'(b_{1},b_{2})=r(b_{2})$ where $r:\mathsf{C}(A,S(\Gamma_{K}))\to A$ is the canonical retraction. Suppose $b_1\subset D$ and $r(b_{2})=b'_{2}$ is an edge.

\textit{Case 1:} $b_2$ is not a core edge. Then $b_1$ is not a core edge. In this case, $b_2$ and $b'_2$ are mapped to the same edge in $S(\Ga_K)$. Hence the same is true for $b'_2$ and $b_1$. Hence they are parallel.

\textit{Case 2:} $b_2$ is a core edge. Then $b_1$ is also a core edge. Moreover $b_{2}=b'_{2}\subset A$ by Lemma \ref{retraction}. Thus the hyperplane dual to $b_{1}$ and the hyperplane dual to $b_{2}$ are in the same equivalent class. For $i=1,2$, let $C_i$ be the circle made of core edges that contains $b_i$. Then $\wpj_{K}(C_1\to C_2)=C_2$. Moreover, $C_2\subset A$ and $C_1\subset D$. Thus $b'_2\subset \wpj_{K}(D\to A)$.
\end{proof}

\begin{remark}
\label{core edge}
It follows from the above proof that if $r'(e)$ is a core edge in $A$, then $e$ has to be a core edge in $\hat{D}$.
\end{remark}
\begin{lem}
\label{length of circle}
Let $C\subset K$ be a circle made of core edges. Then the inverse image of $C$ under $\mathsf{C'}(A,K)\to K$ is a disjoint union of circles whose combinatorial length are equal to the length of $C$.
\end{lem}

\begin{proof}
Suppose $length(C)=l$. Let $e$ be the image of $C$ under $K\to S(\Ga_K)$ ($e$ is an edge). We claim the inverse image of $e$ under $\mathsf{C}(A, S(\Ga_K))\to S(\Ga_K)$ consists of circles of length $1$ or $l$. It suffices to show the inverse image of $e$ under $A\to S(\Ga_K)$ consists of circles of length $l$ or isolated points. Suppose $C'\subset A$ is a circle mapped to $e$. Then hyperplanes dual to edges in $C'$ are equivalent to hyperplanes dual to edges in $C$, hence edges of $C$ and $C'$ are labelled by the same vertex $v\in\Ga$, moreover, $C$ and $C'$ are in the same $St(v)$-component. Thus they have the same length. Now the lemma follows from the claim and the construction of $\mathsf{C'}(A,K)$.
\end{proof}

One can compare the above result with Remark~\ref{larger circle}.

\begin{cor}
\label{circle isomorphism}
Let $C\subset\mathsf{C'}(A,K)$ be a circle made of core edges and let $r':\mathsf{C'}(A,K)\to A$ be the modified retraction. Then either $r'(C)$ is a point, or $r'|_{C}$ is an isomorphism. More precisely, let $C_1\subset K$ be a circle made of core edges.
\begin{enumerate}
\item If $\wpj_{K}(C_1\to A)$ does not contain any edge, then $r'$ sends each component in the inverse image of $C_1$ under $\mathsf{C'}(A,K)\to K$ to a point.
\item If $\wpj_{K}(C_1\to A)$ contains at least one edge, then either $r'$ maps a component as above to a point, or $r'$ maps it isometrically into $A$. And there exists at least one component such that the latter case happens.
\end{enumerate}
\end{cor}

\begin{proof}
Suppose edges in $C_1$ are labelled by $v$. Then (1) follows by applying Lemma \ref{modified wpj} to the $v$-component that contains $C_1$. Under the assumption in (2), there exists a circle made of core edges $C_2\subset A$ such that $\wpj_{K}(C_1\to C_2)=C_2$. The images of $C_1$ and $C_2$ under $K\to S(\Ga_K)$ give rise to the same edge $e\subset S(\Ga_K)$. Thus the pair $(C_1,C_2)$ represents a circle in $\mathsf{C'}(A,K)$ which is mapped isometrically to $C_2$ under the map $r'$. The rest of (2) follows from the proof of Lemma~\ref{length of circle}.
\end{proof}

Let $A$ be as above. Now we want to temporarily forget about the fact that $A$ is inside a larger space $K$. We define two hyperplanes of $A$ are \textit{equivalent} if and only if there exists a vertex $v\in\Ga'$ such that these two hyperplanes are dual to core edges the same $v$-component. Let $\Gamma_{A}$ be a graph on the equivalence classes of hyperplanes in $A$ defined in a similar way to $\Ga_K$. We define $\mathsf{C}'(A,A)$ to be the pullback that fits into the following diagram:
\begin{center}
$\begin{CD}
\mathsf{C'}(A,A)                         @>>>       \mathsf{C}(A,S(\Gamma_{A}))\\
@VVV                                   @VVV\\
A         @>>>        S(\Gamma_{A})
\end{CD}$
\end{center}

\begin{lem}
\label{inverse image}
If the inclusion $A\to K$ induces an injective map of equivalence classes of hyperplanes, then there is a canonical isomorphism $\phi$ from $\mathsf{C'}(A,A)$ to the inverse image of $A$ under the covering map $\mathsf{C'}(A,K)\to K$. Moreover, let $r'_{1},r'_{2}$ be modified retractions, then the following diagrams commutes:
\begin{diagram}
\mathsf{C'}(A,A) &\rTo^{\phi} &\mathsf{C'}(A,K)     &\ \ \ \ \ \ \ \ \   &\mathsf{C'}(A,A) &\rTo^{\phi}           &\mathsf{C'}(A,K)\\
\dTo             &     &\dTo                 &\ \ \ \ \ \ \ \ \   &                 &\rdTo_{r'_{1}} &\dTo_{r'_{2}}\\
A                &\rTo     &K                &\ \ \ \ \ \ \ \ \   &                 &               &A
\end{diagram}
\end{lem}

\begin{proof}
The proof is similar to \cite[Lemma 3.13]{haglund2012combination}. By assumption, there is a natural embedding $S(\Gamma_{A})\to S(\Ga_K)$ (this may not be a local-isometry). This induces an embedding $\mathsf{C}(A,S(\Gamma_{A}))\to \mathsf{C}(A,S(\Ga_K))$ such that the following diagram commutes ($r,r'$ are canonical retractions):
\begin{diagram}
\mathsf{C}(A,S(\Gamma_{A})) &\rTo &\mathsf{C}(A,S(\Ga_K))\\
&\rdTo_{r'} &\dTo_{r}\\
& &A
\end{diagram}
Moreover, $\mathsf{C}(A,S(\Gamma_{A}))$ is the inverse image of $S(\Gamma_{A})$ under $\mathsf{C}(A,S(\Ga_K))\to S(\Ga_K)$. Now the lemma follows from the definition of pullback and modified retraction. 
\end{proof}

\begin{lem}
The assumption of Lemma~\ref{inverse image} is satisfied if $A$ is wall-injective in $X$.
\end{lem}

\begin{proof}
Suppose there are two different classes of core hyperplanes $\mathcal{C}_1$ and $\mathcal{C}_2$ which are mapped to the same class $\mathcal{C}$ of hyperplanes of $X$. Then there is a circle made of core edges $C\subset A$ such that $\mathcal{C}_1$ and $\mathcal{C}$ are made of hyperplanes dual to edges in $C$. Since $A$ is wall-injective, there is 1-1 correspondence between elements in $\mathcal{C}_1$ and elements in $\mathcal{C}$. Thus there exist a hyperplane in $\mathcal{C}_1$ and a hyperplane in $\mathcal{C}_2$ which are mapped to the same hyperplane in $\mathcal{C}$. This yields a contradiction. 
\end{proof}

\begin{remark}
\label{label-preserving isomorphism}
Let $A\to A'$ be a label-preserving isomorphism. Then it induces an isomorphism $S(\Gamma_{A})\to S(\Gamma_{A'})$ which fits into the diagram below on the left. This induces the commuting digram in the middle, which gives rise to an isomorphism $\mathsf{C}'(A, A)\to \mathsf{C}'(A',A')$ that fits into the commuting diagram on the right, whose vertical maps can be both covering maps or both modified retractions.
\begin{diagram}
A                &\rTo   &A'     &\ \ \        &\C(A,S(\Ga_A))        &\rTo   &\C(A',S(\Ga_{A'}))         &\ \ \        &\mathsf{C'}(A,A)     &\rTo          &\mathsf{C'}(A',A')\\
\dTo             &         &\dTo  &\ \ \       &\dTo             &         &\dTo          &\ \ \        &\dTo                 &              &\dTo\\
S(\Ga_A)         &\rTo     &S(\Ga_{A'}) &\ \ \  &S(\Ga_A)         &\rTo     &S(\Ga_{A'})        &\ \ \        &A                    &\rTo          &A'
\end{diagram}
Actually, there is an injective homomorphism from the group of label-preserving automorphisms of $A$ to the group of label-preserving automorphisms of $\C'(A,A)$.
\end{remark}

\section{Construction of the finite cover}
\label{sec_construction of the finite cover}
\subsection{A graph of spaces}
Throughout this section, $\Ga$ will be a finite simplicial graph without induced 4-cycles. And $H$ is a group acting on the blow-up building $Y(\Ga)$ geometrically by label-preserving automorphisms. Our goal is to show $H$ has a finite index torsion free subgroup $H'\le H$ such that $Y(\Ga)/H'$ is a special cube complex. This can be reduced to the following claim.

\begin{claim}
\label{induction claim}
For every vertex $u\in\Ga$, there exists a finite index torsion free subgroup $\bar{H}\le H$ such that for any vertex $\bar{u}\in\P(\Ga)$ labelled by $u$, the factor action (see Section~\ref{subsec_quasi action}) $\rho_{\bar{u}}:\bar{H}_{\bar{u}}\acts\Z_{\bar{u}}$ is conjugate to an action by translations.
\end{claim}

For each vertex of $\Ga$, we find a finite index subgroup of $H$ as in the above claim. Let $H'$ be the intersection of all these subgroups. Then $H'$ satisfies condition (1) of Lemma~\ref{orientation and specialness}, hence $Y(\Ga)/H'$ is a special cube complex.

We will induct on the number of vertices in $\Ga$, and assume the above claim is true for graphs with $\le n-1$ vertices. Let $\Ga$ be a graph of $n$ vertices without induced $4$-cycles. Note that any subgraph of $\Ga$ does not have induced 4-cycles. 

Pick vertex $u\in\Ga$. Let $\Lambda\subset\Ga$ be the induced subgraph spanned by vertices in $\Ga\setminus\{u\}$ and let $\Lambda_u\subset\Lambda$ be the induced subgraph spanned by vertices adjacent to $u$ (it is possible that $\Lambda_u=\emptyset$ or $\Lambda_u=\Lambda$). By Lemma \ref{finite index subgroup without inversion}, we assume $H$ acts on $Y(\Ga)$ without inversions. Moreover, let $\h_u$, $T$ and the action $H\acts T$ be as in Lemma \ref{finite index subgroup without inversion}. Let $K=Y(\Ga)/H$. We label the vertices and edges of $K$ as in Section \ref{subsec_branched complex basics}.

There is a restriction quotient map $q:Y(\Ga)\to T$ (see \cite[Section 2.3]{caprace2011rank}). In other words, $q$ is the cubical map that collapses every edge of $Y(\Ga)$ which is not labelled by $u$ to a point. Note that $q$ is $H$-equivariant. Moreover, $q$ maps each $u$-component isometrically into $T$, whose image is called a \textit{$u$-component} of $T$. The collection of $u$-components in $T$ covers $T$, and these $u$-components are permuted under the $H$-action. We define the \textit{tips} of $T$ to be union of the tips of $u$-components in $T$ (note that each $u$-component is a branched line). The sets of tips of $T$ is $H$-invariant. If $x\in T$ is not a tip, then there is a unique $u$-component that contains $x$. Moreover, $q^{-1}(x)$ is isometric to some $\Lambda_u$-component of $Y(\Ga)$. If $x\in T$ is a tip, then $q^{-1}(x)$ is a standard $\Lambda$-component of $Y(\Ga)$. 

Let $\G=T/H$ be the quotient graph. We define \textit{tips} and \textit{$u$-components} in $\G$ to be the images of tips and $u$-components in $T$. Here is an alternative characterization of $u$-components of $\G$. Pick a vertex $\bar{u}\in\P(\Ga)$ labelled by $u$ and let $\beta_{P_{\bar{u}}}=\beta_{\bar{u}}\times\beta^{\perp}_{\bar{u}}$ be as in Lemma~\ref{parallel set in blow-up building}. We identify $\beta_{\bar{u}}$ with its image in $T$ under $Y(\Ga)\to T$. Note that if $h\in H$ satisfies that $h(\beta_{\bar{u}})\cap\beta_{\bar{u}}$ contains a point which is not a tip of $\beta_{\bar{u}}$, then $h(\beta_{\bar{u}})=\beta_{\bar{u}}$. Thus there is an embedding $\beta_{\bar{u}}/\stab(\beta_{\bar{u}})\to \G$, whose image is a $u$-component. Note that $\stab(\beta_{\bar{u}})=\stab(\beta_{P_{\bar{u}}})$.

Let $E$ be an Eilenberg–MacLane space for $H$. Then the diagonal action $H\acts Y(\Ga)\times E$ is free. Moreover, there is an $H$-equivariant projection $Y(\Ga)\times E\to Y(\Ga)$, which descends to a projection $p:\K=(Y(\Ga)\times E)/H\to K$. Note that for any $x\in K$ and $y\in Y(\Ga)$ which is mapped to $x$ under $Y(\Ga)\to K$, $\pi_1(p^{-1}(x))$ is isomorphic to the stabilizer of $y$. We can view $K$ as a developable complex of groups and $\K$ as the corresponding complex of spaces (see \cite[Chapter III.$\mathcal{C}$]{bridson1999metric}). 

The above map $q:Y(\Ga)\to T$ induces maps $\pi:K\to \G$ and $\bar{\pi}:\K\to \G$. These maps induce graph of spaces decompositions
\begin{equation*}
K=\bigg(\bigsqcup_{v\in \textmd{Vertex}(\G)} K_{v}\sqcup\bigsqcup_{e\in \textmd{Edge}(\G)}(N_{e}\times [0,1])\bigg)	/\sim
\end{equation*} 
and 
\begin{equation*}
\K=\bigg(\bigsqcup_{v\in \textmd{Vertex}(\G)} \K_{v}\sqcup\bigsqcup_{e\in \textmd{Edge}(\G)}(\n_{e}\times [0,1])\bigg)	/\sim.
\end{equation*}
Note that $\K_v=p^{-1}(K_v)$ and $\n_{e}=p^{-1}(N_e)$. There is a 1-1 correspondence between covers of $\K$ and orbifold covers of $K$, and these covers have induced graph of spaces structures. 

Let $\{A(v,i)\}_{i\in I_v}$ be the collection of images of boundary morphisms inside $K_{v}$. Then each $A(v,i)$ is a $\Lambda_u$-component in $K$. This component is standard if and only if $v$ is a tip. Let $N(v,i)$ and $\partial_{v,i}:N(v,i)\to A(v,i)$ be the associated edge space and boundary morphism. It is possible that $N(v,i)\neq N(v,j)$, but $A(v,i)=A(v,j)$. However, this can not happen when $v$ is a tip, since in such case for each vertex in $A(v,i)$, there is only one edge labelled by $u$ which contains this vertex. Note that each $\partial_{v,i}$ preserves edge-labellings. We define $\{\A(v,i)\}_{i\in I_v}$ and $\bar{\partial}_{v,i}:\n(v,i)\to \A(v,i)$ in a similar way. Each $\bar{\partial}_{v,i}$ is a covering map of finite degree. If $v$ is a tip, then $\bar{\partial}_{v,i}$ is a homeomorphism; if $v$ is not a tip, then $A(v,i)=K_v$ and $\A(v,i)=\K_v$.

For subgraph $\Lambda_1\subset\Lambda$, a \textit{$\Lambda_1$-component} $\K_1\subset \K$ is the inverse image of a $\Lambda_1$-component of $K$ under $p:\K\to K$. The \textit{parallel set} of $\K_1$ is the unique $\Lambda_2$-component that contains $\K_1$, where $\Lambda_2=St(\Lambda_1)$ (Definition \ref{definition of links}). Two $\Lambda_1$-components of $\K$ are \textit{parallel} if they are in the same $\Lambda_2$-component of $\K$. We can define parallel sets and parallelism between $\Lambda_1$-components of $K$ in a similar way. 

The image of every $St(u)$-component in $\K$ under the map $\bar{\pi}:\K\to \G$ is a $u$-component in $\G$. Thus $\bar{\pi}$ induces a 1-1 correspondence between $St(u)$-components in $\K$ and $u$-components in $\G$. If a $u$-component in $\G$ contains a circle $C$, then $\bar{\pi}^{-1}(C)$ is a bundle over $C$, whose fibres are homeomorphic to some $\Lambda_u$-component of $\K$. The map $\bar{\pi}$ also induces a 1-1 correspondence between $\Lambda$-components in $\K$ and tips in $\G$.

Since $\Ga$ contains no induced 4-cycle, when $\Lambda_{u}$ is not a clique, every $St(u)$-component in $\K$ is the parallel set of a $\Lambda_u$-component in $\K$. We divide the proof of Claim~\ref{induction claim} into the case when $\Lambda_u$ is a clique, and the case when $\Lambda_u$ is not a clique. Also we make the additional assumption that $\Ga\neq St(u)$, since the $\Ga=St(u)$ case of the claim follows from Lemma~\ref{subgroup with no twist} and Lemma~\ref{conjugate to action by translations}. When $\Ga\neq St(u)$, all $St(u)$-components and $\Lambda$-components are standard (Lemma \ref{Ga'-component} (2)) and they intersects along standard $\Lambda_u$-components (Lemma \ref{intersection of standard components downastairs}).

Recall that by induction hypothesis, if $\Ga'\subset\Ga$ is a induced subgraph that does not contain all vertices of $\Ga$, then each $\Ga'$-component has a finite sheet torsion free special cover. This in particular applies to $\Lambda,\Lambda_u$ and $St(u)$.

To simply notation, we will view $K$ and $\Ga'$-components of $K$ ($\Ga'\subset\Ga$ is a subgraph) as developable complexes of groups, and work with their orbifold covers. This is equivalent to working with $\K$ and $\Ga'$-components of $\K$, and using the usual covering space theory. The fundamental group of a $\Ga'$-component of $K$ is understood to be the orbifold fundamental group of this component, which is isomorphic to the usual fundamental group of the corresponding $\Ga'$-component in $\K$. 

For vertex $v\in\Ga$, a \textit{$v$-edge} in $K$ is an edge labelled by $v$ and a \textit{$v$-circle} in $K$ is a core circle made of $v$-edges. We record the following observation.

\begin{lem}
\label{length of circle0}
Suppose $\{K_i\}_{i=1}^{n}$ are finite covers of $K$ such that there exists a torsion free regular cover $\bar{K}\to K$ such that each $K_i$ factors through $\bar{K}$. Let $K'$ be the smallest regular cover of $K$ which factors through each $K_i$. Pick vertex $w\in\Ga$. Suppose $\ell$ is an integer such that for each $i$ and each $w$-circle in $K_i$, its length divides $\ell$. Then the length of each $w$-circle divides $\ell$.
\end{lem}

\begin{proof}
Let $H_i$ and $\bar{H}$ be the subgroups of $H$ corresponding to $K_i$ and $\bar{K}$. Note that the action $\bar{H}\acts Y(\Ga)$ is free, and $hH_ih^{-1}\le \bar{H}$ for each $i$ and $h\in H$. Thus it suffices to show for each line $L\subset Y(\Ga)$ made of $w$-edges, each $i$ and each $h\in H$, the translation length of the generator of $\stab_{hH_ih^{-1}}(L)$ divides $\ell$. But this follows from the assumption.
\end{proof}

Again, the above lemma may not hold if we do not assume $K_i$ factors through $\bar{K}$. See the example after Lemma \ref{intersection and normal subgroup for reducible action}.

\subsection{Matching finite covers of vertex spaces and edge spaces}
\label{subsec_matching}
The space $A(v,i)$ is called a \textit{gate} if $v$ is a tip. Our strategy is to construct suitable finite sheet covers for each $\Lambda$-component and $St(u)$-component of $K$, and glue them together along the elevations of gates.  

\begin{lem}
\label{collection of covers}
There exists a collection $\mathcal{C}$ of finite sheet regular special covers with trivial holonomy, one for each $\Lambda$-component and $St(u)$-component in $K$, such that for each gate $A(v,i)$, there exists a torsion free regular cover $A_r(v,i)$ such that each possible elevation of $A(v,i)$ to some element of $\mathcal{C}$ factors through $A_r(v,i)$.
\end{lem}

\begin{proof}
For each $\Lambda$-component and $St(u)$-component of $K$, we find a finite sheet special cover which is regular and has trivial holonomy. This is possible by Lemma \ref{finite cover with trivial holonomy} and Lemma \ref{normal subgroup and trivial holonomy}. We denote the resulting collection of covers by $\mathcal{C}'$. For each gate $A(v,i)$, let $A_r(v,i)$ be the smallest regular cover of $A(v,i)$ which factors through each possible elevation of $A(v,i)$ to elements in $\mathcal{C}'$. Pick an element $K'_v\in\mathcal{C}'$ which covers a $\Lambda$-component $K_v\subset K$. Let $K''_{v}$ be the smallest regular cover of $K_v$ such that for each gate $A(v,i)$ in $K_v$, $K''_v$ factors through each component of the canonical completion $\C(A_r(v,i),K'_v)$. We replace $K'_v$ by $K''_v$ and replace other elements of $\mathcal{C}'$ in a similar way to obtain a collection $\mathcal{C}''$. Moreover, we replace each element in $\mathcal{C}''$ by a further cover which is regular and has trivial holonomy, using Lemma \ref{finite cover with trivial holonomy} and Lemma \ref{normal subgroup and trivial holonomy}. The resulting collection has the required properties.
\end{proof}

\subsubsection{Case 1: $\Lambda_u$ is a clique}
\label{one clique}
We first look at the case $\Lambda_u=\emptyset$. It follows from the induction hypothesis that each $\Lambda$-component of $K$ has a finite torsion free cover. Moreover, each $St(u)$-component (or equivalently, $u$-component) of $K$ has a finite torsion free cover whose fundamental group is $\Z$. One can glue together suitable many copies of these covers to form a finite cover of $K$. This cover satisfies the requirement in Claim~\ref{induction claim}. It is torsion free, since each of its vertex spaces and edge spaces is torsion free.

Now we assume $\Lambda_u$ is a non-empty clique. Let $\mathcal{C}$ be the collection of covers as in Lemma \ref{collection of covers} and let $\ell$ be a positive integer such that the length of each core circle (i.e. circle made of core edges) in every element of $\mathcal{C}$ divides $\ell$. It follows from Lemma~\ref{equal torus} below that there exist suitable further finite covers of elements in $\mathcal{C}$ such that we can glue together suitable many copies of them along standard $\Lambda_u$-components to form a finite cover of $K$, which satisfies Claim~\ref{induction claim} (note that two covers of the a $\Lambda_u$-component $L\subset K$ may not be isomorphic as covering spaces of $L$ even when they are the same $\ell$-branched torus, however, this is true if they factors through a common torsion free regular cover of $L$). 

\begin{lem}
\label{equal torus}
Suppose $K(\Ga)$ has a finite regular cover $\bar{K}(\Ga)$ which is special. Let $\ell$ be a positive integer such that for each vertex $u\in\Ga$, the length of each $u$-circle in $K(\Ga)$ divides $\ell$. Then $K(\Ga)$ has a finite index regular cover $K'$ such that for each clique $\Delta\subset\Ga$, each standard $\Delta$-component of $K'$ is an $\ell$-branched torus, i.e. it is isomorphic as cube complexes to a product of branched circles whose core circle have length $=\ell$.
\end{lem}

\begin{proof}
By Lemma \ref{trivial holonomy of reduction}, we can assume $\bar{K}(\Ga)$ is a finite cover of the Salvetti complex $S(\Ga)$ (note that all $u$-circles in $K(\bar{\Ga})$ has length $=\ell$ if and only if the same is true for the reduction of $K(\bar{\Ga})$). For a clique $\Delta\in\Ga$, let $T_{\Delta}\to S(\Delta)$ be the $\ell$-branched torus which covers $S(\Delta)$. Define $\mathring{K}(\Ga)$ to be the pull-back as follows.
\begin{center}
$\begin{CD}
\mathring{K}(\Ga)                        @>>>       \C(T_{\Delta},S(\Ga))\\
@VVV                              @VVV\\
\bar{K}(\Ga)                           @>>>        S(\Ga)
\end{CD}$
\end{center}
Since $T_{\Delta}$ factors through each $\Delta$-component of $\bar{K}(\Ga)$, each $\Delta$-component of $\mathring{K}(\Ga)$ is isomorphic to $T_{\Delta}$. We claim that (1) for each $u\in\Ga$, the length of each $u$-circle in $\mathring{K}(\Ga)$ divides $\ell$; (2) if we already know there is a clique $\Delta'\subset\Ga$ such that each $\Delta'$-component of $\bar{K}(\Ga)$ is isomorphic to $T_{\Delta'}$, then the same is true for $\Delta'$-component of $\mathring{K}(\Ga)$. To see (1), note that each $u$-circle in $S(\Ga)$ has length $=1$, and a $u$-circle in $\C(T_{\Delta},S(\Ga))$ has length equal to either $\ell$ or $1$. (2) follows from (1) and the fact that $\C(T_{\Delta},S(\Ga))$ has trivial holonomy (Lemma \ref{trivial holonomy and completion}).

Now we can repeat the above process for each clique in $\Ga$ to obtain a finite cover $\hat{K}(\Ga)$ of $\bar{K}(\Ga)$ such that for each clique $\Delta\subset\Ga$, each $\Delta$-component in $\hat{K}(\Ga)$ is isomorphic to $T_{\Delta}$. Let $K'$ to be the smallest regular cover of $K(\Ga)$ that factors through $\hat{K}(\Ga)$. Then $K'$ has the required property by Lemma \ref{length of circle0}.
\end{proof}

\subsubsection{Case 2: $\Lambda_u$ is not a clique} Pick a $St(u)$-component $L\subset K$. Then $L$ can be viewed as a graph of spaces over a $u$-component $\G_L\subset\G$ whose boundary morphisms are covering maps of finite degree. A finite sheet cover $L'$ of $L$ is \textit{admissible} if 
\begin{enumerate}
\item there exists a torsion free regular cover $\bar{L}\to L$ such that each component of $L'$ factors through $\bar{L}$;
\item each component of $L'$ is a trivial bundle over a branched circle (we only require the bundle is topologically trivial, however, it may have non-trivial $u$-holonomy).
\end{enumerate}
The existence of such cover follows from Lemma \ref{subgroup with no twist} and the induction hypothesis. It follows from Lemma \ref{orientation and specialness} that to show Claim \ref{induction claim}, it suffices to find a finite torsion free cover of $K$ such that each of its $St(u)$-component is admissible.

\begin{ob}
\label{admissible}
Let $M$ be an admissible cover of $L$ with trivial $u$-holonomy.
\begin{enumerate}
\item The smallest regular cover of $L$ which factors through each component of $M$ also has trivial $u$-holonomy.
\item Each component of $\C'(M,M)$ has trivial $u$-holonomy. 
\item Let $M'\subset M$ be a vertex space or edge space. Then the inverse image of $M'$ under the covering $\C'(M,M)\to M$ is naturally isomorphic to $\C(M',M')$.
\end{enumerate}
\end{ob}

Here (1) follows from Lemma \ref{intersection and normal subgroup for reducible action}. (2) follows from the argument in Lemma \ref{trivial holonomy and completion}. Since $M$ has trivial $u$-holonomy, it is a product. Thus $M'$ is wall-injective in $M$ and (3) follows from Lemma \ref{inverse image}.

A collection $\Sigma$ of (not necessarily connected) finite sheet covers, one for each gate, is called \textit{admissible} if there exists a collection $\Phi$ of admissible covers, one for each $St(u)$-component in $K$, such that for each gate, its inverse image in the element of $\Phi$ that covers this gate is a disjoint union of spaces, each of which is isomorphic to the element in $\Sigma$ that covers this gate, in the sense of covering spaces. $\Sigma$ \textit{has trivial $u$-holonomy} if we can choose the collection $\Phi$ such that each of its element has trivial $u$-holonomy. $\Sigma$ is \textit{regular admissible} if each element in $\Sigma$ is connected and each element in $\Phi$ is a regular cover. In generally, being a regular admissible collection is stronger than being an admissible collection of regular coverings.

In the rest of this section, we will use a modified version of the argument in \cite[Section 6]{haglund2012combination}. The reader could also consult \cite[Page 51-52]{wise2012riches} for a pictorial illustration of the strategy in the malnormal case.
\begin{lem}
\label{first step}
For each tip $v\in\G$ and each gate $A(v,i)$ in $K_v$, we can find a finite sheet base pointed cover $K(v,i)\to K_v$ (the base point of $K_v$ is in $A(v,i)$, and we allow the base point to change for different gates in $K_v$) such that the following properties hold.
\begin{enumerate}
\item Each $K(v,i)$ is torsion free and special. 
\item The based elevation $A'(v,i)\to K(v,i)$ of $A(v,i)\to K_v$ is wall-injective. 
\item Each $A'(v,i)$ has trivial holonomy. The collection of all such coverings $A'(v,i)\to A(v,i)$ is a regular admissible collection with trivial $u$-holonomy.
\item For each gate $A(v,i)$, there exists a torsion free regular cover $A_r(v,i)$ such that each elevation of $A(v,i)$ to $K(v,j)$ (it is possible that $j\neq i$) factors through $A_r(v,i)$.
\item $\wpj_{K(v,i)}(B\to A'(v,i))$ is a disjoint union of branched torus and isolated points for any $\Lambda_u$-component $B\subset K(v,i)$ different from $A'(v,i)$. 
\end{enumerate}
\end{lem}

\begin{proof}
Let $\mathcal{C}$ be the collection in Lemma \ref{collection of covers} and let $L'\in \mathcal{C}$ be an element which covers a $St(u)$-component $L\subset K$. Pick a gate $A(v,i)\subset L$. Let $A_1(v,i)\to A(v,i)$ be a finite cover which factors through each elevation of $A(v,i)$ to $L'$ and $K'_v$ ($K'_v$ is the element $\mathcal{C}$ that covers $K_v$). We apply Corollary~\ref{projections are circles} to $A_1(v,i)\to K'_v$ to obtain a finite cover $A_2(v,i)\to A_1(v,i)$ such that any further finite cover of $A_2(v,i)$ with trivial holonomy will satisfy the conclusion of Corollary~\ref{projections are circles}. Let $L''\to L$ be a regular special cover with trivial holonomy such that for each gate $A(v,i)\subset L$, $L''$ factors through each component of the canonical completion $\C(A_2(v,i),L')$. We choose $A'(v,i)$ to be an elevation of $A(v,i)$ to $L''$ (since $L''$ is regular, all elevations are the same). (3) is clear since $L''$ has trivial holonomy. Since $A'(v,i)$ factors through $A_2(v,i)$, by Corollary~\ref{projections are circles}, we can find finite cover $K(v,i)\to K'_v$ which satisfies (1), (2) and (5). (4) follows from Lemma \ref{collection of covers}.
\end{proof}

For each $A'(v,i)$, we form the modified completions $\mathsf{C}'(A'(v,i),A'(v,i))$ and $\mathsf{C}'(A'(v,i),K(v,i))$, and denote them by $\C'_{v,i}(A',A')$ and $\C'_{v,i}(A',K)$ for simplicity. By Lemma \ref{inverse image}, there is a canonical inclusion $\C'_{v,i}(A',A')\to \C'_{v,i}(A',K)$. Moreover, it follows from Observation \ref{admissible} (3) that the collection of all $\C'_{v,i}(A',A')$'s is admissible.

For vertex $a\in \Ga$, we define $\ell_{a}$ to be the least common multiple of the lengths of all $a$-circles in the collection $\{\C'_{v,i}(A',K)\}_{v\in \textmd{Tip}(\G),i\in I_v}$. 

Let $\pi:\C'_{v,i}(A',K)\to K(v,i)$ be the covering map and let $r':\C'_{v,i}(A',K)\to A'(v,i)$ be the modified retraction. Pick a $\Lambda_u$-component $A'\subset\C'_{v,i}(A',K)$, by Corollary~\ref{projections are circles} and Lemma \ref{modified wpj}, if $\pi(A')\neq A'(v,i)$, then $r'(A')$ is contained in a branched torus. Let $T$ be the smallest branched torus containing $r'(A')$ and let $\{v_i\}_{i=1}^{n}$ be the label of edges in $T$. Since $A'(v,i)$ has trivial holonomy, $T$ has a cubical product decomposition $T=\prod_{i=1}^{n}C_i$ where each $C_i$ is a $v_i$-component in $T$. Let $C'_i$ be a cover of $C_i$ of degree $=\ell_{v_i}/n_i$ where $n_i$ is the length of the core of $C_i$, and let $T'=\prod_{i=1}^{n}C'_i$. We define the \textit{shadow} of $A'$ (when $\pi(A')\neq A'(v,i)$) to be one connected component of the pull-back $A'_S$ as follows.
\begin{center}
$\begin{CD}
A'_{S}                        @>>>       T'\\
@VVV                              @VVV\\
A'                            @>r'>>       T
\end{CD}$
\end{center}
It is possible that both $T'$ and $T$ are one point, in which case $A'_S=A'$. The definition of shadow does not depend on the choice of components in $A'_S$, since all of them give the same regular cover of $A'$. We record the following two observations.
\begin{remark}\
\label{circle in shadow}
(1) Let $T_1\to T$ be any cover such that for each $i$, the length of the core of $v_i$-components in $T_1$ divides $\ell_{v_i}$. Then $T'$ factors through $T_1$.

(2) Let $C'$ be a circle made of core edges in $A'$. By Corollary \ref{circle isomorphism}, either $r'(C')$ is a point, in which case all inverse images of $C'$ in $A'_{S}$ are circles which have the same length as $C'$; either $r'|_{C'}$ is an isomorphism, in which case the inverse image of $C'$ in $A'_{S}$ is a circle of length $\ell_a$ where $a$ is the label of edges in $C'$.
\end{remark}

Let $A(v,i)$ be a gate. Let $\dot{A}(v,i)\to A(v,i)$ be the smallest regular cover that factors through each elevation of $A(v,i)$ to $\C'_{v,j}(A',K)$ and the shadow of this elevation (if exists), here $(v,j)$ ranges over all pairs such that $A(v,j)\subset K_{v}$. 

Let $L$ be the $St(u)$-component that contains $A(v,i)$ and let $L'$ be the cover of $L$ induced by the admissible collection of $A'(v,i)$'s. Recall that $L'$ has trivial $u$-holonomy, hence it is isomorphic to a product of $A'(v,i)$ with a $u$-component in $L'$. Then $\dot{A}(v,i)\to A'(v,i)$ induces a cover of $L'$, which we denote by $\dot{L}(v,i)$. Let $\bar{L}$ be the smallest regular cover of $L$ which factors through each $\dot{L}(v',i')$, where $(v',i')$ is a pair such that $A(v',i')$ is a gate in $L$. Then $\bar{L}$ has trivial $u$-holonomy by Lemma \ref{intersection and normal subgroup for reducible action}. Let $\bar{A}(v,i)\to A(v,i)$ be the regular cover induced by $\bar{L}\to L$. Then the collection of all $\bar{A}(v,i)$'s is regular admissible with trivial $u$-holonomy.

It follow from Remark~\ref{circle in shadow} (2) that the length of each $a$-circle in the spaces involved in the construction of $\bar{A}(v,i)$ divides $\ell_a$, thus the lemma below follows from Lemma \ref{length of circle0} and Lemma \ref{first step} (4).

\begin{lem}
\label{length divide}
Let $a\in\Lambda_u$ be a vertex. Then the length of any $a$-circle in $\bar{A}(v,i)$ divides $\ell_a$.
\end{lem}

For each pair $(v,i)$, we define the pull-back $\overline{\C'_{v,i}(A',K)}$ and $\overline{\C'_{v,i}(A',A')}$ as follows.
\begin{center}
$\begin{CD}
\overline{\C'_{v,i}(A',K)}                 @>>>      \bar{A}(v,i)@.\ \ \ \ \ \overline{\C'_{v,i}(A',A') }                 @>>>      \bar{A}(v,i)\\
@VVV                              @VVV                              @VVV                              @VVV\\
\C'_{v,i}(A',K)        @>>>       A'(v,i)@.\ \ \ \ \ \C'_{v,i}(A',A')        @>>>       A'(v,i)
\end{CD}$
\end{center}
Then one deduce from Lemma \ref{inverse image} that there is a naturally defined embedding $i_1:\overline{\C'_{v,i}(A',A')}\to \overline{\C'_{v,i}(A',K)}$ which fits into the following commuting diagram. Moreover, $f^{-1}(i_{2}(\C'_{v,i}(A',A')))=i_1(\overline{\C'_{v,i}(A',A')})$.
\begin{center}
$\begin{CD}
\overline{\C'_{v,i}(A',A')}                       @>i_1>>       \overline{\C'_{v,i}(A',K)}\\
@VVV                              @VVfV\\
\C'_{v,i}(A',A')                           @>i_2>>        \C'_{v,i}(A',K)  
\end{CD}$
\end{center}

We claim the collection of $\overline{\C'_{v,i}(A',A')}$'s is admissible. Pick a $St(u)$-component, and let $M_1,M_2$ and $M_3$ be the covers of this component induced by the collection of $\C'_{v,i}(A',A')$'s, $A'(v,i)$'s and $\bar{A}(v,i)$'s respectively. Recall that $M_2$ and $M_3$ have trivial $u$-holonomy, and the same is true for $M_1$ by Observation \ref{admissible} (2). By Observation \ref{admissible} (3), the modified retraction $M_1\to M_2$ is compatible with the retraction $\C'_{v,i}(A',A')\to A'(v,i)$. Then the claim follows by considering the pull-back of the covering $M_3\to M_2$ under the retraction $M_1\to M_2$.

\begin{lem}
\label{factor through}
Let $A(v,i)$ be a gate. Suppose $\bar{A}$ is either (1) an elevation of $A(v,i)$ to $\overline{\C'_{v,i}(A',K)}$ such that the image of $\bar{A}$ under the covering map $\overline{\C'_{v,i}(A',K)}\to K(v,i)$ is distinct from $A'(v,i)$; or (2) an elevation of $A(v,i)$ to $\overline{\C'_{v,j}(A',K)}$ with $i\neq j$. Then $\bar{A}(v,i)$ factors through $\bar{A}$.
\end{lem}

\begin{proof}
We prove case (2). The proof of case (1) is similar. The following diagram indicate the history of $\bar{A}$:
\begin{center}
$\begin{CD}
\bar{A} @>>> A_{2} @>>> A_{1} @>>> A(v,i)\\
@VVV        @VVV         @VVV           @VVV\\
\overline{\C'_{v,j}(A',K)}  @>>> \C'_{v,j}(A',K)  @>>> K(v,j)  @>>> K_{v}
\end{CD}$
\end{center}
Since $i\neq j$, $A_{1}\neq A'(v,j)$. By Theorem \ref{projections are circles} and Lemma \ref{modified wpj}, $r'(A_{2})$ is contained in a branched torus. Let $T$ be the smallest such branched torus. The case $T$ is a point is clear. Otherwise, let $\bar{T}$ be any lift of $T$ in $\bar{A}(v,j)$. It follows from Lemma \ref{length divide} that the length of any $a$-circle in $\bar{T}$ divides $\ell_a$. It follows from Remark \ref{circle in shadow}(1) that the shadow of $A_{2}$ factors through $\bar{A}$, thus $\bar{A}(v,i)$ factors through $\bar{A}$. 
\end{proof}

For each tip $v\in\G$, let $\mathring{K}(v,i)$ be the smallest regular cover of $K_{v}$ factoring through each component of $\overline{\C'_{v,i}(A',K)}$ and let $\mathring{K}_{v}$ be the smallest regular cover of $K_{v}$ factoring through each $\mathring{K}(v,i)$. Let $\mathring{A}(v,i)$ be the smallest regular cover of $A(v,i)$ factoring through each component of $\overline{\C'_{v,i}(A',A')}$.

\begin{lem}\
\label{lift}
\begin{enumerate}
\item The collection $\{\mathring{A}(v,i)\}_{v\in \textmd{Tip}(\G),i\in I_v}$ is regular admissible.
\item Each elevation of $A(v,i)$ to $\mathring{K}_{v}$ is isomorphic to $\mathring{A}(v,i)$ as covering spaces of $A(v,i)$.
\end{enumerate}
\end{lem}

\begin{proof}
Since the collection of $\overline{ \C'_{v,i}(A',A')}$'s is admissible, (1) follows. Let $\mathring{A}$ be a lift of $A(v,i)$ to $\mathring{K}_{v}$. By Lemma \ref{inverse image}, the lift of $A(v,i)$ to $\C'_{v,i}(A',K')$ could be any component of $\C'_{v,i}(A',A')$, hence the lift of $A(v,i)$ to $\overline{\C'_{v,j}(A',K)}$ could be any component of $\overline{\C'_{v,j}(A',A')}$. Thus $\mathring{A}$ factors through $\mathring{A}(v,i)$. On the other hand, Lemma \ref{factor through} implies that $\mathring{A}(v,i)$ factors through each elevation of $A(v,i)$ to $\overline{ \C'_{v,j}(A',K)}$ for each $j$. Thus $\mathring{A}(v,i)$ factors through $\mathring{A}$.
\end{proof}

The collection in Lemma~\ref{lift} (1) induces a finite sheet cover for each $St(u)$-component in $K$. Lemma~\ref{lift} (2) enables us to glue suitable many copies of these covers together with copies of $\mathring{K}_{v}$'s to form a finite sheet cover of $K$, which satisfies Claim~\ref{induction claim}. Thus we have proved the following result.

\begin{cor}
\label{good cover}
Suppose $\Ga$ is a graph without induced $4$-cycle. Let $H\acts Y(\Ga)$ be a group acting geometrically by label-preserving automorphisms on a blow-up building $Y(\Ga)$. Then $H$ has a finite index subgroup $H'\le H$ such that $Y(\Ga)/H'$ is a special cube complex. 
\end{cor}

The next result follows from Definition \ref{an equivariant construction}, Lemma~\ref{label-preserving} and Corollary~$\ref{good cover}$:

\begin{thm}
Suppose $\Ga$ is a graph such that
\begin{enumerate}
\item $\Ga$ is star-rigid;
\item $\Ga$ does not contain induced 4-cycle;
\item $\out(G(\Ga))$ is finite.
\end{enumerate}
Then any finite generated group quasi-isometric to $G(\Ga)$ is commensurable to $G(\Ga)$.
\end{thm}

\section{Uniform lattices in $\aut(X(\Ga))$}
\label{sec_uniform lattice}
In this section we study groups acting geometrically on $X(\Ga)$.

\subsection{Admissible edge labellings of $X(\Ga)$}
Recall that we label circles in the 1-skeleton of the Salvetti complex $S(\Gamma)$ by vertices in $\Gamma$. This induces a $G(\Ga)$-invariant edge-labelling on the universal cover $X(\Gamma)$. Let $x$ be a vertex in $S(\Gamma)$ or $X(\Gamma)$. Then each vertex in the link of $x$ comes from an edge, hence inherits a label. Let $F(\Ga)$ be the flag complex of $\Ga$. Then there is a projection map which preserves the label of vertices.
\begin{equation}
\label{projection based at x}
p_x:Lk(x,X(\Ga))\to F(\Ga)
\end{equation}
See Definition \ref{definition of links} for $Lk(x,X(\Ga))$. Note that for each vertex $v\in F(\Ga)$, $p^{-1}_{x}(v)$ is a pair of vertices. We also pick an orientation for each edge in $S(\Ga)$, and lift them to orientations of edges in $X(\Ga)$ which is $G(\Ga)$-invariant. We fix a $G(\Ga)$-invariant labelling and a $G(\Ga)$-invariant orientation of edges in $X(\Ga)$, and call them the \textit{reference labelling} and \textit{reference orientation}.

Throughout this section, $\Ga$ will be a graph without induced 4-cycle. A vertex $v\in\Ga$ is of \textit{type I} if there exists a vertex $w\in\Ga$ not adjacent to $v$ such that $lk(v)\subset St(w)$. Otherwise $v$ is of \textit{type II}. Two vertices $v_1$ and $v_2$ of type I are \textit{equivalent} if they are adjacent. Now we verify this is indeed an equivalence relationship. Since $\Ga$ has no induced 4-cycle, the link of each vertex of type I has to be a (possibly empty) clique, thus $v_1$ and $v_2$ are adjacent if and only if $St(v_1)=St(v_2)$.

Given vertex $v$ of type I, let $[v]$ be the clique in $\Ga$ spanned by type I vertices which are equivalent to $v$ and let $lk([v])$ be the clique spanned by vertices in $St(v)\setminus [v]$. The definition of $lk([v])$ does not depend on the choice of representative in $v$. Any vertex in $lk([v])$ is of type II. Since the closed star of each vertex in $[v]$ is a clique, a vertex $w\in\Ga\setminus [v]$ is in $lk([v])$ if and only if $w$ is adjacent to one vertex in $[v]$ if and only if $w$ is adjacent to each vertex in $[v]$. 

Let $\Ga_i$ be the induced subgraph of $\Ga$ spanned by vertices of type $i$. Then $\Ga_1$ is a disjoint unique of cliques, one for each equivalent class of vertices of type I. Let $\Delta\subset\Ga_2$ be a clique. Let $\Ga_{1,\Delta}\subset\Ga_1$ be the union of all $[v]$'s such that $lk([v])=\Delta$. We allow $\Delta=\emptyset$. Note that $\Ga_{1,\emptyset}$ is always a (possibly empty) discrete graph. Also it follows from the definition that for cliques $\Delta_1,\Delta_2\subset \Ga_2$, if $\Delta_1\neq\Delta_2$, then $\Ga_{1,\Delta_1}\cap \Ga_{1,\Delta_2}=\emptyset$.

\begin{lem}
\label{standard geodesic line}
Pick $\alpha\in\aut(X(\Ga))$ and pick vertex $v\in\Ga_2$.
\begin{enumerate}
\item For any $v$-component $\ell\subset X(\Ga)$, $\alpha(\ell)$ is a $v'$-component for vertex $v'\in\Ga_2$. 
\item Let $\Delta\subset\Ga_2$ be a clique. Suppose $\alpha$ sends a $\Delta$-component to a $\Delta'$-component for $\Delta'\subset\Ga_2$ (this follows from (1)). Then $\Ga_{1,\Delta}$ is isomorphic to $\Ga_{1,\Delta'}$.
\end{enumerate}
\end{lem}

\begin{proof}
Let $P_{\ell}=\ell\times \ell^{\perp}$ be the $St(v)$-component containing $\ell$ with its natural splitting. Then $\alpha(P_{\ell})=\alpha(\ell)\times\alpha(\ell^{\perp})$. Let $L_1$ and $L_2$ be the collection of labels of edges in $\alpha(\ell)$ and $\alpha(\ell^{\perp})$ respectively. To prove (1), it suffices to show $L_1$ is made of one vertex. Suppose the contrary is true. Let $K$ be the $L_1$-component that contains $\alpha(\ell)$. Since each vertex in $L_1$ is adjacent to every vertex in $L_2$, there exists a copy of $K\times \alpha(\ell^{\perp})$ in $X(\Ga)$ which contains $\alpha(P_{\ell})$. Then $\ell\times \ell^{\perp}\subset \alpha^{-1}(K)\times\ell^{\perp}$. Thus the label of each edge in $\alpha^{-1}(K)$ is adjacent to every vertex in $lk(v)$. Since $\alpha^{-1}(K)$ is not isometric to a line, it has an edge whose label (denoted by $w$) is different from $v$. Thus $lk(v)\subset St(w)$ and $w\notin lk(v)$, which contradicts that $v$ is of type II.

Suppose edges of $\alpha(\ell)$ are labelled by $v'$. Now we show $v'\in\Ga_2$. If the contrary is true, then there is a vertex $w'\in\Ga\setminus St(v')$ such that $lk(v')\subset lk(w')$. Let $K$ be the $\{v',w'\}$-component containing $\alpha(\ell)$. Since $L_2\subset lk(v')$, $\alpha(\ell)\times\alpha(\ell^{\perp})\subset K\times\alpha(\ell^{\perp})\subset X(\Ga)$. Now we can reach a contradiction as in the previous paragraph.

Now we prove (2). The case $\Delta=\Delta'=\emptyset$ is trivial. We assume $\Delta\neq\emptyset$. Let $F$ and $F'=\alpha(F)$ be the $\Delta$-component and $\Delta'$-component as in (2). Pick vertex $x\in F$ and let $x'=\alpha(x)$. Note that $\alpha$ induces an isomorphism $\alpha_{x}:Lk(x,X(\Ga))\to Lk(x',X(\Ga'))$. Let $p_x$ and $p_{x'}$ be the maps defined as in (\ref{projection based at x}). A vertex $v$ in $Lk(x,X(\Ga))$ (or $Lk(x',X(\Ga'))$) is of \textit{type I} if $p_x(v)$ (or $p_{x'}(v)$) is of type I, otherwise $v$ is of \textit{type II}. It follows from (1) that $\alpha_x$ induces a bijection between vertices of type II in $Lk(x,X(\Ga))$ and $Lk(x',X(\Ga))$. Thus if we replace type II by type I in the previous sentence, it still holds. 

Pick a vertex $u\in Lk(x,X(\Ga))$ such that $p_x(u)\in \Ga_{1,\Delta}$. Let $e_u$ be the edge containing $x$ which gives rise to $u$ in the link of $x$. Note that $u'=\alpha_x(u)$ is a vertex of type I. Moreover, since $e_u$ is orthogonal to $F$, $e_{u'}=\alpha(e_u)$ is orthogonal to $F'$. Thus $p_{x'}(u')\in\Ga_{1,\Delta'_1}$ for some clique $\Delta'_1$ which contains $\Delta'$. We claim $\Delta'=\Delta'_1$. If $\Delta'\subsetneq\Delta'_1$, let $F'_1$ be the standard flat such that its support is $\Delta'_1$ and it contains $F'$. Let $F_1=\alpha^{-1}(F'_1)$. Then $F_1$ is a standard flat (this follows from (1) since $\Delta'_1\subset\Ga_2$), and $F\subsetneq F_1$. Since $e_{u'}$ is orthogonal to $F'_1$, $e_u$ is orthogonal to $F_1$. Thus $p_x(u)\in \Ga_{1,\Delta_1}$ for some $\Delta_1$ containing the support of $F_1$. This contradicts $p_x(u)\in \Ga_{1,\Delta}$ since $\Delta_1\neq \Delta$. Thus the claim holds. Let $V$ (or $V'$) be vertices in $Lk(x,X(\Ga))$ (or $Lk(x',X(\Ga))$) whose $p_x$-images (or $p_{x'}$-images) are in $\Ga_{1,\Delta}$ (or $\Ga_{1,\Delta'}$). Then the claims implies $\alpha_x$ induces a bijection between $V$ and $V'$, hence it also induced an isomorphism between the full subcomplexes of $Lk(x,X(\Ga))$ and $Lk(x',X(\Ga))$ spanned by $V$ and $V'$ respectively. However, these two full subcomplexes are isomorphic to the links of the base points in the Salvetti complexes $S(\Ga_{1,\Delta})$ and $S(\Ga_{1,\Delta'})$ respectively. Thus $\Ga_{1,\Delta}$ and $\Ga_{1,\Delta'}$ are isomorphic (however, the isomorphism between them may not be induced by $\alpha_x$).
\end{proof}

\begin{lem}
\label{edge-labelling}
Let $X$ be a $CAT(0)$ cube complex. Suppose it is possible to label the edges of $X$ by vertices of $\Gamma$ such that for each vertex $x\in X$, the link of $x$ in $X$ is isomorphic to the link of the base point in $S(\Gamma)$ in a label-preserving way. Then there is a label-preserving isomorphism $X(\Ga)\cong X$.
\end{lem}

\begin{proof}
Pick a vertex $a\in \Ga$. It follows from the assumption that each $a$-component in $X$ is a line, which is called an $a$-line. Moreover, the following are true:
\begin{enumerate}
\item Each $a$-line is a convex subcomplex of $X$.
\item Two $a$-lines $l_{1}$ and $l_{2}$ are parallel if and only if there exist edges $e_{1}\subset l_{1}$ and $e_{2}\subset l_{2}$ such that $e_{1}$ and $e_{2}$ are parallel.
\end{enumerate}

We orient each edge of $X$ as follows. For each vertex $a\in\Ga$, we group the collection of $a$-lines into parallel classes. In each parallel class, we choose an orientation for one $a$-line, and extent to other $a$-lines in the parallel class by parallelism. It follows from (2) of the previous paragraph that if two edges are parallel, then their orientation is compatible with the parallelism.

Pick base vertices $x\in X(\Gamma)$ and $x'\in X$. By looking at the oriented labelling on the edges of $X$ and $X(\Gamma)$, every word in $G(\Ga)$ corresponds to a unique edge path in $X(\Ga)$ (or $X$) starting at $x$ (or at $x'$), and vice verse. We define a map $f:X(\Gamma)^{(0)}\to X^{(0)}$ as follows. Pick vertex $y\in X(\Ga)$ and an edge path $\omega$ from $x$ to $y$, then there is an edge path $\omega'\subset X$ from $x'$ to $y'$ such that $\omega$ and $\omega'$ correspond to the same word. We define $f(y)=y'$. This definition does not depend on the choice of $\omega$. Actually if we pick two words $w_{1}$ and $w_{2}$ whose corresponding endpoints are $y$, then we can obtain $w_{2}$ from $w_{1}$ by performing the following two basic moves:
\begin{enumerate}
\item $waa^{-1}w'\to ww'$.
\item $wabw'\to wbaw'$ when $a$ and $b$ commutes.
\end{enumerate}
However, these two moves do not affect the position of $y'$. Note that $f$ is surjective and can be extended to a local isometry from $X(\Ga)$ to $X$. Thus $X(\Ga)\cong X$.
\end{proof}

\begin{definition}
A labelling of edges in $X(\Ga)$ is \textit{admissible} if it satisfies the assumption of Lemma~\ref{edge-labelling}.
\end{definition}

\begin{cor}
\label{conjugation and admissible labelling}
Suppose $H$ acts on $X(\Ga)$ by automorphisms such that it preserves an admissible labelling of $X(\Ga)$. Then there exists $g\in\aut(X(\Ga))$ such that $gHg^{-1}$ preserves the reference labelling of $X(\Ga)$.
\end{cor}

\begin{lem}
\label{conjugation and admissible orientation}
Suppose $H$ acts on $X(\Ga)$ geometrically by automorphisms such that it preserves the reference labelling of $X(\Ga)$. If $\Ga$ does not have induced $4$-cycle, then there exist a torsion free finite index subgroup $H'\le H$ and an element $g\in \aut(X(\Ga))$ such that $gH'g^{-1}$ preserves both the reference labelling and the reference orientation. Hence $gH'g^{-1}\le G(\Ga)$.
\end{lem}

\begin{proof}
By Corollary~\ref{good cover}, there is a finite index torsion free subgroup $H'\le H$ be such that $X(\Ga)/H'$ is special. Lemma~\ref{orientation and specialness} implies that we can orient each standard geodesic line in a way which is compatible with parallelism and $H'$-action. The argument in Lemma~\ref{edge-labelling} implies that there exists a label-preserving automorphism $g\in\aut(X(\Ga))$ which maps this orientation to the reference orientation. Thus the lemma follows.
\end{proof}

We recall the following result from \cite{bass1990uniform}.

\begin{thm}
\label{tree lattice}
Let $T$ be a uniform tree and let $H_1,H_2\le\aut(T)$ be two uniform lattices (i.e. they acts geometrically on $T$). Then there exists $g\in \aut(T)$ such that $gH_1g^{-1}\cap H_2$ is of finite index in both $H_2$ and $gH_1g^{-1}$.
\end{thm}

\begin{cor}
\label{existence of admissible labelling on tree}
Suppose $\Ga$ is a disjoint of cliques. Let $H$ be a group acting geometrically on $X(\Ga)$ by automorphisms. Then there exist a finite index torsion free subgroup $H'\le H$ and an admissible labelling of $X(\Ga)$ which is invariant under $H'$.
\end{cor}

Actually in this case $H$ has a finite index subgroup which is isomorphic to a free product of finite generated free Abelian groups \cite[Theorem 2]{behrstock2009commensurability}. However, the corollary does not follow directly from this fact.

\begin{proof}
We form a partition $\Ga=\sqcup_{i=1}^{k}\Ga_i$ such that each $\Ga_i$ is made of cliques of the same size, and cliques in different $\Ga_i$'s have different size. We claim it suffices to prove the corollary for each $\Ga_i$. To see this, suppose $k\ge 2$. Let $t_{k}$ be the tree of diameter 2 with $k$ vertices of valance one. For each $i$, we glue $S(\Ga_i)$ to an endpoint of $t_k$ along the unique vertex in $S(\Ga_i)$ such that different $S(\Ga_i)$'s are glued to different endpoints of $t_k$. Then the resulting space $\bar{S}(\Ga)$ is homotopic equivalent to $S(\Ga)$. We pass to the universal cover $\bar{X}(\Ga)$, and collapse each elevation of $S(\Ga_i)$ ($1\le i\le k$) in $\bar{X}(\Ga)$. The resulting space is a tree, which we denote by $T$. 

Since the action $H\acts X(\Ga)$ maps $\Ga_i$-components to $\Ga_i$-components, there is an induced action $H\acts\bar{X}(\Ga)$, moreover, $H$ permutes the elevations of $S(\Ga_i)$'s. Thus there is an induced action $H\acts T$. By passing to a finite index subgroup, we assume $H$ acts on $T$ without inversion. Then there is a graph of spaces decomposition of the orbifold $\bar{X}(\Ga)/H$ along the graph $T/H$. The fundamental group of each edge orbifold is finite, and the fundamental group of each vertex orbifold is either finite, or isomorphic to a group acting geometrically on $X(\Ga_i)$. If the corollary is true for each $\Ga_i$, then we can argue as in Section \ref{sec_construction of the finite cover} to pass to suitable torsion free finite sheet covers of each edge space and vertex space, and glue them together to form a finite sheet cover of $\bar{X}(\Ga)/H$, which gives the required subgroup of $H$.

Now we look at the case when $\Ga$ is a disjoint union of $p$ copies of $n$-cliques. First we assume $n\ge 2$. By the argument of the previous paragraph (we consider the space obtained by attaching $p$ tori to the valance one vertices of $t_p$, then $H$ acts on the universal cover of this space), we can assume $H$ is torsion free by passing to a finite index subgroup. Moreover, we can assume $K=X(\Ga)/H$ is a union of tori (i.e. $K$ does not Klein bottles, however, we do not require these tori to be embedded). It suffices to show $K$ has a finite sheet cover which covers $S(\Ga)$. 

We pick a copy of $\Bbb S^{3}$ with a free $\Z/p$ action. Pick a $\Z/p$-orbit of $\Bbb S^3$ and we glue a $n$-torus to $\Bbb S^3$ along each point in the orbit. The resulting space has the same fundamental group as $S(\Ga)$, and admits a free $\Z/p$ action. Denote the quotient space by $S_p(\Ga)$. We modify the complex $K$ in a similar way. Namely for each vertex $x\in K$, there are $p$ tori containing $x$. We replace $x$ by a copy of $\Bbb S^3$ and re-glue those tori along a $\Z/p$-orbit in $\Bbb S^3$. The resulting complex $K'$ has the same fundamental group as $K$. Since $K'$ does not contain Klein bottles, there exists a finite sheet covering map $K'\to S_p(\Ga)$, which implies $K$ has a finite sheet cover which covers $S(\Ga)$. When $n=1$, it follows from \cite[Theorem 2]{behrstock2009commensurability} that $H$ has a finite index torsion free subgroup. The rest follows from Theorem \ref{tree lattice}. Alternatively, it is not hard to modify the above argument such that it also works for $n=1$. 
\end{proof}

\subsection{The conjugation theorem}
\begin{thm}
\label{conjugation}
Suppose $\Ga$ is a simplicial graph such that
\begin{enumerate}
\item $\Ga$ is star-rigid;
\item $\Ga$ does not contain induced 4-cycle.
\end{enumerate}
Let $H$ be a group acting geometrically on $(X(\Ga))$ by automorphisms. Then $H$ and $G(\Ga)$ are commensurable. Moreover, let $H_1, H_2\le\aut(X(\Ga))$ be two uniform lattices. Then there exists $g\in\aut(X(\Ga))$ such that $gH_1 g^{-1}\cap H_2$ is of finite index in both $gH_1 g^{-1}$ and $H_2$.
\end{thm}

\begin{proof}
If $\Ga_2=\emptyset$, then $\Ga$ is a discrete set and the theorem follows from Corollary \ref{existence of admissible labelling on tree}. Now we assume $\Ga_2\neq\emptyset$. Pick vertex $v\in\Ga_2$, it follows from Lemma \ref{standard geodesic line} (1) that $\alpha$ sends each standard geodesic in $X(\Ga)$ labelled by a vertex in $\Ga_2$ to another standard geodesic labelled by a (possibly different) vertex in $\Ga_2$. Thus the collection of all $\Ga_2$-components in $X(\Ga)$ is $H$-invariant. It follows that the stabilizer of each $\Ga_2$-component acts cocompactly on itself.

Pick vertex $x\in X(\Ga)$. Then $\alpha$ induces an isomorphism between the links of $x$ and $\alpha(x)$, which gives rise to an automorphism $\alpha_x:\Ga_2\to\Ga_2$ by Lemma \ref{standard geodesic line} (1). We claim that $\alpha_x$ can be extended to an automorphism of $\Ga$. Recall that $\Ga_1$ is a disjoint union of $\Ga_{1,\Delta}$ with $\Delta$ ranging over cliques (including the empty clique) of $\Ga_2$. By Lemma~\ref{standard geodesic line} (2), we can specify an isomorphism $\Ga_{1,\Delta}\to\Ga_{1,\alpha_x(\Delta)}$ for each $\Delta$, and use them to define the extension of $\alpha_x$ as required. 

Pick another vertex $y\in X(\Ga)$ which is adjacent to $x$ along an edge $e$. We claim $\alpha_x=\alpha_y$. Let $v_e\in\Ga$ be the label of $e$. Then $\alpha_x(v)=\alpha_y(v)$ for any $v\in\Ga_2$ which is adjacent to $v_e$. Suppose $v_e\notin\Ga_2$. Then $\alpha_x(lk([v_e]))=\alpha_y(lk([v_e]))$. We can extend $\alpha_x$ and $\alpha_y$ to be automorphisms of $\Ga$ which agree on $St(v_e)$ (recall that $St(v_e)$ is a clique spanned by $[v_e]$ and $lk([v_e])$). Thus $\alpha_x=\alpha_y$ by condition (1) of the theorem. Suppose $v_e\in\Ga_2$ and let $v\in\Ga_1$ be a vertex adjacent to $v_e$. Then $lk([v])\subset \Ga_2\cap St(v_e)$. Thus $\alpha_x$ and $\alpha_y$ agree on $lk([v])$. Again we can extent $\alpha_x$ and $\alpha_y$ to automorphisms of $\Ga$ which agree on $St(v_e)$ and deduce $\alpha_x=\alpha_y$.

The above claim implies $\alpha$ determines a well-defined element in $\aut(\Ga_2)$. Thus we have a homomorphism $H\to \aut(\Ga_2)$. By replacing $H$ by the kernel of this homomorphism, we assume $H$ maps $v$-components to $v$-components for $v\in\Ga_2$. Thus for each clique $\Delta\subset\Ga_2$, $H$ preserves $\Delta$-components. Then it follows from the proof of Lemma \ref{standard geodesic line} (2) that for each edge $e\in X(\Ga)$ labelled by a vertex in $\Ga_{1,\Delta}$, the label of $\alpha(e)$ is also inside $\Ga_{1,\Delta}$. Thus $H$ preserves $\Ga_{1,\Delta}$-components.

Let $\{\Delta_i\}_{i=1}^{k}$ be the collection of cliques in $\Ga_2$ such that $\Ga_{1,\Delta_i}\neq\emptyset$. Let $\Lambda_i$ be the induced subgraph spanned by $\Delta_i$ and $\Ga_{1,\Delta_i}$. Now we define a new cube complex 
\begin{center}
$\bar{S}(\Ga)=(S(\Ga_2)\sqcup (\sqcup_{i=1}^{k}S(\Lambda_i)) \sqcup (\sqcup_{i=1}^{k}C_i))/\sim$.
\end{center}
Here $C_i=S(\Delta_i)\times [0,1]$ (if $\Delta_i=\emptyset$, then $S(\Delta_i)$ is one point), and we glue one end of $C_i$ to $S(\Delta_i)\subset S(\Ga_2)$, and another end to $S(\Delta_i)\subset S(\Lambda_i)$. There is a homotopy equivalence $\bar{S}(\Ga)\to S(\Ga)$ by collapsing the $[0,1]$ factor of each $C_i$, which lifts to a map between their universal covers $r:\bar{X}(\Ga)\to X(\Ga)$. If an edge of $\bar{X}(\Ga)$ is not degenerate under $r$, then it inherits a label from its image, otherwise we leave it unlabelled. $\bar{S}(\Ga)$ is non-positively curved (\cite[Proposition II.11.13]{bridson1999metric}), hence $\bar{X}(\Ga)$ is a $CAT(0)$ cube complex. Since the action of $H$ on $X(\Ga)$ preserves $\Lambda_i$-components and $\Ga_2$-components, there is an action $H\acts \bar{X}(\Ga)$ by automorphisms such that $r$ is $H$-equivariant. 

Let $\h$ be the collection of hyperplanes in $\bar{X}(\Ga)$ which comes from hyperplanes in $C_i$ dual to the $[0,1]$ factor. Note that different elements in $\h$ do not intersect. Let $T$ be the dual tree to the wallspace $(\bar{X}(\Ga),\h)$, i.e. each vertex of $T$ corresponds to a component of $\bar{X}(\Ga)\setminus \h$, and two vertices are adjacent if the corresponding two components are adjacent along an element of $\h$. Then there is an induced action $H\acts T$. Moreover, we have an $H$-equivariant cubical map $q:\bar{X}(\Ga)\to T$ by collapsing edges which are not dual to hyperplanes in $\h$. Pick point $s\in T$. If $s$ is a vertex, then $q^{-1}(s)$ is either a $\Lambda_i$-component or a $\Ga_2$-component. If $s$ is not a vertex, then $q^{-1}(s)$ is isometric to a Euclidean space.

Up to passing to a subgroup of index 2, we assume $H$ acts on $T$ without inversion. Let $\G=T/H$. Then the natural map $\bar{X}(\Ga)/H=K\to \G$ induces a graph of spaces decomposition of the orbifold $K$. Moreover, any cover of $K$ has an induced graph of spaces decomposition. Since the action $H\acts \bar{X}(\Ga)$ maps $\Lambda_i$-components to $\Lambda_i$-components for each $i$, and maps $v$-edges to $v$-edges for each vertex $v\in\Ga_2$, it makes sense to talk about $\Lambda_i$-components and $\Ga_2$-components in $K$.

Let $K_v\subset K$ be a vertex space. A finite sheet cover $K'_v$ of $K_v$ is \textit{good} if the following holds. If $K_v$ is a $\Ga_2$-component, then $K'_v$ is good if it is torsion free and special. Corollary~\ref{good cover} implies such cover exists. Suppose $K_v$ is a $\Lambda_i$-component. Then $K_v=X_v/G_v$, where $X_v$ is a $\Lambda_i$-component in $\bar{X}(\Ga)$ and $G_v\le H$ is the stabilizer of $X_v$. $X_{v}$ has a product decomposition $X_v=X(\Delta_i)\times X(\Ga_{1,\Delta_i})$. Note that $G_v$ maps $w$-edges to $w$-edges for each vertex $w\in\Delta_i$, however, this may not be true if $w\in\Ga_{1,\Delta_i}$. $K'_v$ is good if it corresponds to a subgroup $G'_v\le G_v$ such that (1) $G'_v=\Z^{n}\times L'_v$ where $\Z^{n}$ acts trivially on the $X(\Ga_{1,\Delta_i})$ and $L'_v$ acts trivially on the $X(\Delta_i)$ factor, here $n$ is the number of vertices in $\Delta_i$; (2) there exists a $L'_v$-invariant admissible labelling of $X(\Ga_{1,\Delta_i})$. Since $X(\Delta_i)\cong \E^{n}$, we can apply Lemma~\ref{subgroup with no twist} $n$ times to deduce the existence of cover satisfying (1). Since $\Ga_{1,\Delta_i}$ is a disjoint union of cliques, Corollary~\ref{existence of admissible labelling on tree} implies that there exists a good cover of $K_v$.

Using Lemma~\ref{equal torus} and the argument in Section~\ref{one clique}, we can construct a finite cover $K'\to K$ such that each vertex space of $K'$ is good. Then it is possible to put an admissible labelling on each $\Lambda_i$-component of $\bar{X}(\Ga)$ in a way which is $\pi_1(K')$-invariant. By applying the $H$-equivariant retraction map $r:\bar{X}(\Ga)\to X(\Ga)$, we obtain a $\pi_1(K')$-invariant admissible labelling of $X(\Ga)$. Up to conjugation, we can assume $\pi_1(K')$ preserves the reference labelling (Corollary \ref{conjugation and admissible labelling}). Then the theorem follows from Lemma \ref{conjugation and admissible orientation}.
\end{proof}

\section{Induced 4-cycle and failure of commensurability}
\label{sec_failure of commensurability}
\subsection{Label-preserving action on product of two trees}
\label{subsec_label-preserving action}
Pick two trees $T$ and $T'$, and let $H$ be a torsion free group acting geometrically on $T\times T'$ by automorphisms. $H$ is \textit{reducible} if it has a finite index subgroup which is a product of free groups, or equivalently, $(T\times T')/H$ has a finite sheet cover which is isomorphic to a product of two graphs. Otherwise $H$ is \textit{irreducible}.

Up to passing to a subgroup of index 2, we assume $H$ does not flip the two tree factors. Then there are factors actions $H\acts T$ and $H\acts T'$.

\begin{thm}$($\cite{burger2000lattices}$)$
\label{non-discrete}
Let $h_1:G\to \aut(T)$ and $h_2: G\to \aut(T')$ be the homomorphisms induce by the factor actions. Then the following are equivalent.
\begin{enumerate}
\item The image of $h_1$ is a discrete subgroup of $\aut (T)$.
\item The image of $h_2$ is a discrete subgroup of $\aut (T')$.
\item $G$ is reducible.
\end{enumerate}
\end{thm}

Let $T_n$ be the $n$-valence tree. When $\Ga$ is a 4-gon, $X(\Ga)\cong T_4\times T_4$. Thus we can label edges of $T_4\times T_4$ by vertices of $\Ga$. It is known that there is an irreducible group acting on $T_4\times T_4$ (\cite{janzen2009smallest}), however, such example is not label-preserving. In this section, we will construct an irreducible group acting on $T_4\times T_4$ by label-preserving automorphisms. Given a label-preserving action of $H$ on $T_4\times T_4$, we can pass to an action of $H$ on $T_3\times T_3$ by modifying each $T_4$ as follows.
\begin{center}
\includegraphics[scale=0.50]{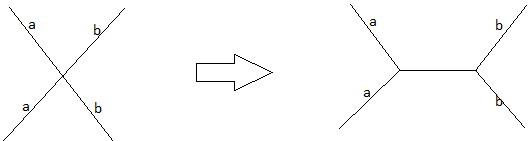}
\end{center}

The other direction is given as follows.
\begin{lem}
\label{T_3}
Suppose there exists an irreducible group $H$ acting geometrically on $T_{3}\times T_3$. Then there exists an irreducible group $H'$ acting geometrically on $T_4\times T_4$ in a label-preserving way.
\end{lem}

\begin{proof}
We modify the first factor $T_3$ to obtain a 4-valence graph $\G$ as follows.
\begin{center}
\includegraphics[scale=0.50]{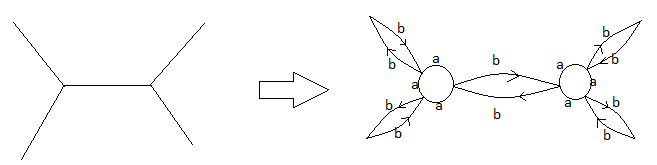}
\end{center}
More precisely, we replace each vertex by a 3-cycle with its edges labelled by $a$, and double each edge to obtain two edges labelled by $b$. We also orient each $b$-edge of $\G$ such that the orientations on two consecutive $b$-edges are consistent. Then each automorphism of $T_3$ induces a unique label-preserving automorphism of $\G$ which respects the orientation of $b$-edges. Thus the factor action $H\acts T_3$ induces a label-preserving action $H\acts \G$. In particular, we obtain an action $H\acts \G\times T_3$. This action is geometric, since there is an $H$-equivariant map $\G\times T_3\to T_3\times T_3$ induced by the natural map $\G\to T_3$. Let $G=\pi_1((\G\times T_3)/H)$. Then $G$ acts geometrically on $T_4\times T_3$, and there is an exact sequence $1\to \pi_1(\G\times T_3)\to G\to H\to 1$. By comparing the action of $H$ and $G$ on the second tree factor, we deduce the irreducibility of $G$ from the irreducibility of $H$ and Theorem \ref{non-discrete}. We can modify the second tree factor in a similar way.
\end{proof}

Now we construct an irreducible group acting on $T_3\times T_3$. First we consider the following modification of the main example in \cite{wise1996non}. Let $X$ be the graph in the top of Figure \ref{the example} below. The top left and top right picture indicate two different ways of labelling and orientating edges of $X$. Let $Y$ be the graph in the bottom. Then there are two different covering maps $f_1:X\to Y$ and $f_2:X\to Y$ induced by the edge labelling and orientation in the top left and top right respectively.

\begin{figure}[h!]
\includegraphics[scale=0.5]{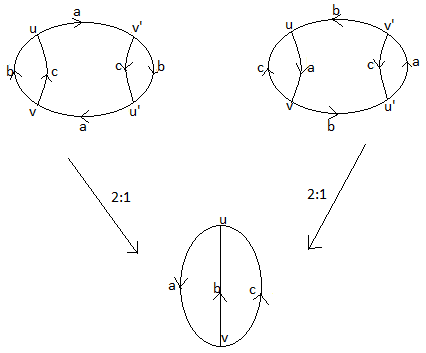}
\caption{}
\label{the example}
\end{figure}

Let $Z'=(X\times[0,1]\sqcup Y\times[0,1])\sim$, where for $i=0,1$, we identify $X\times\{i\}$ with $Y\times\{i\}$ via the covering map $f_{i+1}$ (i.e. $x$ and $f_{i+1}(x)$ are identified). One readily verify that with the natural cube complex structure on $Z'$, the universal cover of $Z'$ is isomorphic to $T_3\times T_3$. Now we collapse the $[0,1]$ factor in $Y\times [0,1]$, which gives a homotopy equivalence $Z'\to Z$. A minor adjustment of the argument in \cite[Chapter II.3.2]{wise1996non} implies $\pi_1(Z)$ is irreducible, hence $\pi_1(Z')$ is also irreducible. For the convenience of the reader, we give a detailed proof of the irreducibility of $\pi_1(Z)$ in the appendix.

\begin{thm}
\label{irreducible group exists}
There exists a torsion free irreducible group acting geometrically on $T_3\times T_3$. Hence there exists a torsion free irreducible group acting geometrically on $T_4\times T_4$ by label-preserving automorphisms. 
\end{thm}

\subsection{The general case}
Let $G$ be a group. Recall that the \textit{profinite topology} on $G$ is the topology generated by cosets of its finite index subgroup. A subset $K\subset G$ is \textit{separable} if it is closed in the profinite topology. Equivalently, $K$ is separable if for any $g\notin K$, there exists a finite index normal subgroup $N\vartriangleleft G$ such that $\phi(g)\notin \phi(K)$ for $\phi:G\to G/N$. 

$G$ is \textit{residually finite} if the identity element is closed under profinite topology. Being residually finite is a commensurability invariant. If $G$ is residually finite, then the solution to any equation on $G$ is separable, since the group multiplication is continuous with respect to the profinite topology. In particular, the centralizer of every element in $G$ is separable. 

\begin{lem}
Let $\Ga$ be a 4-gon and let $H$ be an irreducible group acting geometrically on $X(\Ga)$ by label-preserving automorphisms. Then $H$ is not residually finite.
\end{lem}

\begin{proof}
We argue by contradiction and assume $H$ is residually finite. For each vertex $v\in\P(\Ga)$, let $P_v$ be the $v$-parallel set defined in Definition \ref{v-parallel set} and let $H_v\le H$ be the stabilizer of $P_v$. Since $H$ is label-preserving, $H_v$ acts on $P_v$ cocompactly. Let $P_{v}=\ell_v\times \ell_{v}^{\perp}$ be the product decomposition of $P_{v}$, where $\ell_v$ is a standard geodesic line with $\Delta(\ell_v)=v$. 

Note that there exists a finite index subgroup $H'_v\le H_v$ such that $H'_v=\Z\oplus L_v$ with $\Z$ acting trivially on the $\ell_{v}^{\perp}$ factor and $L_v$ acting trivially on the $\Z$ factor. Let $g$ be a generator of the $\Z$ factor of $H'_v$. Then the centralizer $C_{H}(g)$ of $g$ in $H$ is of finite index in $H_v$. Since $H$ is residually finite, $C_{H}(g)$ is separable in $H$, hence there exists a finite index subgroup $H_1\le H$ such that $H_1\cap H_v=C_{G}(g)$. Note that the factor action $C_G(g)\acts \ell_v$ does not flip the two ends of $\ell_v$. 

Since there are only finitely many $H$-orbits of parallel sets of standard lines, we can repeat the above argument finitely many times to obtain a finite index subgroup $H'\le H$ such that for each vertex $v\in \P(\Ga)$, the factor action $H'_v\acts\ell_v$ does not flip the two ends of $\ell_v$. In this case, we can orient each standard line in $X(\Ga)$ in an $H'$-invariant way such that they are compatible with parallelism. Thus there exists $k\in\aut(X(\Ga))$ such that $kH'k^{-1}$ preserves both the reference labelling and the reference orientation, which implies $H'$ is reducible.
\end{proof}

\begin{thm}
Let $\Ga$ be any finite simplicial graph which contains an induced 4-cycle. Then there exists a group $H$ acting geometrically on $X(\Ga)$ by label-preserving automorphisms such that it is not residually finite. In particular, $H$ is not commensurable to $G(\Ga)$. 
\end{thm}

\begin{proof}
Let $\Ga'\subset\Ga$ be an induced $4$-cycle and pick a $\Ga'$-component $Y\subset X(\Ga)$. The covering map $X(\Ga)\to S(\Ga)$ induces a local isometry $Y\to S(\Ga)$. Let $Z$ be the canonical completion of $Y\to S(\Ga)$. Then the universal cover of $Z$ is isomorphic to $X(\Ga)$. Let $H$ be a torsion free irreducible group acting on $Y$ in a label-preserving way (Theorem \ref{irreducible group exists}). The $H\acts Y$ extends to a label-preserving free action $H\acts Z$ such that the inclusion $Y\to Z$ is $H$-equivariant. Let $G=\pi_1(Z/H)$. Then $G$ acts geometrically on $X(\Ga)$ by label-preserving automorphisms. Moreover, the inclusion $Y\to X(\Ga)$ induces an embedding $H\to G$. Since $H$ is not residually finite, so is $G$. However, each RAAG is residually finite \cite{haglund2008finite}, thus $H$ and $G(\Ga)$ are not commensurable.
\end{proof}

\begin{cor}
Let $\Ga$ be any finite simplicial graph which contains an induced 4-cycle. Then there exists a group $H$ quasi-isometric to $G(\Ga)$ such that it is not commensurable to $G(\Ga)$. 
\end{cor}

\appendix

\section{Irreducibility of $Z$}
We show the fundamental group of $Z$ in Section \ref{subsec_label-preserving action} is irreducible. We will follow \cite[Chapter II.3.2]{wise1996non} very closely.

Let $X,Y,f_1,f_2$ be as in Section \ref{subsec_label-preserving action}. Note that there is a cubical map $r:X\times [0,1]\to Z$. An edge in $X\times [0,1]$ is \textit{vertical} if it is parallel to the $X$ factor, otherwise it is \textit{horizontal}. An edge in $Z$ is called \textit{vertical} or \textit{horizontal} if it is the $r$-image of a vertical or horizontal edge in $X\times [0,1]$. The \textit{vertical} or \textit{horizontal 1-skeleton} of $Z$, denoted by $Z^{(1)}_v$ or $Z^{(1)}_h$, is the union of all vertical edges or horizontal edges in $Z$. We identify the vertical 1-skeleton $Z_v^{(1)}$ with $Y$, hence edges in $Z_v^{(1)}$ inherit orientation and labelling from $Y$. The two vertices of $Z$ are in $Z_v^{(1)}$, and we also denote them by $u$ and $v$.

Recall that $X$ has 4 vertices $\{u,v,u',v'\}$ (see Figure \ref{the example}), which gives 4 horizontal edges in $X\times [0,1]$. The $r$-image of them are 4 horizontal loops in $Z$, which we denote by $e_u,e_v,e_{u'},e_{v'}$ respectively. The horizontal 1-skeleton of $Z$ has two connected components, one of form $e_u\cup e_{u'}$, and the other of form $e_v\cup e_{v'}$.

Then $Z$ has a graph of spaces decomposition, where the underlying graph is a circle with one vertex, the edge space is $X$, the vertex space is $Y$, and the two boundary morphisms are $f_1$ and $f_2$ respectively. 

Note that if $Z$ has a finite cover which is a product of two graphs, then every pair of vertical edge loop $\ell_1$ (i.e. an edge loop made of vertical edges) and horizontal edge loop $\ell_2$ should virtually commute in $\pi_1(Z)$, i.e. $\ell^{n}_1$ and $\ell^{m}_2$ commute for some non-zero integer $n$ and $m$. So it suffices to find a vertical loop and a horizontal loop which do not virtually commute.

We assume $\ell_1$ and $\ell_2$ are locally geodesic. Pick a lift $\tilde{v}$ of $v$ in the universal cover $\tilde{Z}$ of $Z$. For $i=1,2$, we lift the path $\ell^{\infty}_i$ (which is the concatenation of countably infinitely many $\ell_1$'s) to a geodesic ray $\tilde{\ell}_i$ in $\tilde{Z}$ emanating from $\tilde{v}$. Then $\tilde{\ell}_1$ and $\tilde{\ell}_2$ span a quarter-plane. We identify this quarter plane with $[0,\infty)\times [0,\infty)$ such that $\ell_1=\{0\}\times [0,\infty)$. For integers $n,m\ge 0$, we define $\varphi_{n}(\ell^{m}_1)=p(\{n\}\times [0,m])$ and $\varphi_{n}(\ell^{\infty}_1)=p(\{n\}\times [0,\infty))$ where $p:\tilde{Z}\to Z$.

Let $\omega\subset Y$ be an edge path based at $v$. There are two lifts of $\omega$ with respect to $f_1$, based at $v$ and $v'$ respectively. We map them back in $Y$ via $f_2$ and denote the resulting edge paths by $\phi_1(\omega)$ and $\phi_2(\omega)$ respectively. If $\omega=\ell_1$ is a locally geodesic loop and $\ell_2=e_ v$, then for any $n,m\ge 0$,
\begin{equation}
\label{conjugation appendix}
(\phi_1)^{n}(\omega^{m})=\varphi_n(\omega^{m}).
\end{equation}
Each $\omega$ gives rise to a word on $\{a^{\pm},b^{\pm},c^{\pm}\}$ (however, an arbitrary word may not give rise to an edge path based at $v$). Let $\sharp_{a}(\omega)$ be the number of $a$'s in $\omega$ (counted with sign). We define $\sharp_{b}(\omega)$ and $\sharp_{c}(\omega)$ in a similar way.

Let $\iota$ be the automorphism of $Y$ that flips the $a$-edge and $c$-edge, and fixes the $c$-edge. For $i=1,2$, let $G_i$ be the subgroup of $\pi_1(Y,v)$ induced by $f_i:X\to Y$. We record the following observations.

\begin{lem}
\label{calculation}
Let $\sigma$ be an edge loop in $Y$ based at $v$. 
\begin{enumerate}
\item $\sigma\in G_1\Leftrightarrow \sharp_{b}(\sigma)+\sharp_{c}(\sigma)\equiv_2 0$
\item $\phi_2=\iota\circ\phi_1$
\item $\sigma\notin G_1\Rightarrow \phi_1(\sigma^2)=\phi_1(\sigma)\cdot\phi_2(\sigma)$
\end{enumerate}
\end{lem}

\begin{lem}
\label{double}
Let $\sigma$ be an edge loop based at $v$. If $\sigma\notin G_1$, then $\phi_1(\sigma^{2})\notin G_1$.
\end{lem}

\begin{proof}
By Lemma \ref{calculation} (1), it suffices to prove $\sharp_{b}(\phi_1(\sigma^{2}))+\sharp_{c}(\phi_1(\sigma^{2}))\equiv_2 1$.
\begin{align*}
&\sharp_{b}(\phi_1(\sigma^{2}))+\sharp_{c}(\phi_1(\sigma^{2}))\equiv_2\sharp_{a}(\phi_1(\sigma^{2}))=\\
&\sharp_{a}(\phi_1(\sigma)\cdot\phi_2(\sigma))=\sharp_{a}(\phi_1(\sigma))+\sharp_{a}(\iota\circ\phi_1(\sigma))=\\
&\sharp_{a}(\phi_1(\sigma))-\sharp_{c}(\phi_1(\sigma))\equiv_2\sharp_{a}(\phi_1(\sigma))+\sharp_{c}(\phi_1(\sigma))\equiv_2\\
&\sharp_{b}(\sigma)+\sharp_{c}(\sigma)\equiv_2 1.
\end{align*}
The first equality holds since there are even number of edges in $\phi_1(\sigma^{2})$. The second equality follows from Lemma \ref{calculation} (3). The third equality follows from Lemma \ref{calculation} (2). The fourth equality follows from the definition of $\iota$. Note that for any edge $e\subset\sigma$, the label of $e$ belongs to $\{b,c,b^{-1},c^{-1}\}$ if and only if the label of $\phi_1(e)$ in $\phi_1(\sigma)$ belongs to $\{a,c,a^{-1},c^{-1}\}$. Thus the sixth equality holds. The last equality follows from Lemma \ref{calculation} (1).
\end{proof}

In the rest of this section, we always pick $\ell_2=e_v$ in the definition of $\varphi_n$, and let $\sigma$ be the vertical loop based at $v$ of form $ba$ (see Figure \ref{the example}). We claim $\varphi_n(\sigma^{2^{n}})\notin G_1$ for any integer $n\ge 1$. By (\ref{conjugation appendix}), it suffices to prove $\phi^{n}_1((\sigma)^{2^n})\notin G_1$. Since $\sigma\notin G_1$, then case $n=1$ follows from Lemma \ref{double}. Note that when $n>1$, $\phi^{n-1}_1(\sigma^{2^n})=(\phi^{n-1}_1(\sigma^{2^{n-1}}))^{2}$. To see this, we start with $\phi_1(\sigma^{2^n})=(\phi_1(\sigma^{2}))^{2^{n-1}}$ (since $\sigma^2\in G_1$) and iterate. Thus the claim follows by iterating Lemma \ref{double}. The claim is also true if we replace $\sigma$ by $\sigma^{k}$ with $k$ odd, since $\sigma^{k}\notin G_1$ in this case.

It suffices to show there do not exist integers $n,m\neq 0$ such that $\sigma^{m}$ and $e^{n}_{v}$ commute in $\pi_1(Z,v)$. Suppose $\sigma^{m}$ and $e^{n}_{v}$ commutes. We can assume $m,n>0$. Suppose $m=k\cdot 2^{l}$ with $k$ odd. Then $\varphi_{2nl}((\sigma^{k})^{2^{2nl}}) \notin G_1$. Note that $\varphi_{2nl}((\sigma^{k})^{2^{2nl}})=\varphi_{2nl}(\sigma^{2^{2nl-l}\cdot m})$. Since $\sigma^{m}$ and $e^{n}_{v}$ commute, $\varphi_{2nl}(\sigma^{2^{2nl-l}\cdot m})=\sigma^{2^{2nl-l}\cdot m}\in G_1$ (note that $2^{2nl-l}\cdot m$ is even), which is a contradiction.

\bibliographystyle{acm}
\bibliography{1}

\begin{thebibliography}{10}

\bibitem{abramenko2008buildings}
{\sc Abramenko, P., and Brown, K.~S.}
\newblock {\em Buildings: theory and applications}.
\newblock Springer Science \& Business Media, 2008.

\bibitem{agol2013virtual}
{\sc Agol, I., Groves, D., and Manning, J.}
\newblock The virtual haken conjecture.
\newblock {\em Documenta Mathematica 18\/} (2013), 1045--1087.

\bibitem{bass1972degree}
{\sc Bass, H.}
\newblock The degree of polynomial growth of finitely generated nilpotent
  groups.
\newblock {\em Proceedings of the London Mathematical Society 3}, 4 (1972),
  603--614.

\bibitem{bass1990uniform}
{\sc Bass, H., and Kulkarni, R.}
\newblock Uniform tree lattices.
\newblock {\em Journal of the American Mathematical Society\/} (1990),
  843--902.

\bibitem{behrstock2009commensurability}
{\sc Behrstock, J., Januszkiewicz, T., and Neumann, W.}
\newblock Commensurability and {QI} classification of free products of finitely
  generated abelian groups.
\newblock {\em Proceedings of the American Mathematical Society 137}, 3 (2009),
  811--813.

\bibitem{behrstock2012geometry}
{\sc Behrstock, J., Kleiner, B., Minsky, Y., and Mosher, L.}
\newblock Geometry and rigidity of mapping class groups.
\newblock {\em Geometry \& Topology 16}, 2 (2012), 781--888.

\bibitem{MR2727658}
{\sc Behrstock, J.~A., Januszkiewicz, T., and Neumann, W.~D.}
\newblock Quasi-isometric classification of some high dimensional right-angled
  {A}rtin groups.
\newblock {\em Groups Geom. Dyn. 4}, 4 (2010), 681--692.

\bibitem{behrstock2008quasi}
{\sc Behrstock, J.~A., and Neumann, W.~D.}
\newblock Quasi-isometric classification of graph manifold groups.
\newblock {\em Duke Mathematical Journal 141}, 2 (2008), 217--240.

\bibitem{bks2}
{\sc Bestvina, M., Kleiner, B., and Sageev, M.}
\newblock The asymptotic geometry of right-angled {A}rtin groups. {I}.
\newblock {\em Geom. Topol. 12}, 3 (2008), 1653--1699.

\bibitem{bestvina2008quasiflats}
{\sc Bestvina, M., Kleiner, B., and Sageev, M.}
\newblock Quasiflats in {CAT} (0) complexes.
\newblock {\em arXiv preprint arXiv:0804.2619\/} (2008).

\bibitem{bou2015residual}
{\sc Bou-Rabee, K., Hagen, M.~F., and Patel, P.}
\newblock Residual finiteness growths of virtually special groups.
\newblock {\em Mathematische Zeitschrift 279}, 1-2 (2015), 297--310.

\bibitem{bowditch2015large}
{\sc Bowditch, B.}
\newblock Large-scale rank and rigidity of the teichm{\"u}ller metric:
  preprint, 2015.

\bibitem{bowditch2015large1}
{\sc Bowditch, B.}
\newblock Large-scale rank and rigidity of the weil-petersson metric: preprint,
  2015.

\bibitem{bowditch2015large2}
{\sc Bowditch, B.}
\newblock Large-scale rigidity properties of the mapping class groups:
  preprint, 2015.

\bibitem{bridson1999metric}
{\sc Bridson, M.~R., and Haefliger, A.}
\newblock {\em Metric spaces of non-positive curvature}, vol.~319.
\newblock Springer Science \& Business Media, 1999.

\bibitem{MR1744486}
{\sc Bridson, M.~R., and Haefliger, A.}
\newblock {\em Metric spaces of non-positive curvature}, vol.~319 of {\em
  Grundlehren der Mathematischen Wissenschaften [Fundamental Principles of
  Mathematical Sciences]}.
\newblock Springer-Verlag, Berlin, 1999.

\bibitem{burger_mozes}
{\sc Burger, M., and Mozes, S.}
\newblock Lattices in product of trees.
\newblock {\em Inst. Hautes \'Etudes Sci. Publ. Math.}, 92 (2000), 151--194
  (2001).

\bibitem{burger2000lattices}
{\sc Burger, M., and Mozes, S.}
\newblock Lattices in product of trees.
\newblock {\em Publications Math{\'e}matiques de l'IH{\'E}S 92\/} (2000),
  151--194.

\bibitem{caprace2011rank}
{\sc Caprace, P.-E., and Sageev, M.}
\newblock Rank rigidity for {CAT} (0) cube complexes.
\newblock {\em Geometric and functional analysis 21}, 4 (2011), 851--891.

\bibitem{charney2007introduction}
{\sc Charney, R.}
\newblock An introduction to right-angled {A}rtin groups.
\newblock {\em Geometriae Dedicata 125}, 1 (2007), 141--158.

\bibitem{charney2012random}
{\sc Charney, R., and Farber, M.}
\newblock Random groups arising as graph products.
\newblock {\em Algebr. Geom. Topol 12}, 2 (2012), 979--995.

\bibitem{davis1994buildings}
{\sc Davis, M.~W.}
\newblock Buildings are cat (0).
\newblock {\em Geometry and cohomology in group theory (Durham, 1994) 252\/}
  (1994), 108--123.

\bibitem{day2012finiteness}
{\sc Day, M.~B.}
\newblock Finiteness of outer automorphism groups of random right-angled artin
  groups.
\newblock {\em Algebraic \& Geometric Topology 12}, 3 (2012), 1553--1583.

\bibitem{droms1987isomorphisms}
{\sc Droms, C.}
\newblock Isomorphisms of graph groups.
\newblock {\em Proceedings of the American Mathematical Society 100}, 3 (1987),
  407--408.

\bibitem{dunwoody1985accessibility}
{\sc Dunwoody, M.~J.}
\newblock The accessibility of finitely presented groups.
\newblock {\em Inventiones mathematicae 81}, 3 (1985), 449--457.

\bibitem{eskin1998quasi}
{\sc Eskin, A.}
\newblock Quasi-isometric rigidity of nonuniform lattices in higher rank
  symmetric spaces.
\newblock {\em Journal of the American Mathematical Society 11}, 2 (1998),
  321--361.

\bibitem{eskin1997quasi}
{\sc Eskin, A., and Farb, B.}
\newblock Quasi-flats and rigidity in higher rank symmetric spaces.
\newblock {\em Journal of the American Mathematical Society 10}, 3 (1997),
  653--692.

\bibitem{eskin2015rigidity}
{\sc Eskin, A., Masur, H., and Rafi, K.}
\newblock Rigidity of teichm$\backslash$" uller space.
\newblock {\em arXiv preprint arXiv:1506.04774\/} (2015).

\bibitem{frigerio2011rigidity}
{\sc Frigerio, R., Lafont, J.-F., and Sisto, A.}
\newblock Rigidity of high dimensional graph manifolds.
\newblock {\em arXiv preprint arXiv:1107.2019\/} (2011).

\bibitem{gromov1981groups}
{\sc Gromov, M.}
\newblock Groups of polynomial growth and expanding maps (with an appendix by
  jacques tits).
\newblock {\em Publications Math{\'e}matiques de l'IH{\'E}S 53\/} (1981),
  53--78.

\bibitem{gromov1996geometric}
{\sc Gromov, M.}
\newblock Geometric group theory, vol. 2: Asymptotic invariants of infinite
  groups.

\bibitem{hagen2013cocompactly}
{\sc Hagen, M.~F., and Przytycki, P.}
\newblock Cocompactly cubulated graph manifolds.
\newblock {\em Israel Journal of Mathematics\/} (2013), 1--18.

\bibitem{hagen2013cubulating}
{\sc Hagen, M.~F., and Wise, D.~T.}
\newblock Cubulating hyperbolic free-by-cyclic groups: the irreducible case.
\newblock {\em arXiv preprint arXiv:1311.2084\/} (2013).

\bibitem{hagen2014cubulating}
{\sc Hagen, M.~F., and Wise, D.~T.}
\newblock Cubulating hyperbolic free-by-cyclic groups: the general case.
\newblock {\em Geometric and Functional Analysis 25}, 1 (2014), 134--179.

\bibitem{haglund2006commensurability}
{\sc Haglund, F.}
\newblock Commensurability and separability of quasiconvex subgroups.
\newblock {\em Algebraic \& Geometric Topology 6}, 2 (2006), 949--1024.

\bibitem{haglund2008finite}
{\sc Haglund, F.}
\newblock Finite index subgroups of graph products.
\newblock {\em Geometriae Dedicata 135}, 1 (2008), 167--209.

\bibitem{MR2377497}
{\sc Haglund, F., and Wise, D.~T.}
\newblock Special cube complexes.
\newblock {\em Geom. Funct. Anal. 17}, 5 (2008), 1551--1620.

\bibitem{haglund2010coxeter}
{\sc Haglund, F., and Wise, D.~T.}
\newblock Coxeter groups are virtually special.
\newblock {\em Advances in Mathematics 224}, 5 (2010), 1890--1903.

\bibitem{haglund2012combination}
{\sc Haglund, F., and Wise, D.~T.}
\newblock A combination theorem for special cube complexes.
\newblock {\em Annals of Mathematics 176}, 3 (2012), 1427--1482.

\bibitem{hamenstaedt2005geometry}
{\sc Hamenstaedt, U.}
\newblock Geometry of the mapping class groups {III}: Quasi-isometric rigidity.
\newblock {\em arXiv preprint math/0512429\/} (2005).

\bibitem{huang2014quasi}
{\sc Huang, J.}
\newblock Quasi-isometry rigidity of right-angled {A}rtin groups {I}: the
  finite out case.
\newblock {\em arXiv preprint arXiv:1410.8512\/} (2014).

\bibitem{huang_quasiflat}
{\sc Huang, J.}
\newblock Top dimensional quasiflats in {CAT}(0) cube complexes.
\newblock arXiv:1410.8195.

\bibitem{huang2015cocompactly}
{\sc Huang, J., Jankiewicz, K., and Przytycki, P.}
\newblock Cocompactly cubulated 2-dimensional {A}rtin groups.
\newblock {\em arXiv preprint arXiv:1510.08493\/} (2015).

\bibitem{cubulation}
{\sc Huang, J., and Kleiner, B.}
\newblock Groups quasi-isometric to {RAAG}'s.
\newblock {\em arXiv preprint arXiv:1601.00946\/} (2016).

\bibitem{janzen2009smallest}
{\sc Janzen, D., and Wise, D.~T.}
\newblock A smallest irreducible lattice in the product of trees.
\newblock {\em Algebraic \& Geometric Topology 9}, 4 (2009), 2191--2201.

\bibitem{kapovich1997quasi}
{\sc Kapovich, M., and Leeb, B.}
\newblock Quasi-isometries preserve the geometric decomposition of {H}aken
  manifolds.
\newblock {\em Inventiones mathematicae 128}, 2 (1997), 393--416.

\bibitem{karrass1973finite}
{\sc Karrass, A., Pietrowski, A., and Solitar, D.}
\newblock Finite and infinite cyclic extensions of free groups.
\newblock {\em Journal of the Australian Mathematical Society 16}, 04 (1973),
  458--466.

\bibitem{kim2013embedability}
{\sc Kim, S.-h., and Koberda, T.}
\newblock Embedability between right-angled {A}rtin groups.
\newblock {\em Geometry \& Topology 17}, 1 (2013), 493--530.

\bibitem{kleiner1999local}
{\sc Kleiner, B.}
\newblock The local structure of length spaces with curvature bounded above.
\newblock {\em Mathematische Zeitschrift 231}, 3 (1999), 409--456.

\bibitem{kleiner1997rigidity}
{\sc Kleiner, B., and Leeb, B.}
\newblock Rigidity of quasi-isometries for symmetric spaces and {E}uclidean
  buildings.
\newblock {\em Comptes Rendus de l'Acad{\'e}mie des Sciences-Series
  I-Mathematics 324}, 6 (1997), 639--643.

\bibitem{liu2013virtual}
{\sc Liu, Y.}
\newblock Virtual cubulation of nonpositively curved graph manifolds.
\newblock {\em Journal of Topology\/} (2013), jtt010.

\bibitem{ollivier2011cubulating}
{\sc Ollivier, Y., and Wise, D.}
\newblock Cubulating random groups at density less than 1/6.
\newblock {\em Transactions of the American Mathematical Society 363}, 9
  (2011), 4701--4733.

\bibitem{przytycki2012mixed}
{\sc Przytycki, P., and Wise, D.~T.}
\newblock Mixed 3-manifolds are virtually special.
\newblock {\em arXiv preprint arXiv:1205.6742\/} (2012).

\bibitem{przytycki2013graph}
{\sc Przytycki, P., and Wise, D.~T.}
\newblock Graph manifolds with boundary are virtually special.
\newblock {\em Journal of Topology\/} (2013), jtt009.

\bibitem{scott1978subgroups}
{\sc Scott, P.}
\newblock Subgroups of surface groups are almost geometric.
\newblock {\em Journal of the London Mathematical Society 2}, 3 (1978),
  555--565.

\bibitem{stallings1968torsion}
{\sc Stallings, J.~R.}
\newblock On torsion-free groups with infinitely many ends.
\newblock {\em Annals of Mathematics\/} (1968), 312--334.

\bibitem{stallings1983topology}
{\sc Stallings, J.~R.}
\newblock Topology of finite graphs.
\newblock {\em Inventiones mathematicae 71}, 3 (1983), 551--565.

\bibitem{whyte2010coarse}
{\sc Whyte, K.}
\newblock Coarse bundles.
\newblock {\em arXiv preprint arXiv:1006.3347\/} (2010).

\bibitem{wisestructure}
{\sc Wise, D.}
\newblock The structure of groups with a quasiconvex hierarchy. preprint
  (2011).

\bibitem{wise1996non}
{\sc Wise, D.~T.}
\newblock {\em Non-positively curved squared complexes aperiodic tilings and
  non-residually finite groups}.
\newblock Princeton University, 1996.

\bibitem{wise2002residual}
{\sc Wise, D.~T.}
\newblock The residual finiteness of negatively curved polygons of finite
  groups.
\newblock {\em Inventiones mathematicae 149}, 3 (2002), 579--617.

\bibitem{wise2004cubulating}
{\sc Wise, D.~T.}
\newblock Cubulating small cancellation groups.
\newblock {\em Geometric \& Functional Analysis GAFA 14}, 1 (2004), 150--214.

\bibitem{wise2009research}
{\sc Wise, D.~T.}
\newblock Research announcement: the structure of groups with a quasiconvex
  hierarchy.
\newblock {\em Electron. Res. Announc. Math. Sci 16\/} (2009), 44--55.

\bibitem{wise2012riches}
{\sc Wise, D.~T.}
\newblock {\em From riches to raags: 3-manifolds, right-angled Artin groups,
  and cubical geometry}.
\newblock No.~117. American Mathematical Soc., 2012.

\end{thebibliography}

\end{document}